\documentclass[reqno]{amsart}
\usepackage{amsmath, amsfonts, amsthm, amssymb, setspace, textcomp, bbm, multirow, tikz,mathtools,enumitem,booktabs,longtable}
\usepackage{geometry}
\newtheorem{theorem}{Theorem}[section]
\newtheorem{lemma}[theorem]{Lemma}
\newtheorem{corollary}[theorem]{Corollary}
\newtheorem{proposition}[theorem]{Proposition}

\theoremstyle{definition}

\theoremstyle{remark}

\numberwithin{equation}{section}

\newcommand{\aqprod}[3]{\left(#1;#2\right)_{#3}}

\newcommand{\Floor}[1]{\left\lfloor #1 \right\rfloor}
\newcommand{\Fractional}[1]{\left\{#1\right\} }

\newcommand{\tmu}{\widetilde{\mu}}
\newcommand{\wt}[1]{\widetilde{#1}}
\newcommand{\SLTwo}{\mbox{SL}_2(\mathbb{Z})}
\newcommand{\sgn}[1]{{\rm sgn}\! \left(#1\right)}
\newcommand{\IM}[1]{{\rm Im}\! \left(#1\right)}
\newcommand{\RE}[1]{{\rm Re}\! \left(#1\right)}

\newcommand{\boxWithLabel}[4]{
	\draw (#2,#3) -- (#2-#1,#3) -- (#2-#1,#3-#1) -- (#2,#3-#1) -- (#2,#3);
	\node (a) at (#2-0.5*#1,#3-0.5*#1) {#4};
}

\author{CHRIS JENNINGS-SHAFFER}
\address{Mathematical Institute, University of Cologne, Weyertal 86-90, 50931 Cologne, Germany}
\email{cjenning@math.uni-koeln.de}

\author{Dillon Reihill}
\address{Mathematical Institute, University of Cologne, Weyertal 86-90, 50931 Cologne, Germany}
\email{dillon.reihill@ucdconnect.ie}



\title{Asymptotic Formulas related to the $M_2$-rank of Partitions without Repeated Odd Parts}

\allowdisplaybreaks
\begin{document}
%\today

\allowdisplaybreaks

\begin{abstract}
We give asymptotic expansions for the moments of the $M_2$-rank generating function
and for the $M_2$-rank generating function at roots of unity.
For this we apply the Hardy-Ramanujan circle method extended to
mock modular forms. Our formulas for the $M_2$-rank at roots of unity lead
to asymptotics for certain combinations of $N2(r,m,n)$ (the number of 
partitions without repeated odd parts of $n$ with $M_2$-rank congruent
to $r$ modulo $m$). This allows us to deduce inequalities among certain
combinations of $N2(r,m,n)$. In particular, we resolve a few conjectured 
inequalities of Mao.
\end{abstract}

\maketitle

\allowdisplaybreaks

%Last modified \today.

\section{Introduction}

In this article we study a certain statistic defined on integer partitions.
In particular, we give asymptotics for the moments of the $M_2$-rank 
generating function and 
asymptotics for the $M_2$-rank generating function evaluated at roots of
unity. 
We recall that a partition of a non-negative integer $n$ is a non-increasing 
sequence of positive integers that sum to $n$. Rather than studying all
partitions, our attention will be focused on partitions without 
repeated odd parts. However, to describe our results and how they fit into 
the current theory, it is best to begin our discussion with ordinary
partitions. As an example, the partitions of $5$ are 
$5$, $4+1$, $3+2$, $3+1+1$, $2+2+1$, $2+1+1+1$, and $1+1+1+1+1$,
while the partitions of $5$ without repeated odd parts are
$5$, $4+1$, $3+2$, and $2+2+1$. We let $p(n)$ denote the number of partitions
of $n$ and let $p_2(n)$ denote the number of partitions of $n$ without
repeated odd parts.
 
A classic statistic defined on integer partitions is Dyson's rank of a partition
\cite{Dyson1}. The rank of a partition is defined as the largest part minus the
number of parts. With the partitions of $5$ listed above, the respective
ranks are $4$, $2$, $1$, $0$, $-1$, $-2$, and $-4$. We let $N(m,n)$ denote the
number of partitions of $n$ with rank $m$. Of course, the rank was
defined with a purpose in mind. Two of Ramanujan's three famous congruences for
$p(n)$ are $p(5n+4)\equiv0\pmod{5}$ and $p(7n+5)\equiv0\pmod{7}$. With 
$N(r,m,n)$ denoting the number of partitions
of $n$ with rank congruent to $r$ modulo $m$, Dyson conjectured that
$N(r,5,5n+4)=\frac{p(5n+4)}{5}$ and $N(r,7,7n+5)=\frac{p(7n+5)}{7}$. That is to
say, grouping the partitions of $5n+4$ according to their rank modulo $5$ gives
$5$ equinumerous sets and grouping the partitions of $7n+5$ according to their 
rank modulo $7$ gives $7$ equinumerous sets. One can verify this is indeed the 
case with the partitions of $5$ listed above. This conjecture was resolved by
Atkin and Swinnerton-Dyer by non-trivial means in \cite{AtkinSwinnertonDyer1}.

There is much interest in the rank past these two congruences. Before continuing,
we should give names to our generating functions. We let
\begin{align*}
P(q) &:= \sum_{n=0}^\infty p(n)q^n,&
P2(q) &:= \sum_{n=0}^\infty p_2(n)q^n,&
R(\zeta;q) &:= \sum_{n=0}^\infty\sum_{m=-\infty}^\infty N(m,n)\zeta^mq^n.
\end{align*}
A key issue is that while $P(q)$ (and $P2(q)$) is essentially a weight 
$-\frac{1}{2}$ modular form, the function $R(\zeta;q)$ is not. Many of Ramanujan's
odd order mock theta functions can be expressed in terms of $R(\zeta;q)$
(see \cite{GordonMcintosh1} where $R(\zeta;q)$ is $h_3(\zeta,q)$ and
\cite{HickersonMortenson1} where $R(\zeta;q)$ is $g(\zeta,q)$ up to minor 
factors). For this reason $R(\zeta;q)$ is called a universal mock theta function.
While once illusive, the automorphic properties of $R(\zeta;q)$ are now well 
understood. The rank generating function is 
essentially a mock Jacobi form and, when $\zeta$ is specialized to a
root of unity times a fractional power of $q$, $R(\zeta;q)$ is a mock modular form
(in that it is the so-called holomorphic part of a harmonic Maass form). The
definitions of these terms are somewhat involved, so we direct the reader to 
\cite{BringmannFolsomOnoRolen1, BringmannOno1, Zagier1}.

There is another rank function, the $M_2$-rank, and it fits into a similar 
framework. For a partition $\pi$, we let $\ell(\pi)$ denote the largest part 
of $\pi$ and $\#(\pi)$ denote the number of parts of $\pi$. The $M_2$-rank of $\pi$ is
defined as $\left\lceil\ell(\pi)/2\right\rceil-\#(\pi)$. 
To see this is a reasonable
definition from the standpoint of simple combinatorics, one should view partitions
in terms of their $2$-modular Ferrers diagram. For an ordinary Ferrers diagram
we take a partition $\pi=\pi_1+\pi_2+\dotsb+\pi_m$ and first draw a row of 
$\pi_1$ boxes, then $\pi_2$ boxes below that, and so on, ending with a row of
$\pi_m$ boxes. In this way each box has weight $1$. For a $2$-modular graph,
we instead write each part $\pi_i$ as a sequence of $2$'s possibly followed
by a single $1$. We then use this $2$-modular representation of each $\pi_i$
to draw and label the boxes. The $2$-modular Ferrers diagrams are given in Figure 1
for the partitions of $5$ without repeated odd parts.
While the rank of a partition is the length of the first row
minus the number of rows in the ordinary Ferrers diagram, the $M_2$-rank is
the length of the first row minus the number of rows in the 
$2$-modular Ferrers diagram. It turns out it is natural to consider the 
$M_2$-rank just for partitions without repeated odd parts. As a quick justification
for this, we notice that conjugating a Ferrers diagram (flipping the picture along
the main diagonal) of an ordinary partition results in another Ferrers diagram.
However, to conjugate a $2$-modular Ferrers diagram and get another $2$-modular
Ferrers diagram, the underlying partition must not have any repeated odd parts.
We let $N2(m,n)$ denote the number of partitions without repeated odd parts
of $n$ with $M_2$-rank $m$.

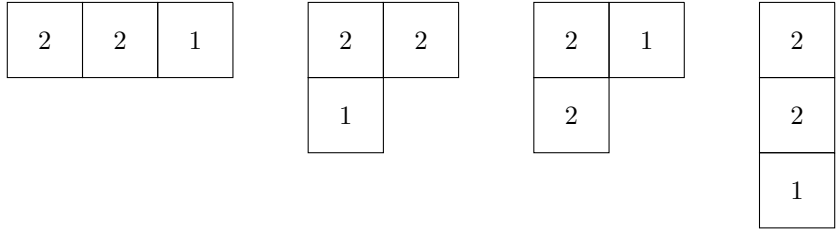
\begin{figure}\label{Figure 1}\caption{2-Modular Diagrams of $5$, $4+1$, $3+2$, and $2+2+1$}
\begin{tikzpicture}[scale=1]
\node at (0,0) {};	

\boxWithLabel{1}{0}{0}{2}
\boxWithLabel{1}{1}{0}{2}
\boxWithLabel{1}{2}{0}{1}

\boxWithLabel{1}{4}{0}{2}
\boxWithLabel{1}{5}{0}{2}
\boxWithLabel{1}{4}{-1}{1}

\boxWithLabel{1}{7}{0}{2}
\boxWithLabel{1}{8}{0}{1}
\boxWithLabel{1}{7}{-1}{2}
		
\boxWithLabel{1}{10}{0}{2}
\boxWithLabel{1}{10}{-1}{2}
\boxWithLabel{1}{10}{-2}{1}
\end{tikzpicture}
\end{figure}

The $M_2$-rank, as defined above, was introduced by Berkovich and Garvan in
\cite{BerkovichGarvan1}. The $M_2$-rank enjoys many of the same properties as
the ordinary rank. While it is not used to establish congruences for $p_2(n)$,
it is used in proving congruences for certain other partition functions
\cite{GarvanJenningsShaffer1}. It turns out the generating function for 
$N2(m,n)$, which we denote by $R2(\zeta;q)$, is also a universal mock theta function as many
even order mock theta functions can be expressed in terms of $R2(\zeta;q)$
(see \cite{GordonMcintosh1} where $R2(\zeta;q)$ is $h_2(\zeta,-q)$ and
\cite{HickersonMortenson1} where $R2(\zeta;q)$ is $k(\zeta^{\frac12},-q)$ up to minor
factors). The function $R2(\zeta;q)$ is in the same class of automorphic functions
as $R(\zeta;q)$. Furthermore, both the generating functions of the rank and the 
$M_2$-rank can be found among the identities in Ramanujan's lost notebook (for 
$R(\zeta;q)$ see \cite[Chapter 2]{AndrewsBerndt2} and for $R2(\zeta;q)$ see 
\cite[Chapter 12]{AndrewsBerndt1}).

Before finally explaining the contributions of this article, we must speak a 
bit about asymptotics. Hardy and Ramanujan's asymptotic for the partition 
function is $p(n)\sim \frac{1}{4n}\exp\left(\pi\sqrt{\frac{2n}{3}}\right)$ as 
$n\rightarrow\infty$, and their asymptotic expansion \cite{HardyRamanujan1}  is
\begin{align*}
p(n) = \frac{1}{\sqrt{2}\pi}\sum_{k=1}^{\lfloor\sqrt{n}\rfloor} \mathcal{A}_k(n) k^{\frac12}
\frac{d}{dn}\left(
	\frac{ \exp\left(\frac{\pi}{k}\sqrt{\frac{2}{3}\left( n-\frac{1}{24} \right)}\right) }
	{2\sqrt{n-\frac{1}{24}}}
\right)
+O\left(n^{-\frac14}\right),
\end{align*}
where 
\begin{align*}
\mathcal{A}_k(n)&:=\sum_{\substack{0\le h<k,\\(h,k)=1}}\omega_{h,k}\exp\left(-\tfrac{2\pi inh}{k}\right),
\end{align*}
and $\omega_{h,k}$ are certain $24$th roots of unity. Rademacher \cite{Rademacher1},
improving upon Hardy and Ramanujan's circle method, found that in fact 
\begin{align*}
p(n)&=\frac{1}{\sqrt{2}\pi}
\sum_{k=1}^\infty \mathcal{A}_k(n)k^{\frac{1}{2}}
\frac{d}{dn}
\left(
	\frac{\sin\left(\frac{\pi}{k}\sqrt{\frac{2}{3}\left( n-\frac{1}{24} \right)}\right)}
	{\sqrt{n-\frac{1}{24}}}
\right).
\end{align*}

As students of Rademacher, Dragonette \cite{Dragonette1} and Andrews 
\cite{Andrews2} established related asymptotics for the coefficients of $R(-1;q)$.
Both gave their estimates as exponential sums similar to Hardy and Ramanujan's 
formula for $p(n)$. They conjectured that a Rademacher type formula also existed. 
In particular, if $\alpha(n)$ is the coefficient of $q^n$ in $R(-1;q)$, then
\begin{align*}
\alpha(n)
&=
	\sum_{k=1}^\infty 
	\frac{ (-1)^{\lfloor\frac{k+1}{2}\rfloor} \mathcal{A}_{2k}\left(n-\frac{k(1+(-1)^k)}{4}\right)
		\exp\left(\frac{\pi}{k}\sqrt{\frac{n}{6}-\frac{1}{144}}\right)	}
	{\sqrt{k\left(4n-\frac{1}{6}\right)}}
.
\end{align*}
In \cite{BringmannOno2}, Bringmann and Ono proved this conjecture.

For a modern exposition of the circle method applied to the partition function,
one should consult \cite[Chapter 5]{Andrews1}.
Bringmann \cite{Bringmann1} demonstrated that one can extend this method to 
mock modular forms. In particular, Bringmann gave a formula for the coefficients of 
$R(e^{\frac{2\pi ia}{c}};q)$ in a form similar to that of Hardy and Ramanujan's
formula for $p(n)$. Going further Bringmann, Mahlburg, and Rhoades
\cite{BringmannMahlburgRhoades1} found it possible to also deduce such formulas
for the moments of $R(\zeta;q)$, which are defined as
$N_\ell(n):=\sum_{m=-\infty}^\infty m^\ell N(m,n)$.

This is the treatment we give $R2(\zeta;q)$ and is the main content of our article. 
Specifically, we first consider the $M_2$-rank moments $N2_\ell(n)$.
We deduce an asymptotic expansion for $N2_\ell(n)$ of a form similar to that 
of Hardy and Ramanujan for 
$p(n)$.  This formula is stated in Theorem \ref{TheoremRankMomentExpansions} and
the resulting asymptotic value of $N2_\ell(n)$ is \eqref{EqN2RankMomentAsymptoticValue}.
Second, we determine an expansion for $A\left(\frac{a}{c};n\right)$,
which are the coefficients of $R2(e^{\frac{2\pi ia}{c}};q)$. These formulas
are stated in Theorem \ref{TheoremAsymptoticForM2Rank}. 
Third, using Theorem \ref{TheoremAsymptoticForM2Rank} we determine asymptotic 
values for certain combinations
of $N2(r,m,n)$ (the number of partitions of $n$ without repeated odd parts
and $M_2$-rank congruent to $r$ modulo $m$). Using these asymptotic
values, we deduce a few inequalities among certain $N2(r,m,n)$. With these
inequalities are included the remaining conjectured inequalities of Mao from 
\cite{Mao1}. This is contained in Section 7. In the next section we give
the necessary definitions to state our results, introduce the various
functions relevant to our study, and end with an outline of the rest of the
article.

In particular, the conjectured inequalities of Mao that we prove are as
follows. We note that while our proofs of these inequalities are via
asymptotics, we do indeed prove that the inequalities hold for all
stated values of $n$.

\begin{theorem} Let $N2(r,m,n)$ denote the number of partitions without
repeated odd parts of $n$ with $M2$-rank congruent to $r$ modulo $m$.
Then the following inequalities hold,
\begin{align*}
N2(0,6,n)+N2(1,6,n)&> N2(2,6,n)+N2(3,6,n) &\mbox{for }n\ge0
,\\
N2(1,10,n)+N2(2,10,n)&> N2(3,10,n)+N2(4,10,n) &\mbox{for }n\ge3.
\end{align*}
\end{theorem}

\section*{Acknowledgments}
The authors thank Kathrin Bringmann for suggesting this project and for
useful comments and discussion.

\section{Preliminaries and Statement of Main Results}

As discussed in the introduction, we let
$N2(m,n)$ denote the number of partitions of $n$ without repeated odd parts and
with $M_2$-rank equal to $m$. The generating function for $N2(m,n)$ is 
$R2(\zeta;q)$ and from \cite{LovejoyOsburn1} we have that
\begin{align}\label{M2RankAsGeneralizedLambert}
R2(\zeta;q)
&=
\sum_{n=0}^\infty
\frac{q^{n^2}\aqprod{-q}{q^2}{n}}{\aqprod{\zeta q^2,\zeta^{-1}q^2}{q^2}{n}}
=
\frac{\aqprod{-q}{q^2}{\infty}}{\aqprod{q^2}{q^2}{\infty}}
\left(1+
	\sum_{n=1}^\infty \frac{ (1-\zeta)(1-\zeta^{-1})(-1)^nq^{2n^2+n}(1+q^{2n}) }
		{(1-\zeta q^{2n})(1-\zeta^{-1}q^{2n})}
\right).
\end{align}
Here we use the standard $q$-Pochhammer notation given as
$\aqprod{a}{q}{n}:=\prod_{j=0}^{n-1}(1-aq^j)$, for $n$ either a non-negative
integer or $\infty$, and
$\aqprod{a_1,\dotsc,a_m}{q}{n}:=\aqprod{a_1}{q}{n}\dotsm\aqprod{a_m}{q}{n}$.
For $\ell$ a non-negative integer, we define the $\ell$-th $M_2$-rank moment and its
generating function by
\begin{align*}
N2_\ell(n) &:= \sum_{m=-\infty}^\infty m^\ell N2(m,n)
,&
R2_\ell(q) &:= \sum_{n=0}^\infty N2_\ell(n)q^n.
\end{align*}
Due to the symmetry $R2(\zeta ;q)=R2(\zeta^{-1};q)$, we have that $N2(m,n)=N2(-m,n)$.
In particular $N2_{\ell}(n)=0$ when $\ell$ is odd and so only the even moments are of interest to
us. 

Our goal is to determine asymptotics via the circle method for the coefficients 
of $R2_{2k}(q)$ and $R2(e^{\frac{2\pi ia}{c}};q)$. We state these formulas 
shortly, but first we require some additional definitions and notation. 
For $h$ an integer, we let $\widetilde{h}$ denote the value of
$h$ modulo $4$ with $\widetilde{h}\in\{-1,0,1,2\}$. 
For relatively prime integers $h$ and $k$, with $k>0$, we let $[-h]_k$ denote 
a choice of an inverse of $-h$ modulo $k$. When $k$ is odd, we choose $[-h]_k$
so that $32$ divides $[-h]_k$. We will make additional assumptions on how
to choose $[-h]_k$, which are discussed later. These additional assumptions
only appear in certain proofs, and are not needed to correctly read the statements
of propositions and theorems. For real $x$, we let $\lfloor x\rfloor$ denote the floor of $x$ and let
$\{x\}$ denote the fractional part of $x$ as given by $\{x\}:=x-\lfloor x\rfloor$.
Throughout this article, we let $\tau$ denote a point in the upper half-plane
$\mathcal{H}$ (so that $\IM{\tau}>0$) and let $q:=\exp\left(2\pi i\tau\right)$ 
(so that $|q|<1$). 
We use big $O$ notation and $\ll$ interchangeably, and use $O_\ell$ and $\ll_\ell$
to indicate dependencies of implicit constants.

Dedekind's  $\eta$-function is defined by 
$\eta(\tau):=q^{\frac{1}{24}}\aqprod{q}{q}{\infty}$.
This function satisfies the modular transformation
$\eta(A\tau)=\nu(A)\sqrt{\gamma\tau+\delta}\, \eta(\tau)$,
where $A=\begin{psmallmatrix}\alpha&\beta\\\gamma&\delta\end{psmallmatrix}
\in\SLTwo$ acts on the upper half-plane by Mobius transformations. 
We recall Mobius transformations are defined
as $A\tau=\frac{\alpha\tau+\beta}{\gamma\tau+\delta}$.
Here
$\nu(A)$ is a $24$th root of unity determined only by the matrix $A$ and
the standard branch of the square root is taken so that its value has positive
real part. Since $\nu(A)$ will appear in our formulas, we note a convenient form 
for the $\eta$-multiplier, which can be found as Theorem 2 in Chapter 4 of 
\cite{Knopp1}, is
\begin{align*}
\nu(A)
&=
\begin{cases}
\big(\frac{\delta}{|\gamma|} \big)
\exp\left(\frac{\pi i}{12}\left( (\alpha+\delta)\gamma - \beta\delta(\gamma^2-1) - 3\gamma\right)\right)
&
\mbox{ if } \gamma \equiv 1 \pmod{2},
\\				
\big(\frac{\gamma}{\delta} \big)
\exp\left(\frac{\pi i}{12}\left( 	(\alpha+\delta)\gamma - \beta\delta(\gamma^2-1) + 3\delta - 3 - 3\gamma\delta\right)\right)
&
\mbox{ if } \delta\equiv 1\pmod{2},
\end{cases}
\end{align*}
where $\left(\frac{m}{n}\right)$ is the generalized Legendre symbol as in
\cite{Shimura1}.

We define a related $24$th root of unity as follows. Suppose $h$ and $k$ 
are relatively prime integers with $k>0$. We see that
$\begin{psmallmatrix}h&-\frac{1+h[-h]_k}{k}\\k&-[-h]_k\end{psmallmatrix}$
is an element of $\SLTwo$ and 
$\frac{h+iz}{k}
=\begin{psmallmatrix}h&-\frac{1+h[-h]_k}{k}\\k&-[-h]_k\end{psmallmatrix}
\frac{[-h]_k+i/z}{k}$. We then define $\chi(h,[-h]_k,k)$ to be the 
$24$th root of unity given by
$\eta\left(\frac{h+iz}{k}\right) = \chi(h, [-h]_k, k) \sqrt{\frac{i}{z}}
\eta\left(\frac{[-h]_k+i/z}{k}\right)$. Using the above formula for 
$\nu(A)$ and properties of $\left(\frac{m}{n}\right)$, we see that
$\chi(h,[-h]_k,k)$ depends on the choice of $[-h]_k$ modulo $24k$
and $\chi(h,[-h]_k,k)^3$ depends on the choice of $[-h]_k$ modulo $8k$.

We let
\begin{align*}
\xi(h,[-h]_k,k)
&:=
	e^{\frac{\pi i}{4}}\chi\left( h,[-h]_k,k \right)^{-3}(-1)^{\frac{h-\wt{h}}{4}}
	\exp\left( \tfrac{\pi i\wt{h}}{8k} - \tfrac{\pi i[-h]_k\left(\wt{h}-h\right)^2}{16k} \right)
,\\
\xi_{\ell}^{\pm}(h,k)
&:=
	(-1)^{\ell+1}
	\exp\left(
		\tfrac{-\pi ih(2\ell+1)^2}{4k}
		\pm\tfrac{\pi i\left(\wt{h}-h\right)(2\ell+1)}{4k}
		+\tfrac{\pi i\wt{h}}{8k}		
	\right)
,\\
\alpha^{\pm}(\ell,k)
&:=
	\tfrac{1}{k}\left(-\ell + \tfrac{k-1}{2} \pm \tfrac{1}{4} \right)
.
\end{align*}
In principle, one can make any choice of $[-h]_k$, but must carry that specific
choice through all relevant calculations. To allow simplifications in various
formulas and calculations, we make some assumptions about our choice of 
$[-h]_k$. These assumptions deal only with fixed $k$ and $h$
and only appear in the proofs. When $k\equiv 0\pmod{4}$ we assume
$[-h]_{4k}=[-h]_k=[-h]_{k/4}$. In particular, when $k\equiv 0\pmod{4}$ we have
$\frac{1+[-h]_{k/4}}{2k}\equiv 0\pmod{4}$.  When $k\equiv 1\pmod{2}$, we assume
that $[-4h]_k=\frac{[-h]_k}{4}$. These are viewed as 
choices made for a fixed value of $k$ and $h$, so such a choice is possible. 
Clearly such choices would be impossible for all $k$ and $h$ simultaneously.

We let $I_\alpha(x)$ denote the modified Bessel function of the first kind. 
We recall the Bernoulli polynomials $B_n(x)$ are defined by
$\frac{t\exp(xt)}{\exp(t)-1}=:\sum_{n=0}^\infty \frac{B_n(x)t^n}{n!}$. Lastly
we define the constants $\kappa(a,b,c)$, for non-negative integers $a$, $b$, and $c$,
by $\kappa(a,b,c):=\frac{(-1)^{a+c} (2(a+b+c))! B_{2c}(\frac{1}{2})}{a!(2b+1)!(2c)!\pi^a4^{a+b}}$.
We now state our main results.

\begin{theorem}\label{TheoremRankMomentExpansions}
For $\ell$ a positive integer, and $N=\lfloor\sqrt{n}\rfloor$, we have the asymptotic expansion
\begin{align*}
N2_{2\ell}(n)
&=
	2\pi\hspace{-1em}
	\sum_{\substack{ 1\le k\le N,\\ k\equiv 0\pmod{4}}}	
	\frac{A_k(n)}{k}
	\sum_{a+b+c=\ell}
		k^a (8n-1)^{\frac{2a+4c-3}{4}} \kappa(a,b,c) 
		I_{a+2c-\frac{3}{2}}\left( \frac{\pi}{2k}\sqrt{8n-1} \right)
	\nonumber\\&\quad
	+
	\frac{\pi}{\sqrt{2}}\hspace{-1em}
	\sum_{\substack{ 1\le k\le N,\\ k\equiv\pm 1\pmod{4}}}	
	\frac{A_k(n)}{k}
	\sum_{a+b+c=\ell}
		(2k)^a (8n-1)^{\frac{2a+4c-3}{4}} \kappa(a,b,c)   
	I_{a+2c-\frac{3}{2}}\left( \frac{\pi}{4k}\sqrt{8n-1} \right)
	\nonumber\\&\quad
	+	
	E_{\ell}(n)
,
\end{align*}
where
\begin{align*}
A_k(n)
&:=
\begin{cases} \displaystyle
	-i\sum_{\substack{ 0\le h< k,\\ (h,k)=1}}	
	\wt{h} \exp\left(-\tfrac{2\pi ihn}{k}\right) \xi\left(h,[-h]_{4k},\tfrac{k}{4}\right) 
	&
	\mbox{ if } k\equiv 0\pmod{4},
	\\
	\displaystyle
	-\sum_{\substack{ 0\le h< k,\\ (h,k)=1}}	
	\frac{ \exp\left(-\frac{2\pi ihn}{k} - \frac{\pi i [-h]_k}{16k}\right)   \xi(4h,\tfrac{[-h]_k}{4},k)}
		{\sin(\tfrac{\pi k}{4})}
	&
	\mbox{ if } k\equiv \pm1\pmod{4},
\end{cases}
\\
E_{\ell}(n) 
&\ll_\ell 
\begin{cases}	
	n\log(n)
	&
	\mbox{ if } \ell=1,
	\\
	n^{2\ell-1}
	&
	\mbox{ if } \ell\ge 2,
\end{cases}
\end{align*}
as $n\rightarrow\infty$.
\end{theorem}

When $\alpha$ is a half-integer,
$I_{\alpha}(x) \sim \frac{e^x}{\sqrt{2\pi x}}$ as $x\rightarrow\infty$,
and so taking the terms corresponding to $k=1$ and $c=\ell$ give the main 
asymptotic. In particular, we find that
\begin{align}\label{EqN2RankMomentAsymptoticValue}
	N2_{2\ell}(n)
	&\sim
	(-1)^{\ell}\sqrt{2}(8n)^{\ell-1}B_{2\ell}(\tfrac{1}{2}) \exp\left( \pi\sqrt{\tfrac{n}{2}} \right)
	,
\end{align}
as $n\rightarrow\infty$.

\begin{theorem}\label{TheoremAsymptoticForM2Rank}
Suppose $a$ and $c$ are integers with $c>0$ and $c\nmid 2a$.
Then with $A\left(\frac{a}{c};n\right)$ defined by
$R2(e^{\frac{2\pi ia}{c}};q)=\sum_{n=0}^\infty A\left(\frac{a}{c};n\right)q^n$,
and $N=\lfloor\sqrt{n}\rfloor$,
we have that
\begin{align*}
A\left(\frac{a}{c};n\right)
&=
-\frac{8i\sin(\frac{\pi a}{c})}{\sqrt{8n-1}}
\sum_{\substack{1\leq k\leq N,\\ k\equiv 0\!\!\pmod{4},\\ 2c\nmid ka }}
\frac{(-1)^{\Floor{\frac{ka}{2c}}}}{\sqrt{k}}
\sum_{j=1}^{5}
\sum_{m=0}^{M_j}
\cosh\left(\frac{\pi}{k}\sqrt{r_{j,a,c,k}(m)(8n-1)}\right)
D_{j,a,c,k,n}(m)
\\&\quad
+
\frac{8i\sin(\frac{\pi a}{c})}{\sqrt{8n-1}}
\sum_{\substack{1\leq k\leq N,\\ k\equiv 0\!\!\pmod{4},\\ \Fractional{\frac{ka}{2c}}>\frac{3}{4} }}
\frac{(-1)^{\Floor{\frac{ka}{2c}}}}{\sqrt{k}}
\cosh\left(\frac{2\pi(\Fractional{\frac{ka}{2c}}-\frac{3}{4})\sqrt{(8n-1)}}{k}\right)
D^{-}_{a,c,k,n}
\\&\quad
+
\frac{8i\sin(\frac{\pi a}{c})}{\sqrt{8n-1}}
\sum_{\substack{1\leq k\leq N,\\ k\equiv 0\!\!\pmod{4},\\ \Fractional{\frac{ka}{2c}}>\frac{1}{4} }}
\frac{(-1)^{\Floor{\frac{ka}{2c}}}}{\sqrt{k}}
\cosh\left(\frac{2\pi(\Fractional{\frac{ka}{2c}}-\frac{1}{4})\sqrt{(8n-1)}}{k}\right)
D^{+}_{a,c,k,n}
\\&\quad
-
\frac{4\sin(\frac{\pi a}{c})}{\sqrt{8n-1}}
\sum_{\substack{1\leq k\leq N,\\ k\equiv 0\!\!\pmod{4},\\ 2c\mid ka }}
\frac{(-1)^{\frac{ka}{2c}}}{\sqrt{k}}
\cosh\left(\frac{\pi \sqrt{8n-1} }{2k}\right)
C_{0,a,c,k,n}
\\&\quad
-
\frac{4\sin(\frac{\pi a}{c})}{\sqrt{8n-1}}
\sum_{\substack{1\leq k\leq N,\\ k\equiv 1\!\!\pmod{2},\\ c\nmid 2ka }}
\frac{(-1)^{\Floor{\frac{2ka}{c}}}}{\sqrt{k}}
\sum_{j=6}^{7}
\sum_{m=0}^{M_j}
\cosh\left(\frac{\pi}{k}\sqrt{r_{j,a,c,k}(m)(8n-1)}\right)
D_{j,a,c,k,n}(m)
\\&\quad
+
\frac{8\sin(\frac{\pi a}{c})}{\sqrt{8n-1}}
\sum_{\substack{1\leq k\leq N,\\ k\equiv 1\!\!\pmod{2},\\ \Fractional{\frac{2ka}{c}}>\frac{1}{2} }}
\frac{(-1)^{\Floor{\frac{2ka}{c}}}}{\sqrt{k}}
\cosh\left(\frac{\pi(\Fractional{\frac{2ka}{c}}-\frac{1}{2})\sqrt{(8n-1)}}{2k}\right)
D_{a,c,k,n}
\\&\quad
+
\frac{2i\sin(\frac{\pi a}{c})}{\sqrt{8n-1}}
\sum_{\substack{1\leq k\leq N,\\ k\equiv 1\!\!\pmod{2},\\ c\mid 2ka }}
\frac{(-1)^{\frac{2ka}{c}}}{\sqrt{k}}
\cosh\left(\frac{\pi \sqrt{8n-1} }{4k}\right)
C_{1,a,c,k,n}
+
O_{a,c}\left(\sqrt{n}\right),
\end{align*}
where
\begin{align*}
D_{1,a,c,k,n}(m)
&:=
	\sum_{\substack{0\leq h<k,\\ (h,k)=1}}
	\wt{h}	\xi\left(h,[-h]_{4k},\tfrac{k}{4} \right)
	\exp\left( 
		-\tfrac{2\pi inh}{k}
		+\tfrac{4\pi i[-h]_{4k}}{k} \left(
			-\Floor{\tfrac{ka}{2c}}^2 
			- (2m+\tfrac{3}{2})\Floor{\tfrac{ka}{2c}} 
		\right)	
	\right)	
,\\
D_{2,a,c,k,n}(m)
&:=
	\sum_{\substack{0\leq h<k,\\ (h,k)=1}}
	\wt{h} 	\xi\left(h,[-h]_{4k},\tfrac{k}{4}\right)
	\exp\left( 
		-\tfrac{2\pi inh}{k}
		+\tfrac{4\pi i[-h]_{4k}}{k} \left(
			-\Floor{\tfrac{ka}{2c}}^2 
			+ (2m+\tfrac{1}{2})\Floor{\tfrac{ka}{2c}} 
			+ 2m
			+ \tfrac{3}{2}
		\right)	
	\right)	
,\\
D_{3,a,c,k,n}(m)
&:=	
	\sum_{\substack{0\leq h<k,\\ (h,k)=1}}
	\wt{h}	\xi\left(h,[-h]_{4k},\tfrac{k}{4}\right)
	\exp\left( 
		-\tfrac{2\pi inh}{k}		
		+\tfrac{4\pi i[-h]_{4k}}{k} \left(
			-\Floor{\tfrac{ka}{2c}}^2 
			- (2m+\tfrac{1}{2})\Floor{\tfrac{ka}{2c}} 
		\right)	
	\right)	
,\\
D_{4,a,c,k,n}(m)
&:=
	\sum_{\substack{0\leq h<k,\\ (h,k)=1}}
	\wt{h} \xi\left(h,[-h]_{4k},\tfrac{k}{4}\right)
	\exp\left(
		-\tfrac{2\pi inh}{k}		
		+\tfrac{4\pi i[-h]_{4k}}{k} \left(
			-\Floor{\tfrac{ka}{2c}}^2 
			- (2m+\tfrac{1}{2})\Floor{\tfrac{ka}{2c}} 
			+ \tfrac{1}{2} 
		\right)	
	\right)	
,\\
D_{5,a,c,k,n}(m)
&:=
	\sum_{\substack{0\leq h<k,\\ (h,k)=1}}
	\wt{h}	\xi\left(h,[-h]_{4k},\tfrac{k}{4}\right)
	\exp\left( 
		-\tfrac{2\pi inh}{k}
		+\tfrac{4\pi i[-h]_{4k}}{k} \left(
			-\Floor{\tfrac{ka}{2c}}^2 
			+ (2m+\tfrac{3}{2})\Floor{\tfrac{ka}{2c}} 
			+2m
			+ \tfrac{5}{2}
		\right)	
	\right)	
,\\
D_{6,a,c,k,n}(m)
&:=
	\sum_{\substack{0\leq h<k,\\ (h,k)=1}}
	\cos\left( \tfrac{\pi}{2} \Floor{\tfrac{2ka}{c}}k \right)
	\csc\left( \tfrac{\pi k}{4} \right)
	\xi\left(4h,\tfrac{[-h]_k}{4},k\right)
	\exp\left( 
		-\tfrac{2\pi inh}{k}
	\right)
	\\[-3.75ex]&\quad\quad\,\,\qquad\times
	\exp\left(
		\tfrac{\pi i [-h]_k}{4k} \left(
			-\Floor{\tfrac{2ka}{c}}^2-(2m+1)\Floor{\tfrac{2ka}{c}}-\tfrac{1}{4}
		\right)
	\right)
,\\[1ex]
D_{7,a,c,k,n}(m)
&:=
	\sum_{\substack{0\leq h<k,\\ (h,k)=1}}
	i^{1+k}
	\sin\left( \tfrac{\pi}{2}  \Floor{\tfrac{2ka}{c}} k \right)
	\csc\left( \tfrac{\pi k}{4} \right)
	\xi\left(4h,\tfrac{[-h]_k}{4},k\right)
	\exp\left(
		-\tfrac{2\pi inh}{k}
	\right)
	\\[-3.75ex]&\quad\quad\,\,\qquad\times
	\exp\left(
		\tfrac{\pi i[-h]_k}{4k} \left(
			-\Floor{\tfrac{2ka}{c}}^2+(2m+1)\Floor{\tfrac{2ka}{c}}+\tfrac{7}{4} +2m
		\right) 	
	\right)
,\\[1ex]
D^{\pm}_{a,c,k,n}
&:=
	\sum_{\substack{0\leq h<k,\\ (h,k)=1} }
	\pm\wt{h}
	\xi\left(h,[-h]_{4k},\tfrac{k}{4}\right)
	\exp\left( 
		-\tfrac{2\pi inh}{k}
	\right)
	\\[-3.75ex]&\quad\quad\,\,\qquad\times
	\exp\left(
		\tfrac{4\pi i[-h]_{4k}}{k} \left(
			-\Floor{\tfrac{ka}{2c}}^2 
			- \Floor{\tfrac{ka}{2c}} 
			\mp\wt{h}(2\Floor{\tfrac{ka}{2c}}+1)\tfrac{h-\wt{h}}{4}			
			- \tfrac{1}{4}
		\right)	
		\mp
		\tfrac{\pi i\wt{h}}{k}\left(2\Floor{\tfrac{ka}{2c}}+1\right)
	\right)	
,\\[1ex]
D_{a,c,k,n}
&:=
	\sum_{\substack{0\leq h<k,\\ (h,k)=1} }
	\sin\left( 	\tfrac{\pi \left(2\Floor{\frac{2ka}{c}}+1\right) k }{4}  \right)
	\xi\left(4h,\tfrac{[-h]_k}{4},k\right)
	\exp\left( 
		-\tfrac{2\pi inh}{k}
		+\tfrac{\pi i[-h]_k}{4k}\left( -\Floor{\tfrac{2ka}{c}}^2-\Floor{\tfrac{2k}{c}} -\tfrac{1}{4} \right)
	\right)			
,\\
C_{0,a,c,k,n}
&:=
	\sum_{\substack{0\leq h<k,\\ (h,k)=1}}
	\wt{h}
	\csc\left( \tfrac{\pi a[-h]_{4k}}{c} \right)		
	\xi\left(h,[-h]_{4k},\tfrac{k}{4}\right)
	\exp\left( -\tfrac{2\pi inh}{k}	-\tfrac{\pi i[-h]_{4k} ka^2}{c^2} \right)
,\\
C_{1,a,c,k,n}
&:=
	\sum_{\substack{0\leq h<k,\\ (h,k)=1}}
	\frac{ \cos\left( \frac{\pi a (1+h[-h]_{k})}{c} \right) } 
		{ \sin\left(\frac{\pi k}{4} \right) \sin\left(\frac{\pi a[-h]_{k}}{2c}\right) }
	\xi\left(4h,\tfrac{[-h]_k}{4},k\right)
	\exp\left( 
		-\tfrac{2\pi inh}{k}
		-\tfrac{\pi ika^2[-h]_k}{c^2}
		-\tfrac{\pi i[-h]_k}{16k}
	\right)
,
\end{align*}
and
\begin{align*}
r_{1,a,c,k}(m)
&:=
	4\Fractional{\tfrac{ka}{2c}}^2 - 6\Fractional{\tfrac{ka}{2c}} 
	- 8m\Fractional{\tfrac{ka}{2c}} + \tfrac{1}{4}
,\\
r_{2,a,c,k}(m)
&:=
	4\Fractional{\tfrac{ka}{2c}}^2 + 2\Fractional{\tfrac{ka}{2c}} 
	- 8m(1-\Fractional{\tfrac{ka}{2c}}) - \tfrac{23}{4}
,\\
r_{3,a,c,k}(m)
&:=
	4\Fractional{\tfrac{ka}{2c}}^2 - 2\Fractional{\tfrac{ka}{2c}} 
	- 8m\Fractional{\tfrac{ka}{2c}} + \tfrac{1}{4}
,\\
r_{4,a,c,k}(m)
&:=
	4\Fractional{\tfrac{ka}{2c}}^2 - 2\Fractional{\tfrac{ka}{2c}} 
	- 8m\Fractional{\tfrac{ka}{2c}} - \tfrac{7}{4}
,\\
r_{5,a,c,k}(m)
&:=
	4\Fractional{\tfrac{ka}{2c}}^2 + 6\Fractional{\tfrac{ka}{2c}} 
	- 8m(1-\Fractional{\tfrac{ka}{2c}}) - \tfrac{39}{4}
,\\
r_{6,a,c,k}(m)
&:=
	\tfrac{1}{4}\Fractional{\tfrac{2ka}{c}}^2 - \tfrac{1}{4}\Fractional{\tfrac{2ka}{c}} 
	- \tfrac{m}{2}\Fractional{\tfrac{2ka}{c}} + \tfrac{1}{16}
,\\
r_{7,a,c,k}(m)
&:=
	\tfrac{1}{4}\Fractional{\tfrac{2ka}{c}}^2 + \tfrac{1}{4}\Fractional{\tfrac{2ka}{c}} 
	- \tfrac{m}{2}(1-\Fractional{\tfrac{2ka}{c}}) - \tfrac{7}{16}
.
\end{align*}
An explicit bound on the error term is given in \eqref{EqFinalBoundOnErrorTerm}.
\end{theorem}

We introduce the various functions that will arise in the proofs of these 
theorems. As is now common when dealing with mock modular forms, we use various
functions from Zwegers' groundbreaking thesis \cite{Zwegers1}.
For $x\in\mathbb{R}$ we let
\begin{align*}
\beta(x) &:= \int_{x}^\infty t^{-\frac{1}{2}} e^{-\pi t}dt
,&
E(x) &:= \sgn{x}(1-\beta(x^2)).
\end{align*}
For $u\in\mathbb{C}$ and $\tau\in\mathcal{H}$, we let
\begin{align*}
\vartheta(u;\tau)
:&=
	\sum_{n\in\frac12+\mathbb{Z}} q^{\frac12n^2} \exp\left( 2\pi in\left(u+\tfrac12\right)\right)
	=	
	-iq^{\frac18}e^{-\pi iu}
	\aqprod{ e^{2\pi iu}, e^{-2\pi iu}q, q}
		{q}{\infty}
,\\
S(u;\tau)
:&=
	\sum_{n\in \frac{1}{2}+\mathbb{Z}}	
	\left(
		\sgn{n} 
		- 
		E\left( \left(n+\tfrac{\IM{u}}{\IM{\tau}}\right)\sqrt{2\IM{\tau}} \right)		
	\right)	
	(-1)^{n-\frac{1}{2}}q^{-\frac12 n^2} e^{- 2\pi inu }
,\\
H(u;\tau)
:&=
	\int_{-\infty}^\infty
	\frac{\exp\left(\pi i\tau x^2 - 2\pi xu\right)}{\cosh(\pi x)}dx
.
\end{align*}
For $u,v\in \mathbb{C}\backslash (\mathbb{Z}\tau+\mathbb{Z})$ and $\tau\in\mathcal{H}$,
we let
\begin{align*}
\mu(u,v;\tau)
&:=
	\frac{ e^{\pi iu}}{\vartheta(v;\tau)}
	\sum_{n=-\infty}^\infty
	\frac{ (-1)^n q^{\frac12n(n+1)} e^{2\pi inv}  }{1 -  e^{2\pi iu}q^{n} }
,\\
\tmu(u,v;\tau) &:= \mu(u,v;\tau)+\tfrac{i}{2}S(u-v;\tau).
\end{align*}
Any use of these functions is under the assumption that the parameters
are chosen so that the functions are well defined, even if we do not state
these conditions explicitly.
These functions satisfy various elliptic and modular transformations.
In particular,
\begin{align}
\label{EqHEven}
H(-u;\tau) &= H(u;\tau)
,\\
\label{EqHEllipticProperty1}
H(u+\tau;\tau) &= -q^{\frac12}e^{2\pi iu}H(u;\tau)+2q^{\frac38}e^{2\pi iu}
,\\
\label{EqSEllipticProperty1}
S(u+1;\tau) &= -S(u;\tau)
,\\
\label{EqSModularProperty1}
S(u;\tau+1) &= e^{-\frac{\pi i}{4}}S(u;\tau)
,\\
\label{EqSModularProperty3}
S\left(\frac{u}{\tau};-\frac{1}{\tau}\right)	
&=
	\sqrt{-i\tau}\exp\left( -\tfrac{\pi iu^2}{\tau}\right)(H-S)(u;\tau)
.
\end{align}
If $k$, $\ell$, $m$, $n\in\mathbb{Z}$, then
\begin{align}
\label{EqMuTildeEllipticProperty}	
\tmu\left( u + k\tau + \ell, v + m\tau + n ;\tau  \right)
&=
	(-1)^{k+\ell+m+n} q^{\frac12(k-m)^2} e^{2\pi i(k-m)(u-v)}	
	\tmu\left(u,v;\tau \right)	
.
\end{align}
If $A=\begin{psmallmatrix}\alpha&\beta\\\gamma&\delta\end{psmallmatrix}\in\SLTwo$, then
\begin{align}
\label{EqMuTildeModularProperty}	
\tmu\left( \frac{u}{\gamma\tau+\delta}, \frac{v}{\gamma\tau+\delta} ; \frac{\alpha\tau+\beta}{\gamma\tau+\delta}  \right)
&=
\nu(A)^{-3}\sqrt{\gamma\tau+\delta} \exp\left(-\frac{\pi i\gamma(u-v)^2}{\gamma\tau+\delta}\right) \tmu\left(u,v;\tau\right)
.
\end{align}

As a matter of notation, sums indexed by $\pm$ should be read as
\begin{align*}
\sum_{\pm}f(\pm) &= f(+) + f(-),
\end{align*}
where $f(\pm)$ is a summand depending on the choice of $+$ or $-$.

By elementary rearrangements of the series on the far right side of
\eqref{M2RankAsGeneralizedLambert}, we find that
\begin{align}\label{M2RankInTermsOfMu}
R2(\zeta;\tau) &= i(1-\zeta)\left(\zeta^{-1}\mu(2u,-\tau;4\tau) - \mu(2u,\tau;4\tau)\right),
\end{align}
where $\zeta=\exp(2\pi iu)$. From this it is now clear that the functions introduced 
above are indeed relevant to our study of $R2(\zeta;\tau)$. We note this form of 
$R2(\zeta;\tau)$ has a removable singularity at $\zeta=\pm1$.

So that the purpose of our later calculations is clear, we briefly recall the
circle method. If $\frac{h_0}{k_0}$, $\frac{h}{k}$, and $\frac{h_1}{k_1}$
are three successive Farey fractions of order $N$, then we define
\begin{align*}
\vartheta^{\prime}_{0,1} 
&:=
	\frac{1}{N+1}
,&
\vartheta^{\prime}_{h,k} 
&:=
	\frac{h}{k} - \frac{h_0+h}{k_0+k}
	=
	\frac{1}{k(k_0+k)}		
	\quad\quad\mbox{ for }
	h>0
,\\
\vartheta^{\prime\prime}_{h,k} 
&:=
	\frac{h_1+h}{k_1+k} - \frac{h}{k} 
	=
	\frac{1}{k(k_1+k)}		
.
\end{align*}
We note that these measure the distance from $\frac{h}{k}$ to the mediants
with the neighboring Farey fractions. Using Cauchy's theorem, given a function
$F(q) = \sum_{n=0}^\infty a_n q^n$, 
we let $C$ be the circle centered at the origin of radius $\exp\left(-\frac{2\pi}{N^2}\right)$,
where $N = \lfloor\sqrt{n}\rfloor$ and find that
\begin{align*}
a_n
&=
	\frac{1}{2\pi i}\int_{C} \frac{F(q)}{q^{n+1}} dq
=
	\int_{0}^{1} F\left( \exp\left(-\tfrac{2\pi}{N^2} + 2\pi it\right) \right)
	\exp\left( \tfrac{2\pi n}{N^2} - 2\pi it \right) dt
\\
&=
	\sum_{\substack{ 0\le h< k\le N\\(h,k)=1}}	
	\exp\left(-\tfrac{2\pi ihn}{k}\right)
	\int_{-\vartheta^{\prime}_{h,k}}^{\vartheta^{\prime\prime}_{h,k}}
	F\left( \exp\left( \tfrac{2\pi i}{k} \left( h + i\left( \tfrac{k}{N^2}-ik\Phi \right) \right)\right)  \right)
	\exp\left( \tfrac{2\pi n}{k}\left( \tfrac{k}{N^2}-ik\Phi  \right)\right) d\Phi
.
\end{align*}
With $z=\frac{k}{N^2}-ik\Phi$, this becomes
\begin{align*}
a_n
&=
	\sum_{\substack{ 0\le h< k\le N\\(h,k)=1}} \exp\left(-\tfrac{2\pi ihn}{k}\right)
	\int_{-\vartheta^{\prime}_{h,k}}^{\vartheta^{\prime\prime}_{h,k}}
	F\left( \exp\left(\tfrac{2\pi i}{k}(h+iz)\right) \right) \exp\left( \tfrac{2\pi nz}{k}\right) d\Phi
.
\end{align*}
We look to apply a modular transformation with
$\frac{h+iz}{k}
=\begin{psmallmatrix}h&-\frac{1+h[-h]_k}{k}\\k&-[-h]_k\end{psmallmatrix}
\frac{[-h]_k+i/z}{k}$ and recognize the resulting integral
as representing a Bessel function 
(see equations \eqref{EqCircleMethodRankMoment4} and \eqref{EqCircleMethodRankMoment5},
as well as Proposition \ref{MainTermFromIntegrals}) and an error term.
This method applies to mock modular forms for the reason that while
they do no satisfy a modular transformation, mock modular
forms can be completed to harmonic Maass forms (which do satisfy a modular transformation),
and the part of the harmonic Maass form other than the mock modular form can
often be shown to only contribute to the error term.

We apply this method to $R2_{2\ell}(q)$ for Theorem \ref{TheoremRankMomentExpansions}
and to $R2(e^{\frac{2\pi ia}{c}};q)$ for Theorem \ref{TheoremAsymptoticForM2Rank}.
In determining the relevant transformation formula for 
$R2(e^{\frac{2\pi ia}{c}};q)$ we use \eqref{M2RankInTermsOfMu}. It turns out
we can also use \eqref{M2RankInTermsOfMu} to determine the relevant
transformation for $R2_{2\ell}(q)$ by considering the function
$\mathcal{R}2(u;q):=R2(e^{2\pi iu};q)$ and recognizing $\mathcal{R}2(u;q)$ as
\begin{align*}
\mathcal{R}2(u;q) &= \sum_{\ell=0}^\infty R2_\ell(q) \frac{(2\pi iu)^\ell}{\ell!}.
\end{align*}
The final transformation for $R2_{2\ell}(q)$ is stated in Corollary
\ref{CorollaryRankMomentTransformationsFinal}. Due to the more complicated nature
of the transformations for $R2(e^{\frac{2\pi ia}{c}};q)$, the final transformations
are stated in Propositions \ref{PropFinalBounds0Mod4}, \ref{PropFinalBounds2Mod4}
and \ref{PropFinalBoundsOdd}.

The rest of the article is organized as follows. 
In Section 3 we give the transformation formulas, bounds, and identities relevant 
to applying the circle method to $R2_{2\ell}(q)$. In Section 4 we
prove Theorem \ref{TheoremRankMomentExpansions} and discuss some calculations
to support its validity. In Section 5 we give the transformation formulas, 
bounds, and identities relevant to applying the circle method to 
$R2(e^{\frac{2\pi ia}{c}}q)$. This turns out to be more involved than the 
corresponding results for $R2_{2\ell}(q)$. Furthermore, since we use 
Theorem \ref{TheoremAsymptoticForM2Rank} to prove inequalities, we must keep track
of explicit upper bounds for the error terms. In Section 6 we prove
Theorem \ref{TheoremAsymptoticForM2Rank} and give an explicit upper bound for the
error term. In Section 7 we discuss and prove a few inequalities related to
$N2(r,m,n)$ (the number of partitions of $n$ without repeated odd parts and
with $M_2$-rank $m$). In Section 8 we give our final discussion and closing 
remarks.

\section{Identities Relevant to $N2_\ell(n)$}

We follow the development in \cite{BringmannMahlburgRhoades1}
and the culmination of this section is Corollary
\ref{CorollaryRankMomentTransformationsFinal}.
To understand
$R2_{2\ell}\big(e^{\frac{2\pi i(h+iz)}{k}}\big)$, we must determine 
modular transformations for $\mu\big(2u,\pm\frac{h+iz}{k};\frac{4(h+iz)}{k}\big)$.
For this we first investigate $S\left(u\mp\frac{h+iz}{4k}\frac{h+iz}{k}\right)$
and $\tmu\left(u, \pm\frac{h+iz}{4k}; \frac{h+iz}{k}\right)$.
Many of our proofs require lengthy but straightforward calculations. 
We omit the details when these calculations are nothing more than reducing various 
exponents and basic algebra.

%%%%%%%%%%%%%%%%%%%%%%%%%%%%%%%%%%%%%%%%%%%%%%%%%%%%%%%%%%%%%%%%%%%%%%%%%%%%%%%
%%%%%%%%%%%%%%%%%%%%%%%%%%%%%%%%%%%%%%%%%%%%%%%%%%%%%%%%%%%%%%%%%%%%%%%%%%%%%%%
%%%%%%%%%%%%%%%%%%%%%%%%%%%%%%%%%%%%%%%%%%%%%%%%%%%%%%%%%%%%%%%%%%%%%%%%%%%%%%%
\begin{proposition}\label{PropositionSTransformation}
Suppose $h$ and $k$ are relatively prime integers with $k>0$, 
$u,z\in\mathbb{C}$, and $\RE{z}>0$. Then
\begin{align*}
S\left( u \mp \tfrac{h+iz}{4k}; \tfrac{h+iz}{k} \right)
&=
	\frac{
		\exp\left(
			-\frac{\pi z}{16k}	
			+
			\frac{\pi\left(\widetilde{h} \mp 4ku\right)^2}{16kz}
			\mp
			\frac{\pi iu}{2}
		\right)	
	}{\sqrt{kz}}		
	\sum_{\ell=0}^{k-1}
	\xi^{\mp}_\ell(h,k)		
	(S-H)\left(
		\tfrac{iu}{z} \mp \tfrac{i\widetilde{h}}{4kz} + \alpha^{\pm}(\ell,k)
		 ; \tfrac{i}{kz}   
	\right)			
.
\end{align*}	
\end{proposition}
\begin{proof}
By Proposition 2.3 of \cite{BringmannFolsom1} we have that
\begin{align*}
S\left( u \mp \tfrac{h+iz}{4k}; \tfrac{h+iz}{k} \right)
&=	
	\sum_{\ell=0}^{k-1}
	\exp\left( 
		\tfrac{\pi z(2\ell - k + 1 \mp 1)(2\ell -k + 1)}{4k}					
	+
		\tfrac{-\pi i(2\ell-k+1)(2\ell h - hk + 4ku + 2k + h \mp h)}{4k}  
	\right)
	\\&\quad\times
	S\left( ku + iz\left(\ell - \tfrac{k-1}{2} \mp \tfrac{1}{4}\right) \mp \tfrac{h}{4}
		+ \ell h + \tfrac{(1-h)(k-1)}{2}
		; kh+ikz   
	\right)
.
\end{align*}
By applying (\ref{EqSEllipticProperty1}), (\ref{EqSModularProperty1}), and
(\ref{EqSModularProperty3}) we compute
\begin{align*}
&S\left( ku + iz\left(\ell - \tfrac{k-1}{2} \mp \tfrac{1}{4}\right) \mp \tfrac{h}{4}
	+ \ell h + \tfrac{(1-h)(k-1)}{2}
	; kh+ikz   
\right)
\\
&=
	(-1)^{1+\ell h + \frac{(1-h)(k-1)}{2} + \frac{h-\widetilde{h}}{4}  }
	\exp\left(
		-\pi i\left(
			\tfrac{kh}{4}
			\pm
			\tfrac{(\widetilde{h} \mp 4ku)(4\ell - 2k + 2 \mp 1)}{8k}
		\right)
		+
		\tfrac{\pi(\widetilde{h} \mp 4ku)^2}{16kz} 
		-
		\tfrac{\pi z(4\ell - 2k + 2 \mp 1)^2}{16k}
	\right)	
	\\&\quad\times
	\frac{1}{\sqrt{kz}}		
	(S-H)\left(
		\tfrac{iu}{z} \mp \tfrac{i\widetilde{h}}{4kz} + \alpha^{\pm}(\ell,k)
		 ; \tfrac{i}{kz}   
	\right)			
.
\end{align*}
Direct calculations reveal that
\begin{align*}
&(-1)^{1+\ell h + \frac{(1-h)(k-1)}{2} + \frac{h-\widetilde{h}}{4}  }
\exp\left(-\pi i\left(
	\tfrac{kh}{4}
	\pm
	\tfrac{\left(\widetilde{h} \mp 4ku\right)(4\ell - 2k + 2 \mp 1)}{8k}
	+
	\tfrac{(2\ell-k+1)(2\ell h - hk + 4ku + 2k + h \mp h)}{4k}  
\right)\right)
\\
&=
	\exp\left( \mp\tfrac{\pi iu}{2} \right)
	\xi^{\mp}_\ell(h,k)
,\\
&\exp\left( \pi z\left(
	\tfrac{(2\ell - k + 1 \mp 1)(2\ell -k + 1)}{4k}			
		-
	\tfrac{(4\ell - 2k + 2 \mp 1)^2}{16k} 
\right)\right)	
=
\exp\left( -\tfrac{\pi z}{16k} \right)
,
\end{align*}
and so the proposition holds.
\end{proof}
%%%%%%%%%%%%%%%%%%%%%%%%%%%%%%%%%%%%%%%%%%%%%%%%%%%%%%%%%%%%%%%%%%%%%%%%%%%%%%%
%%%%%%%%%%%%%%%%%%%%%%%%%%%%%%%%%%%%%%%%%%%%%%%%%%%%%%%%%%%%%%%%%%%%%%%%%%%%%%%
%%%%%%%%%%%%%%%%%%%%%%%%%%%%%%%%%%%%%%%%%%%%%%%%%%%%%%%%%%%%%%%%%%%%%%%%%%%%%%%

\begin{proposition}\label{PropositionTildeMuTransformation}
Suppose $h$ and $k$ are relatively prime integers with $k>0$, 
$u,z\in\mathbb{C}$, and $\RE{z}>0$. Then
\begin{align*}
&\tmu\left( u, \pm\tfrac{h+iz}{4k} ; \tfrac{h+iz}{k}	\right)
\\
&=
	\exp\left(
		-\tfrac{\pi z}{16k}		
		+
		\tfrac{\pi\left(\widetilde{h} \mp 4ku\right)^2}{16kz}
		\mp
		\tfrac{\pi iu}{2}
	\right)
	\xi(h,[-h]_k,k)
	\frac{1}{\sqrt{z}}
	\tmu\left( \tfrac{iu}{z}, 
		\pm\tfrac{\widetilde{h}([-h]_k+i/z)}{4k} 
		\mp\tfrac{1+h[-h]_k}{4k}
		; \tfrac{[-h]_k+i/z}{k}	
	\right)
.
\end{align*}
\end{proposition}
\begin{proof}
By (\ref{EqMuTildeModularProperty}), with 
$A=\begin{psmallmatrix}h&-\frac{1+[-h]_k}{k}\\k&-[-h]_k\end{psmallmatrix}$, we have that
\begin{align*}
\tmu\left( u, \pm\tfrac{h+iz}{4k} ; \tfrac{h+iz}{k}	\right)
&=
	\tmu\left( u, \pm\tfrac{h+iz}{4k} ; A\left(\tfrac{[-h]_k+i/z}{k}\right)	\right)
\\			
&=
	\chi\left( h, [-h]_k, k \right)^{-3}
	e^{ \frac{\pi i}{4}} 
	\exp\left(
		-\tfrac{\pi z}{16k}		
		+\tfrac{\pi(h \mp 4ku)^2}{16kz}		
		+\tfrac{\pi i(h \mp 4ku)}{8k}		
	\right)
	\frac{1}{\sqrt{z}}
	\\&\quad\times
	\tmu\left( \tfrac{iu}{z}, 
		\pm\tfrac{(h-\widetilde{h})}{4}\tfrac{([-h]_k+i/z)}{k} 
		\pm\tfrac{\widetilde{h}([-h]_k+i/z)}{4k} 
		\mp\tfrac{1+h[-h]_k}{4k}
		; \tfrac{[-h]_k+i/z}{k}	\right)
.
\end{align*}
Using (\ref{EqMuTildeEllipticProperty}) we find that
\begin{align*}
&
\tmu\left( \tfrac{iu}{z}, 
	\pm\tfrac{\left(h-\widetilde{h}\right)}{4}\tfrac{([-h]_k+i/z)}{k} 
	\pm\tfrac{\widetilde{h}([-h]_k+i/z)}{4k} 
	\mp\tfrac{1+h[-h]_k}{4k}
	; \tfrac{[-h]_k+i/z}{k}	\right)
\\
&=	
	(-1)^{\frac{h-\widetilde{h}}{4}}	
	\exp\left( 
		-\tfrac{\pi i}{16k}\left(
			[-h]_k\big(h-\widetilde{h}\big)^2	
			+
			2\big(h-\widetilde{h}\big)	
		\right)
		-
		\tfrac{\pi}{16kz}\left(
			\big(h-\widetilde{h}\big)^2	
			+
			2\big(h-\widetilde{h}\big)\big(\widetilde{h}\mp 4ku\big)	
		\right)
	\right)	
	\\&\quad\times
	\tmu\left( \tfrac{iu}{z}, 
		\pm\tfrac{\widetilde{h}([-h]_k+i/z)}{4k} 
		\mp\tfrac{1+h[-h]_k}{4k}
		; \tfrac{[-h]_k+i/z}{k}	
	\right)
.
\end{align*}
Noting
\begin{align*}
\chi\left( h, [-h]_k, k \right)^{-3}
e^{ \frac{\pi i}{4} }
(-1)^{\frac{h-\widetilde{h}}{4}}	
\exp\left(
	\tfrac{\pi i(h \mp 4ku)}{8k}		
	 -
	 \tfrac{\pi i}{16k}\left(
		[-h]_k\big(h-\widetilde{h}\big)^2	
		+
		2\big(h-\widetilde{h}\big)	
	\right)
\right)
&=
	\exp\left( \mp\tfrac{\pi iu}{2} \right)
	\xi(h,[-h]_k,k)
,\\
\exp\left(
	\tfrac{\pi(h \mp 4ku)^2}{16kz}		
	-\tfrac{\pi}{16kz}\left(
		\big(h-\widetilde{h}\big)^2	
		+
		2\big(h-\widetilde{h}\big)\big(\widetilde{h}\mp 4ku\big)	
 	\right)
\right)	
&=
	\exp\left( \tfrac{\pi\left(\widetilde{h} \mp 4ku\right)^2}{16kz} \right)
,
\end{align*}
we see that the proposition follows.
\end{proof}
%%%%%%%%%%%%%%%%%%%%%%%%%%%%%%%%%%%%%%%%%%%%%%%%%%%%%%%%%%%%%%%%%%%%%%%%%%%%%%%
%%%%%%%%%%%%%%%%%%%%%%%%%%%%%%%%%%%%%%%%%%%%%%%%%%%%%%%%%%%%%%%%%%%%%%%%%%%%%%%
%%%%%%%%%%%%%%%%%%%%%%%%%%%%%%%%%%%%%%%%%%%%%%%%%%%%%%%%%%%%%%%%%%%%%%%%%%%%%%%

\begin{proposition}\label{PropositionMuModular}
Suppose $h$ and $k$ are relatively prime integers with $k>0$, 
$u,z\in\mathbb{C}$, and $\RE{z}>0$. If $h\not\equiv 2\pmod{4}$, then
\begin{align*}
&
\mu\left(u, \pm \tfrac{h+iz}{4k}; \tfrac{h+iz}{k}\right)	
\\
&=	
	\exp\left(
		-\tfrac{\pi z}{16k}		
		+
		\tfrac{\pi\left(\widetilde{h} \mp 4ku\right)^2}{16kz}
		\mp
		\tfrac{\pi iu}{2}
	\right)
	\frac{1}{\sqrt{z}}
	\Bigg(
		\xi(h,[-h]_k,k)
		\mu\left( \tfrac{iu}{z}, \pm\tfrac{\widetilde{h}([-h]_k+i/z)}{4k} 
			\mp\tfrac{1+h[-h]_k}{4k}; \tfrac{[-h]_k+i/z}{k}	
		\right)
		\\&\quad
		+
		\frac{i}{2\sqrt{k}}
		\sum_{\ell=0}^{k-1}
		\xi^{\mp}_\ell(h,k)		
		H\left(
			\tfrac{iu}{z} \mp \tfrac{i\widetilde{h}}{4kz} + \alpha^{\pm}(\ell,k)
		 	; \tfrac{i}{kz}   
		\right)			
	\Bigg)
.
\end{align*}
If $h\equiv 2\pmod{4}$, then
\begin{align*}
&
\mu\left(u, \pm \tfrac{h+iz}{4k}; \tfrac{h+iz}{k}\right)	
\\
&=	
	\exp\left(
		-\tfrac{\pi z}{16k}		
		+
		\tfrac{\pi\left(\widetilde{h} \mp 4ku\right)^2}{16kz}
		\mp
		\tfrac{\pi iu}{2}
	\right)
	\frac{1}{\sqrt{z}}
	\Bigg(
		\xi(h,[-h]_k,k)
		\mu\left( \tfrac{iu}{z}, \pm\tfrac{\widetilde{h}([-h]_k+i/z)}{4k} 
			\mp\tfrac{1+h[-h]_k}{4k}; \tfrac{[-h]_k+i/z}{k}	
		\right)
		\\&\quad
		+
		\frac{i}{2\sqrt{k}}
		\sum_{\ell=0}^{k-1}
		\xi^{\mp}_\ell(h,k)		
		H\left(
			\tfrac{iu}{z} \mp \tfrac{i\widetilde{h}}{4kz} + \alpha^{\pm}(\ell,k)
		 	; \tfrac{i}{kz}   
		\right)		
		\\&\quad	
		+
		\frac{i}{2}
		\xi(h,[-h]_k,k)
		\exp\left( 
			-\tfrac{\pi}{4kz} \pm \tfrac{\pi u}{z}				 
			+\tfrac{\pi i\left(\wt{h}[-h]_k - h[-h]_k - [-h]_k -1\right)}{4k} 
		\right)
		\\&\quad			
		-
		\frac{i}{\sqrt{2k}}
		\sum_{\ell=0}^{k-1}
		\xi^{\mp}_\ell(h,k)		
		\exp\left( -\tfrac{\pi}{4kz} \pm \tfrac{\pi u}{z} \mp \pi i\alpha^{\pm}(\ell,k) \right)
	\Bigg)		
.
\end{align*}
\end{proposition}
\begin{proof}
In both cases the left and right sides are
meromorphic functions of $z$ and $u$ for $\RE{z}>0$, so it suffices to prove these
identities in the case that $\frac{iu}{z}\in\mathbb{R}$.
By definition,
\begin{align*}	
	\mu\left(u, \pm \tfrac{h+iz}{4k}; \tfrac{h+iz}{k}\right)	
	&=	
		\tmu\left(u, \pm \tfrac{h+iz}{4k}; \tfrac{h+iz}{k}\right)	
		-
		\tfrac{i}{2}S\left(u \mp \tfrac{h+iz}{4k}; \tfrac{h+iz}{k}\right)
	,
\end{align*}	
and so by Propositions \ref{PropositionSTransformation} and 
\ref{PropositionTildeMuTransformation} we deduce that
\begin{align}\label{EqWhyTheSComboIsHolomorphic}
&
\mu\left(u, \pm \tfrac{h+iz}{4k}; \tfrac{h+iz}{k}\right)	
\nonumber\\
&=	
	\exp\left(
		-\tfrac{\pi z}{16k}		
		+
		\tfrac{\pi\left(\widetilde{h} \mp 4ku\right)^2}{16kz}
		\mp
		\tfrac{\pi iu}{2}
	\right)
	\frac{1}{\sqrt{z}}
	\Bigg(
		\xi(h,[-h]_k,k)
		\mu\left( \tfrac{iu}{z}, \pm\tfrac{\widetilde{h}([-h]_k+i/z)}{4k} 
			\mp\tfrac{1+h[-h]_k}{4k}; \tfrac{[-h]_k+i/z}{k}	
		\right)
		\nonumber\\&\quad
		+
		\frac{i}{2}
		\xi(h,[-h]_k,k)
		S\left( \tfrac{iu}{z} \mp\tfrac{\widetilde{h}([-h]_k+i/z)}{4k} 
			\pm\tfrac{1+h[-h]_k}{4k}; \tfrac{[-h]_k+i/z}{k}	
		\right)
		\nonumber\\&\quad	
		-
		\frac{i}{2\sqrt{k}}
		\sum_{\ell=0}^{k-1}
		\xi^{\mp}_\ell(h,k)		
		S\left(
			\tfrac{iu}{z} \mp \tfrac{i\widetilde{h}}{4kz} + \alpha^{\pm}(\ell,k)
		 	; \tfrac{i}{kz}   
		\right)		
		+
		\frac{i}{2\sqrt{k}}
		\sum_{\ell=0}^{k-1}
		\xi^{\mp}_\ell(h,k)		
		H\left(
			\tfrac{iu}{z} \mp \tfrac{i\widetilde{h}}{4kz} + \alpha^{\pm}(\ell,k)
		 	; \tfrac{i}{kz}   
		\right)			
	\Bigg)
.
\end{align}
Next we verify the cancellations between the $S(w;\tau)$ terms. For this
we follow the method used by Bringmann and Mahlburg in \cite{BringmannMahlburg1}
and by Bringmann, Mahlburg, and Rhoades in 
\cite{BringmannMahlburgRhoades1}. The key point is that if a function of the form
\begin{align*}
	\sum_{n\in\mathbb{Q}\backslash\{0\}}
	a(n)\Gamma\left(\tfrac{1}{2};4\pi|n|\IM{\tau} \right)q^{-n}
,
\end{align*}
where $\Gamma(a;x)$ is the incomplete Gamma function, is a holomorphic 
function of $\tau$,
then in fact the function is identically zero. One can quickly deduce this is the case
by using that 
$\frac{\partial}{\partial\overline{\tau}}$ annihilates any holomorphic function of $\tau$.

To begin we note that
\begin{align*}
\beta(x)
&=
	\int_{x}^{\infty} t^{-\frac{1}{2}}e^{-\pi t} dt
=
	\frac{1}{\sqrt{\pi}}
	\Gamma\left(\tfrac{1}{2}; \pi x\right)
.
\end{align*}
We rewrite $S(w;\tau)$ as
\begin{align*}
&S(w;\tau)
\\
&=
	\sum_{n\in \frac{1}{2}+\mathbb{Z}}	
	\left(
		\sgn{n} 
		- 
		\sgn{n+\tfrac{\IM{w}}{\IM{\tau}}}
		+
		\sgn{n+\tfrac{\IM{w}}{\IM{\tau}}}
		\beta\left( \left(n+\tfrac{\IM{w}}{\IM{\tau}}\right)^2 2\IM{\tau} \right)		
	\right)	
	(-1)^{n-\frac{1}{2}}
	q^{-\frac12n^2}
	e^{-2\pi inw}	
.
\end{align*}
For $a$ such that $-\frac{1}{2}<a<\frac{1}{2}$, we have
$\sgn{n}=\sgn{n+a}$ for all $n\in\frac{1}{2}+\mathbb{Z}$. Thus for
$a$, $b\in\mathbb{R}$ with $|a|<\frac{1}{2}$,
\begin{align}\label{EqSNice1}
S(a\tau-b;\tau)
&=
	\frac{1}{\sqrt{\pi}}
	\sum_{n\in \frac{1}{2}+\mathbb{Z}}	
	\sgn{n}
	(-1)^{n-\frac{1}{2}}
	q^{-\frac12n^2}
	e^{- 2\pi in(a\tau-b) }
	\Gamma\left(\tfrac{1}{2}; 2\pi(n+a)^2 \IM{\tau}\right)
.
\end{align}	

In the case when $h\not\equiv 2\pmod{4}$, we claim
\begin{align*}
\xi(h,[-h]_k,k)
S\left( \tfrac{iu}{z} \mp\tfrac{\widetilde{h}([-h]_k+i/z)}{4k} 
	\pm\tfrac{1+h[-h]_k}{4k}; \tfrac{[-h]_k+i/z}{k}	
\right)
-
\frac{1}{\sqrt{k}}
\sum_{\ell=0}^{k-1}
\xi^{\mp}_\ell(h,k)		
S\left(
	\tfrac{iu}{z} \mp \tfrac{i\widetilde{h}}{4kz} + \alpha^{\pm}(\ell,k)
 	; \tfrac{i}{kz}   
\right)		
&=
0.
\end{align*}
We set $\tau = \frac{i}{kz}$ and $w=\frac{iu}{z}$.
For $w,\alpha,\beta\in\mathbb{R}$ and $h\not\equiv 2\pmod{4}$, 
we have by \eqref{EqSNice1} that
\begin{align*}
&S\left( 
		\tfrac{iu}{z} \mp\tfrac{\widetilde{h}i}{4kz} + \alpha; \tfrac{i}{kz}
		+\beta	
	\right)
\\
&=
	\frac{1}{\sqrt{\pi}}
	\sum_{n\in \frac{1}{2}+\mathbb{Z}}	
	\sgn{n}
	(-1)^{n-\frac{1}{2}}
	\exp\left(
		-\pi in^2\left(\tau+\beta\right) 
		- 2\pi in\left(w \mp \tfrac{\wt{h}\tau}{4} + \alpha\right)
	\right)		
	\Gamma\left(
		\tfrac{1}{2}; 
		2\pi\left(n \mp \tfrac{\widetilde{h}}{4}\right)^2 \IM{\tau}
	\right)
.
\end{align*}
This implies
\begin{multline*}	
\exp\left(-\tfrac{\pi i\wt{h}^2\tau}{16}\right)
\Bigg(
	\xi(h,[-h]_k,k)
	S\left( \tfrac{iu}{z} \mp\tfrac{\widetilde{h}([-h]_k+i/z)}{4k} 
		\pm\tfrac{1+h[-h]_k}{4k}; \tfrac{[-h]_k+i/z}{k}	
	\right)
	\\
	-
	\frac{1}{\sqrt{k}}
	\sum_{\ell=0}^{k-1}
	\xi^{\mp}_\ell(h,k)		
	S\left(
		\tfrac{iu}{z} \mp \tfrac{i\widetilde{h}}{4kz} + \alpha^{\pm}(\ell,k)
 		; \tfrac{i}{kz}   
	\right)		
\Bigg)
\end{multline*}
is a holomorphic function of $\tau$,
because \eqref{EqWhyTheSComboIsHolomorphic} shows this difference is expressible in
terms of $\mu$ and $H$ functions,
with an expansion of the form
\begin{align*}
	\sum_{n\in\mathbb{Q}\backslash\{0\}}
	a(n)\Gamma\left(\tfrac{1}{2};4\pi|n| \IM{\tau} \right)q^{-n}
	,
\end{align*}\sloppy	
and so the function is identically zero. The assumption that $\frac{iu}{z}\in\mathbb{R}$
is so that the $\frac{iu}{z}$ term in 
$S\big( \frac{iu}{z} \mp\frac{\widetilde{h}i}{4kz} + \alpha; \tfrac{i}{kz} + \beta\big)$
contributes only to the $b$ term in $S(a\tau-b;\tau)$ (and not to the $a$ term).
	
\fussy	
In the case when $a=\mp \frac{1}{2}$, we have
$\sgn{n}=\sgn{n+a}$ for all $n\in\frac{1}{2}+\mathbb{Z}$ except for
$n=-a$. As such we have for $b\in\mathbb{R}$ and $a=\mp\frac12$ that
\begin{align}\label{EqSNice2}
S(a\tau-b;\tau)
&=
	q^{\frac18}e^{-2\pi iab}
	\nonumber\\&\quad
	+	
	\frac{1}{\sqrt{\pi}}
	\sum_{\substack{n\in \frac{1}{2}+\mathbb{Z},\\ n\not=-a}}	
	\sgn{n}
	(-1)^{n-\frac{1}{2}}
	q^{-\frac12n^2}
	e^{- 2\pi in(a\tau-b)}
	\Gamma\left(\tfrac{1}{2}; 2\pi(n+a)^2 \IM{\tau}\right)
.
\end{align}	
By a similar argument we find the $S(w;\tau)$ terms almost fully cancel when
$h\equiv2\pmod{4}$. However, there are terms that remain due to the 
$q^{\frac18}e^{-2\pi iab}$ in \eqref{EqSNice2}. Upon calculating these
terms, we find that they are as stated in the proposition.
\end{proof}

In Proposition \ref{PropositionMuModular} we found that in some cases
various $S(u;\tau)$ terms do not simplify to zero.
That $S(u;\tau)$ may at times contribute to the holomorphic part of $\tmu(u,v;\tau)$
is entirely expected. In particular one can verify that 
$S\left(\frac{\tau}{2};\tau\right)=q^{\frac18}$.

%%%%%%%%%%%%%%%%%%%%%%%%%%%%%%%%%%%%%%%%%%%%%%%%%%%%%%%%%%%%%%%%%%%%%%%%%%%%%%%
%%%%%%%%%%%%%%%%%%%%%%%%%%%%%%%%%%%%%%%%%%%%%%%%%%%%%%%%%%%%%%%%%%%%%%%%%%%%%%%
%%%%%%%%%%%%%%%%%%%%%%%%%%%%%%%%%%%%%%%%%%%%%%%%%%%%%%%%%%%%%%%%%%%%%%%%%%%%%%%

\begin{lemma}\label{LemmaR2Transformations}
Suppose $h$ and $k$ are relatively prime integers with $k>0$,
$u,z\in\mathbb{C}$, and $\RE{z}>0$.
\begin{enumerate}[leftmargin=*]
\item For $k\equiv 0\pmod{4}$, we have that
\begin{align*}
R2\left(e^{2\pi iu}; e^{\frac{2\pi i(h+iz)}{k}}\right)
&=
	\frac{2\sin(\pi u)}{\sqrt{z}}
	\exp\left( 
		-\tfrac{\pi z}{4k} 
		+ \tfrac{\pi ku^2}{z} 
		+ \tfrac{\pi}{4kz}  
	\right)
	\\&\quad\times
	\sum_{\pm}\mp 
	\exp\left( \mp\tfrac{\pi u\wt{h}}{z} \right)
	\Bigg(
		\xi\left(h, [-h]_{4k}, \tfrac{k}{4}\right)
		\mu\left( 
			\tfrac{2iu}{z}, \pm\tfrac{\wt{h} ( [-h]_{4k}+i/z ) }{k} 
			; \tfrac{4( [-h]_{4k} + i/z )}{k} 
		\right)
		\\&\quad			
		+
		\frac{i}{\sqrt{k}}\sum_{\ell=0}^{\frac{k}{4}-1}
		\xi_{\ell}^{\mp}\left(h, \tfrac{k}{4}\right)
		H\left( 
			\tfrac{2iu}{z} \mp\tfrac{i\wt{h}}{kz} + \alpha^{\pm}\left(\ell,\tfrac{k}{4}\right)
			; \tfrac{4i}{kz}  
		\right)
	\Bigg)
.
\end{align*}

\item For $k\equiv 2\pmod{4}$, we have that
\begin{align*}
R2\left(e^{2\pi iu}; e^{\frac{2\pi i(h+iz)}{k}}\right)
&=
	\frac{\sqrt{2}\sin(\pi u)}{\sqrt{z}}
	\exp\left( -\tfrac{\pi z}{4k} + \tfrac{\pi ku^2}{z} + \tfrac{\pi}{4kz}  \right)
	\sum_{\pm}\mp 
	\exp\left( \mp\tfrac{\pi u}{z} \right)
	\\&\quad\times	
	\Bigg(
		\xi\left(2h, [-2h]_{\frac{k}{2}}, \tfrac{k}{2}\right)
		\mu\left( 
			\tfrac{iu}{z}, \pm\tfrac{2[-2h]_{\frac{k}{2}}+i/z}{2k} 
			\mp \tfrac{1+2h[-2h]_{\frac{k}{2}}}{2k}
			; \tfrac{ 2[-2h]_{\frac{k}{2}} + i/z}{k}
		\right)
		\\&\quad			
		+
		\frac{i}{\sqrt{2k}}\sum_{\ell=0}^{\frac{k}{2}-1}
		\xi_{\ell}^{\mp}\left(2h, \tfrac{k}{2}\right)
		H\left( 
			\tfrac{iu}{z} \mp \tfrac{i}{2kz} + \alpha^{\pm}\left(\ell, \tfrac{k}{2}\right); 
			\tfrac{i}{kz}  
		\right)
	\Bigg)
.
\end{align*}

\item For $k\equiv 1 \pmod{2}$, we have that
\begin{align*}
R2\left(e^{2\pi iu}; e^{\frac{2\pi i(h+iz)}{k}}\right)
&=
	\frac{\sin(\pi u)}{\sqrt{z}}
	\exp\left( -\tfrac{\pi z}{4k} + \tfrac{\pi ku^2}{z}  \right)
	\sum_{\pm}\mp 
	\Bigg(
		\xi\left(4h,\tfrac{[-h]_k}{4},k\right)
		\mu\left( 
			\tfrac{iu}{2z}, \mp\tfrac{1+h[-h]_k}{4k} 
			; \tfrac{[-h]_k + i/z}{4k}  
		\right)
		\\&\quad			
		+
		\frac{i}{2\sqrt{k}}\sum_{\ell=0}^{k-1}
		\xi_{\ell}^{\mp}(4h,k)
		H\left( \tfrac{iu}{2z} + \alpha^{\pm}(\ell,k); \tfrac{i}{4kz}  \right)
	\Bigg)
.
\end{align*}
\end{enumerate}
\end{lemma}
\begin{proof}
The proofs of all three cases follow from similar calculations and so we
only give the proof for the case when $k\equiv2\pmod{4}$. This case is actually
the most complicated, as it requires verifying the cancellation of additional
terms that are not present in the other two cases.
By \eqref{M2RankInTermsOfMu} we find
\begin{align*}
R2\left(e^{2\pi iu}; e^{\frac{2\pi i(h+iz)}{k}}\right)
&=
	i(1-e^{2\pi iu})\left(
		e^{-2\pi iu}\mu\left( 2u, -\tfrac{h+iz}{k}; \tfrac{4(h+iz)}{k} \right)
		-
		\mu\left( 2u, \tfrac{h+iz}{k}; \tfrac{4(h+iz)}{k} \right)
	\right)
\\	
&=
	i(1-e^{2\pi iu})e^{-\pi iu}
	\sum_{\pm} \mp e^{\pm \pi iu}
	\mu\left( 2u, \pm \tfrac{h+iz}{k}; \tfrac{4(h+iz)}{k}   \right)
\\	
&=
	2\sin(\pi u)
	\sum_{\pm} \mp e^{\pm \pi iu}
	\mu\left( 2u, \pm \tfrac{h+iz}{k}; \tfrac{4(h+iz)}{k}   \right)
.
\end{align*}

In the case that $k\equiv 2\pmod{4}$ we note that
$\gcd\left(2h,\frac{k}{2}\right)=1$ and $\wt{2h}=2$.
Applying Proposition \ref{PropositionMuModular} with
$k\mapsto \frac{k}{2}$, $h\mapsto 2h$, $u\mapsto 2u$,
and $z\mapsto 2z$ gives that
\begin{align*}
R2\left(e^{2\pi iu};e^{\frac{2\pi i(h+iz)}{k}}\right)
&=	
	\frac{\sqrt{2}\sin(\pi u)}{\sqrt{z}}
	\exp\left( -\tfrac{\pi z}{4k} + \tfrac{\pi}{4kz} + \tfrac{\pi ku^2}{z} \right)
	\sum_{\pm} \mp 
	\exp\left(	\mp \tfrac{\pi u}{z} \right)
	\\&\quad\times
	\Bigg(
		\xi\left(2h,[-2h]_{\frac{k}{2}},\tfrac{k}{2}\right)
		\mu\left( \tfrac{iu}{z}, \pm\tfrac{2[-2h]_{\frac{k}{2}}+i/z}{2k} 
			\mp\tfrac{1+2h[-2h]_{\frac{k}{2}}}{2k}
			; 
			\tfrac{2[-2h]_{\frac{k}{2}}+i/z}{k}	
		\right)
		\\&\quad
		+
		\frac{i}{\sqrt{2k}}
		\sum_{\ell=0}^{\frac{k}{2}-1}
		\xi^{\mp}_\ell\left(2h,\tfrac{k}{2}\right)		
		H\left(
			\tfrac{iu}{z} \mp \tfrac{i}{2kz} + \alpha^{\pm}\left(\ell,\tfrac{k}{2}\right)
		 	; \tfrac{i}{kz}   
		\right)			
		\\&\quad			
 		+
 		\frac{i}{2}
		\xi\left(2h,[-2h]_{\frac{k}{2}},\tfrac{k}{2}\right)
		\exp\bigg( 
			-\tfrac{\pi}{4kz}\pm\tfrac{\pi u}{z}				 
			+\tfrac{\pi i\big( 2[-2h]_{\frac{k}{2}} - 2h[-2h]_{\frac{k}{2}} - [-2h]_{\frac{k}{2}} -1 \big)}{2k} 
		\bigg)
		\\&\quad
		-
		\frac{i}{\sqrt{k}}
		\sum_{\ell=0}^{\frac{k}{2}-1}
		\xi^{\mp}_\ell\left(2h,\tfrac{k}{2}\right)		
		\exp\left(
			-\tfrac{\pi}{4kz} \pm \tfrac{\pi u}{z} 
			\mp \pi i\alpha^{\pm}\left(\ell,\tfrac{k}{2}\right)
		\right)
	\Bigg)		
.
\end{align*}
To finish the proof of $(2)$ we must verify
\begin{multline*}
0
=
\sum_{\pm} \mp 
\exp\left(	\mp \tfrac{\pi u}{z} \right)
\Bigg(
	\frac{i}{2}
	\xi\left(2h,[-2h]_{\frac{k}{2}},\tfrac{k}{2}\right)
	\exp\bigg( 
		-\tfrac{\pi}{4kz} \pm \tfrac{\pi u}{z}				 
		+\tfrac{\pi i\big( 2[-2h]_{\frac{k}{2}} - 2h[-2h]_{\frac{k}{2}} - [-2h]_{\frac{k}{2}} -1 \big)}{2k} 
	\bigg)
	\\
	-
	\frac{i}{\sqrt{k}}
	\sum_{\ell=0}^{\frac{k}{2}-1}
	\xi^{\mp}_\ell\left(2h,\tfrac{k}{2}\right)		
	\exp\left(
		-\tfrac{\pi}{4kz} \pm \tfrac{\pi u}{z} 
		\mp \pi i\alpha^{\pm}\left(\ell,\tfrac{k}{2}\right)
	\right)
\Bigg)
.
\end{multline*}
To begin we note
\begin{align*}
&
\sum_{\pm} \mp 
\exp\left(	\mp \tfrac{\pi u}{z} \right)
\Bigg(
	\frac{i}{2}
	\xi\left(2h,[-2h]_{\frac{k}{2}},\tfrac{k}{2}\right)
	\exp\bigg( 
		-\tfrac{\pi}{4kz}\pm\tfrac{\pi u}{z}				 
		+\tfrac{\pi i\big( 2[-2h]_{\frac{k}{2}} - 2h[-2h]_{\frac{k}{2}} - [-2h]_{\frac{k}{2}} -1\big)}{2k} 
	\bigg)
	\\&\qquad\qquad\qquad\qquad			
	-
	\frac{i}{\sqrt{k}}
	\sum_{\ell=0}^{\frac{k}{2}-1}
	\xi^{\mp}_\ell\left(2h,\tfrac{k}{2}\right)		
	\exp\left(
		-\tfrac{\pi}{4kz} \pm \tfrac{\pi u}{z} 
		\mp \pi i\alpha^{\pm}\left(\ell,\tfrac{k}{2}\right)
	\right)
\Bigg)
\\
&=
	\frac{i}{\sqrt{k}}
	\exp\left( -\tfrac{\pi}{4kz} \right)
	\sum_{\pm} \pm 
	\sum_{\ell=0}^{\frac{k}{2}-1}
	\xi^{\mp}_\ell\left(2h,\tfrac{k}{2}\right)		
	\exp\left(
		\mp \pi i\alpha^{\pm}\left(\ell,\tfrac{k}{2}\right)
	\right)
.
\end{align*}
A direct calculation reveals that
\begin{align*}
\xi^-_\ell \left(2h,\tfrac{k}{2}\right) \exp\left(-\pi i\alpha^+\left(\ell,\tfrac{k}{2}\right)\right)
&=
	\xi^+_{\frac{k}{2}-\ell-1}\left(2h,\tfrac{k}{2}\right)
	\exp\left( \pi i\alpha^-\left(\tfrac{k}{2}-\ell-1,\tfrac{k}{2}\right) \right)
	.
\end{align*}
From this it follows that
\begin{align*}
\sum_{\pm} \pm 
\sum_{\ell=0}^{\frac{k}{2}-1}
\xi^{\mp}_\ell\left(2h,\tfrac{k}{2}\right)		
\exp\left(
	\mp \pi i\alpha^{\pm}\left(\ell,\tfrac{k}{2}\right)
\right)
&=
	\sum_{\ell=0}^{\frac{k}{2}-1}
	\xi^{-}_\ell\left(2h,\tfrac{k}{2}\right)		
	\exp\left( - \pi i\alpha^{+}\left(\ell,\tfrac{k}{2} \right) \right)
	\\&\quad
	-
	\sum_{\ell=0}^{\frac{k}{2}-1}
	\xi^{+}_{\frac{k}{2}-\ell-1}\left(2h,\tfrac{k}{2}\right)		
	\exp\left( \pi i\alpha^{-}\left(\tfrac{k}{2}-\ell-1,\tfrac{k}{2}\right) \right)
=0.
\end{align*}
This establishes $(2)$.
\end{proof}
%%%%%%%%%%%%%%%%%%%%%%%%%%%%%%%%%%%%%%%%%%%%%%%%%%%%%%%%%%%%%%%%%%%%%%%%%%%%%%%
%%%%%%%%%%%%%%%%%%%%%%%%%%%%%%%%%%%%%%%%%%%%%%%%%%%%%%%%%%%%%%%%%%%%%%%%%%%%%%%
%%%%%%%%%%%%%%%%%%%%%%%%%%%%%%%%%%%%%%%%%%%%%%%%%%%%%%%%%%%%%%%%%%%%%%%%%%%%%%%

Next we need bounds relevant to the $H$ terms appearing in Lemma 
\ref{LemmaR2Transformations}.
For $h$ and $k$ integers with $k>0$, 
$-\frac{1}{2}<\alpha<\frac{1}{2}$,
and $u,z\in\mathbb{C}$ with
$\RE{z}>0$ we define the functions
\begin{align*}
H_{\pm,k,h,\alpha}(u;z)
&:=
H\left(
	\tfrac{iu}{z} + \alpha \mp \tfrac{\wt{h}i}{4kz} ; \tfrac{i}{kz}
\right)
,&
H^{(\ell)}_{\pm,k,h,\alpha}(u;z)
&:=
	\left( \frac{\partial}{\partial u} \right)^{\ell}	
	H_{\pm,k,h,\alpha}(u;z)
.
\end{align*}

\begin{lemma}\label{LemmaMordellIntegralBounds}
Suppose $h$ and $k$ are integers with $k>0$,
$-\frac{1}{2}<\alpha<\frac{1}{2}$, and $\RE{z}>0$.
Then
\begin{align*}
H^{(\ell)}_{\pm,k,h,\alpha}(0;z)	
&\ll_{\ell}
|z|^{-\ell} \exp\left( -\tfrac{\pi \wt{h}^2}{16k}\RE{\tfrac{1}{z}} \right)
.
\end{align*}	
\end{lemma}
\begin{proof}
We essentially use the proof of Lemma 3.4 from \cite{BringmannMahlburgRhoades1},
however we must take some care because in the case of
$\wt{h}=2$ the function $\cosh\left(\pi(w\pm\frac{\wt{h}i}{4})\right)$ has a zero at
$w=0$. However, this technicality amounts to only a small annoyance, and we only
supply the proof for this case.

From 
\begin{align*}
	H_{\pm,k,h,\alpha}(u;z)
	&=
	\int_{\mathbb{R}}
	\frac{\exp\left( -\frac{\pi x^2}{kz } -2\pi x\left(\frac{iu}{z}+\alpha \mp \frac{\wt{h}i}{4kz} \right)\right) }
	{\cosh(\pi x)}
	dx
,
\end{align*}
we see that
\begin{align*}
	H^{(\ell)}_{\pm,k,h,\alpha}(0;z)
	&=
	\left(-\frac{2\pi i}{z}\right)^\ell
	\int_{\mathbb{R}}
	\frac{x^{\ell} \exp\left( -\frac{\pi x^2}{kz} - 2\pi x\left( \alpha \mp \frac{\wt{h}i}{4kz} \right) \right) }
	{\cosh(\pi x)}
	dx
	\\
	&=
		\left(-\frac{2\pi i}{z}\right)^\ell
		\exp\left( -\tfrac{\pi\wt{h}^2}{16kz} \mp \tfrac{\pi i\wt{h}\alpha}{2} \right)
		\int_{\mathbb{R}\mp\frac{\wt{h}i}{4}}
			\frac{ 	
				\left(w\pm \frac{\wt{h}i}{4} \right)^{\ell}
				\exp\left( -\frac{\pi w^2}{kz} - 2\pi w\alpha \right) }
			{\cosh\left(\pi\left( w\pm \frac{\wt{h}i}{4} \right)\right)}
		dw
	.
\end{align*}
We wish to shift the path of integration back to $\mathbb{R}$ by use of the
residue theorem. When $\wt{h}=2$ the integrand,
\begin{align*}
g(w)
:=
\frac{ 
	\left(w\pm \frac{\wt{h}i}{4} \right)^{\ell}
	\exp\left( -\frac{\pi w^2}{kz} - 2\pi w\alpha \right) }
{\cosh\left(\pi\left( w\pm \frac{\wt{h}i}{4} \right)\right)},
\end{align*}
has a simple pole at $w=0$ with residue $-\frac{(\pm i)^{\ell+1}}{2^{\ell}\pi}$.
For $\varepsilon>0$ we let 
$\mathbb{R}_\varepsilon := \mathbb{R}- [-\varepsilon,\varepsilon]$. By the 
residue theorem we now have that
\begin{align*}
H^{(\ell)}_{\pm,k,h,\alpha}(0;z)
&=
	\lim_{\varepsilon\rightarrow 0}
	\left(-\frac{2\pi i}{z}\right)^\ell
	\exp\left( -\tfrac{\pi\wt{h}^2}{16kz} \mp \tfrac{\pi i\wt{h}\alpha}{2} \right)
	\\&\quad\times
	\left(
		\pm\frac{-i(\pm i)^{\ell+1}}{2^{\ell}}
		+
		\int_{\mathbb{R}_\varepsilon}
			\frac{
				\left(w\pm \frac{\wt{h}i}{4} \right)^{\ell}
			 	\exp\left( -\frac{\pi w^2}{kz} - 2\pi w\alpha \right) }
			{\cosh\left(\pi\left( w\pm \frac{\wt{h}i}{4} \right)\right)}
		dw
	\right)
.
\end{align*}
We obtain the stated bound for $H^{(\ell)}_{\pm,k,h,\alpha}(0;z)$,
if the integral near zero is bounded independently of
$\varepsilon$ and $z$. For this, we note
\begin{align*}
	\int_{-\frac{1}{2}}^{-\varepsilon} g(w) dw
	+
	\int_{\varepsilon}^{\frac{1}{2}} g(w) dw
	&=
		\int_{\varepsilon}^{\frac{1}{2}} g(w)+g(-w) dw
,
\end{align*}
and the latter integral is uniformly bounded because $g(w)+g(-w)$ does not have a pole at $w=0$
since the residue of $g(-w)$ is $\frac{(\pm i)^{\ell+1}}{2^{\ell}\pi}$.

\end{proof}
%%%%%%%%%%%%%%%%%%%%%%%%%%%%%%%%%%%%%%%%%%%%%%%%%%%%%%%%%%%%%%%%%%%%%%%%%%%%%%%
%%%%%%%%%%%%%%%%%%%%%%%%%%%%%%%%%%%%%%%%%%%%%%%%%%%%%%%%%%%%%%%%%%%%%%%%%%%%%%%
%%%%%%%%%%%%%%%%%%%%%%%%%%%%%%%%%%%%%%%%%%%%%%%%%%%%%%%%%%%%%%%%%%%%%%%%%%%%%%%

We now establish the main identity for
$R2\big(e^{2\pi iu}; e^{\frac{2\pi i(h+iz)}{k}}\big)$
that leads to the required transformation of $R2_{2\ell}(q)$.
As in \cite{BringmannMahlburgRhoades1}, we let
$
f_\nu(u;z)
:=
	\frac{
			\exp\left( \frac{\nu\pi u^2}{z} \right)	
		\sin(\pi u)}
	{\sinh\left(\frac{\pi u}{z}\right)}
.
$
Lemma 3.1 of \cite{BringmannMahlburgRhoades1} states
\begin{align*}
f_\nu(u;z)
&=
\sum_{r=0}^\infty \frac{(2\pi iu)^{2r}}{(2r)!}
\sum_{a+b+c=r}
\nu^a \kappa(a,b,c)z^{1-a-2c}
,
\end{align*}
where $\kappa(a,b,c)$ was defined in Section 2.
We note this expansion is valid for $|u|<|z|$.

%In the next proposition we finally require Re(1/z)>k/2. It is so
%we know the resulting tau in the proof is large enough that we don't
%have to worry about other poles of mu.
\begin{proposition}\label{PropositionMainRankTransformation}
Suppose $h$ and $k$ are relatively prime integers with $k>0$, and
$u,z\in\mathbb{C}$ with $\RE{\frac{1}{z}}>\frac{k}{2}$ and $u$
sufficiently small. If $k\equiv 0\pmod{4}$, then
\begin{align*}
R2\left(e^{2\pi iu}; e^{\frac{2\pi i(h+iz)}{k}}\right)
&=
	-\frac{i\wt{h} \exp\left( \frac{\pi}{4kz} -\frac{\pi z}{4k}\right) \xi\left(h,[-h]_{4k},\tfrac{k}{4}\right) }
		{\sqrt{z}}
	f_k(u;z)
	+	
	\sum_{\ell=0}^\infty a_\ell(z) \frac{(2\pi iu)^\ell}{\ell!}
.
\end{align*}
If $k\equiv 2\pmod{4}$, then
\begin{align*}
R2\left(e^{2\pi iu}; e^{\frac{2\pi i(h+iz)}{k}}\right)
&=
	-\frac{i\sqrt{2} 
		\exp\left( -\frac{\pi z}{4k} - \frac{\pi i}{2k}\left(1+(2h-1)[-2h]_{\frac{k}{2}}\right) \right)
		\xi\left(2h,[-2h]_{\frac{k}{2}},\tfrac{k}{2}\right)  
	}{\sqrt{z}}
	f_k(u;z)
	\\&\quad
	+	
	\sum_{\ell=0}^\infty a_\ell(z) \frac{(2\pi iu)^\ell}{\ell!}
.
\end{align*}
If $k\equiv 1\pmod{2}$, then
\begin{align*}
R2\left(e^{2\pi iu}; e^{\frac{2\pi i(h+iz)}{k}}\right)
&=
	-\frac{ 
		\exp\left( -\tfrac{\pi z}{4k} + \tfrac{\pi}{16kz} - \tfrac{\pi i [-h]_k}{16k} \right)
		\xi(4h,\tfrac{[-h]_k}{4},k)}
	{2\sin(\tfrac{\pi k}{4})\sqrt{z} }
	f_{2k}(u;2z)
	+	
	\sum_{\ell=0}^\infty a_\ell(z) \frac{(2\pi iu)^\ell}{\ell!}
.
\end{align*}
Here $a_\ell(z) \ll_\ell k^{\frac{1}{2}}|z|^{\frac{1}{2}-\ell}$ 
as $z\rightarrow 0$, with the constants depending on $\ell$ but not  $k$.
\end{proposition}
\begin{proof}
We only give the proof for $k\equiv 0\pmod{4}$. When $k\equiv0\pmod{4}$, we have $\wt{h}=\pm 1$ and so
$\vartheta(\pm\wt{h}\tau;4\tau) = \pm\wt{h}\vartheta(\tau;4\tau)$.
Letting $\tau = \frac{[-h]_{4k}+i/z}{k}$, we calculate that
\begin{align*}
	&\mu\left(
		\tfrac{2iu}{z}, 
		\pm\tfrac{\wt{h}([-h]_{4k}+i/z)}{k}	
		; \tfrac{4([-h]_{4k}+i/z)}{k}	
	\right)
	\\
	&=
		\pm \frac{\wt{h} \exp\left(-\frac{2\pi u}{z}\right) }{\vartheta(\tau;4\tau)}	
		\sum_{n=-\infty}^\infty
		\frac{(-1)^n \exp\left( 4\pi i\tau n(n+1) \pm 2\pi i\wt{h}\tau n\right) }
		{1-\exp\left(-\frac{4\pi u}{z} + 8\pi i\tau n\right) }	
	\\
	&=	
		\pm \frac{i\wt{h}}{2\sinh\left(\frac{2\pi u}{z}\right)}
		\left(e^{2\pi i\tau},e^{6\pi i\tau},e^{8\pi i\tau} ;e^{8\pi i\tau }\right)_\infty^{-1}
		\\&\quad		
		\pm
		i\wt{h} \exp\left(-\tfrac{2\pi u}{z}\right)
		\left(e^{2\pi i\tau},e^{6\pi i\tau},e^{8\pi i\tau} ;e^{8\pi i\tau }\right)_\infty^{-1}
		\sum_{n\in\mathbb{Z}\backslash\{0\}}
		\frac{(-1)^n \exp\left(4\pi i\tau n(n+1) \pm 2\pi i\wt{h}\tau n\right) }
		{1-\exp\left(-\frac{4\pi u}{z} + 8\pi i\tau n\right)}	
	.
\end{align*}
We claim
\begin{align*}
\mu\left(
	\tfrac{2iu}{z}, 
	\pm\tfrac{\wt{h}([-h]_{4k}+i/z)}{k}	
	; \tfrac{4([-h]_{4k}+i/z)}{k}	
\right)
&=
	\pm \frac{i\wt{h}}{2\sinh\left(\frac{2\pi u}{z}\right)}
	+
	\frac{1}{u}
	\sum_{\ell=0}^\infty
	a_\ell(z) \frac{(2\pi i u)^\ell}{\ell!}	
,\\
\end{align*}
where $a_\ell(z) \ll_\ell  |z|^{1-\ell} \exp\left( -\frac{2\pi}{k}\RE{\frac{1}{z}} \right) $
as $z\rightarrow 0$.
We note that if we have two functions $f(u,z)$ and $g(u,z)$ and we wish to say,
\begin{align*}
f(u,z) 
&=
	g(u,z) + u^{-N}\sum_{\ell=0}^\infty a_{\ell}(z) \frac{(2\pi iu)^\ell}{\ell!}
,
\end{align*}
where $a_\ell(z) \ll_\ell |z|^{N-\ell}\epsilon(z)$ as $z\rightarrow 0$, then 
we can prove this by
checking that $f(u,z)-g(u,z)$ is meromorphic with at worst a pole of order $N$ at $u=0$, 
apply Cauchy's theorem along a circle $c_0 z \exp\left(2\pi i\theta\right)$
where $c_0$ is chosen small enough so no other singularities are inside the circle, and verify that
on this circle $|f(u,z)-g(u,z)| \ll_\ell \epsilon(z)$.
In our case we use the circle $\frac{z}{4\pi}\exp\left(2\pi i\theta\right)$, and must establish 
bounds when $u=\frac{zw}{4\pi}$ for $|w|=1$.
We see that
\begin{align*}
	&i \wt{h} \exp\left(-\tfrac{2\pi u}{z}\right)
	\left(e^{2\pi i\tau},e^{6\pi i\tau},e^{8\pi i\tau} ;e^{8\pi i\tau }\right)_\infty^{-1}
	\sum_{n\in\mathbb{Z}\backslash\{0\}}
	\frac{(-1)^n \exp\left(4\pi i\tau n(n+1) \pm 2\pi i\wt{h}\tau n\right) }
	{1-\exp\left(-\frac{4\pi u}{z} + 8\pi i\tau n\right)}	
\end{align*}
is analytic at $u=0$, and at $u=\frac{z w}{4\pi}$ is 
$O(e^{2\pi i\tau}) = O\left(\exp\left(-\frac{2\pi}{k}\RE{\frac{1}{z}}\right)  \right)$.
Additionally,
\begin{align*}
\pm \frac{i\wt{h}}{2\sinh\left(\frac{2\pi u}{z}\right)}
		\left(e^{2\pi i\tau},e^{6\pi i\tau},e^{8\pi i\tau} ;e^{8\pi i\tau }\right)_\infty^{-1}
\mp 
\frac{i\wt{h}}{2\sinh\left(\frac{2\pi u}{z}\right)}		
&=
	\pm
	\frac{i\wt{h}}{2\sinh\left(\frac{2\pi u}{z}\right)}
	\times O(e^{2\pi i\tau})
,
\end{align*}
and so this term has at worst a simple pole at $u=0$,  and at 
$u = \frac{z w}{4\pi}$ is $O\left( e^{2\pi i\tau} \right)$.

From this we see that
\begin{align*}
&\frac{2\sin(\pi u)}{\sqrt{z}}
\exp\left(
	-\tfrac{\pi z}{4k} 
	+ \tfrac{\pi ku^2}{z} 
	+ \tfrac{\pi}{4kz}  
\right)
\sum_{\pm}\mp 
\exp\left( \mp\tfrac{\pi u\wt{h}}{z} \right)
\xi\left(h,[-h]_{4k} ,\tfrac{k}{4}\right)
\mu\left( 
	\tfrac{2iu}{z}, \pm\tfrac{\wt{h} ( [-h]_{4k}+i/z ) }{k} 
	; \tfrac{4( [-h]_{4k} + i/z )}{k} 
\right)
\\
&=
	-
	\frac{i\wt{h} \sin(\pi u)}{\sqrt{z} \sinh\left(\frac{2\pi u}{z}\right)}
	\exp\left( 
		-\tfrac{\pi z}{4k} 
		+ \tfrac{\pi ku^2}{z} 
		+ \tfrac{\pi}{4kz}  
	\right)
	\sum_{\pm}
	\exp\left( \mp\tfrac{\pi u\wt{h}}{z} \right)
	\xi\left(h,[-h]_{4k},\tfrac{k}{4}\right)
	+
	\sum_{\ell=0}^\infty
	a_\ell(z) \frac{(2\pi i u)^\ell}{\ell!}	
,
\end{align*}
where $a_\ell(z) \ll_\ell |z|^{\frac{1}{2}-\ell}\exp\left(-\frac{7\pi}{4k}\RE{\frac{1}{z}}\right)$.
We notice that
\begin{align*}
-\frac{i\wt{h} \sin(\pi u)}{\sqrt{z} \sinh\left(\frac{2\pi u}{z}\right)}
\exp\left( 
	-\tfrac{\pi z}{4k} 
	+ \tfrac{\pi ku^2}{z} 
	+ \tfrac{\pi}{4kz}  
\right)
\sum_{\pm}
	\exp\left( \mp\tfrac{\pi u\wt{h}}{z} \right) 
	\xi\left(h,\tfrac{k}{4}\right)
&=
	-\frac{i\wt{h} \exp\left(\frac{\pi}{4kz} -\frac{\pi z}{4k}\right) \xi\left(h,\tfrac{k}{4}\right) }
		{\sqrt{z}}
	f_k(u;z)
	.
\end{align*}

To handle the contribution from 
$H\left( \frac{2iu}{z} \mp\frac{i\wt{h}}{kz} + \alpha^{\pm}\left(\ell,\frac{k}{4}\right)
; \frac{4i}{kz} \right)$,
we apply Lemma \ref{LemmaMordellIntegralBounds} with
$u\mapsto 2u$ and $k\mapsto \frac{k}{4}$ to find
\begin{align*}
H\left( 
	\tfrac{2iu}{z} \mp\tfrac{i\wt{h}}{kz} + \alpha^{\pm}\left(\ell,\tfrac{k}{4}\right)
; \tfrac{4i}{kz} \right)
&=
\sum_{\ell=0}^\infty a_\ell(z) \frac{(2\pi iu)^\ell}{\ell!}
,
\end{align*}
where $a_\ell(z) \ll_\ell |z|^{-\ell} \exp\left(-\frac{\pi}{4k}\RE{\frac{1}{z}}\right)$.
Thus
\begin{align*}
	\frac{\sin(\pi u) \exp\left(-\frac{\pi z}{4k}+\frac{\pi ku^2}{z} + \frac{\pi}{4kz}\mp\frac{\pi u\wt{h}}{z}\right) }
		{\sqrt{z}}
	H\left( 
		\tfrac{2iu}{z} \mp\tfrac{i\wt{h}}{kz} + \alpha^{\pm}\left(\ell,\tfrac{k}{4}\right)
	; \tfrac{4i}{kz} \right)
	&=
	\sum_{\ell=0}^\infty a_\ell(z) \frac{(2\pi iu)^\ell}{\ell!}
,
\end{align*}
where $a_\ell(z) \ll_\ell |z|^{\frac{1}{2}-\ell}$,
and so
\begin{align*}
&\frac{2i\sin(\pi u) \exp\left(-\frac{\pi z}{4k}+\frac{\pi ku^2}{z} + \frac{\pi}{4kz}\right) }
	{\sqrt{kz}}
\sum_{\pm}\exp\left( \mp\tfrac{\pi u\wt{h}}{z} \right)
\sum_{k=0}^{\frac{k}{4}-1}
\xi^{\mp}_\ell\left(h,\tfrac{k}{4}\right)
H\left( 
	\tfrac{2iu}{z} \mp\tfrac{i\wt{h}}{kz} + \alpha^{\pm}\left(\ell,\tfrac{k}{4}\right)
; \tfrac{4i}{kz} \right)
\\
&=
\sum_{\ell=0}^\infty a_\ell(z) \frac{(2\pi iu)^\ell}{\ell!}
,
\end{align*}
where $a_\ell(z) \ll_\ell k^{\frac{1}{2}}|z|^{\frac{1}{2}-\ell}$.

Altogether, along with part (1) of Lemma \ref{LemmaR2Transformations}, 
we find that
\begin{align*}
R2\left(e^{2\pi iu}; e^{\frac{2\pi i(h+iz)}{k}}\right)
&=
	-\frac{i\wt{h} \exp\left(\frac{\pi}{4kz} -\frac{\pi z}{4k}\right) \xi\left(h,[-h]_{4k},\tfrac{k}{4}\right) }
		{\sqrt{z}}
	f_k(u;z)
	+	
	\sum_{\ell=0}^\infty a_\ell(z) \frac{(2\pi iu)^\ell}{\ell!}
,
\end{align*}
where $a_\ell(z) \ll_\ell k^{\frac{1}{2}}|z|^{\frac{1}{2}-\ell}$.
\end{proof}

The following corollary follows from isolating the coefficient of $u^{2\ell}$ in 
Proposition \ref{PropositionMainRankTransformation}.

\begin{corollary}\label{CorollaryRankMomentTransformationsFinal}
Suppose $h$ and $k$ are relatively prime integers with $k>0$
and $\RE{\frac{1}{z}}>\frac{k}{2}$.
If $k\equiv 0\pmod{4}$, then
\begin{align*}
R2_{2\ell}\left(e^{\frac{2\pi i(h+iz)}{k}}\right)
&=
	-i\wt{h} \exp\left( \tfrac{\pi}{4kz} -\tfrac{\pi z}{4k} \right) 
	\xi\left(h,[-h]_{4k},\tfrac{k}{4}\right) 
	\sum_{a+b+c=\ell}
	k^a \kappa(a,b,c) z^{\frac{1}{2}-a-2c}
	+	
	a_{2\ell}(z)
.
\end{align*}
If $k\equiv 2\pmod{4}$, then
\begin{align*}
&R2_{2\ell}\left(e^{\frac{2\pi i(h+iz)}{k}}\right)
\\
&=
	-i\sqrt{2} \exp\left( -\tfrac{\pi z}{4k} - \tfrac{\pi i}{2k}\left(1+(2h-1)[-2h]_{\frac{k}{2}}\right) \right)
	\xi\left(2h,[-2h]_{\frac{k}{2}},\tfrac{k}{2}\right)  
	\sum_{a+b+c=\ell}
	k^a \kappa(a,b,c) z^{\frac{1}{2}-a-2c}
	+	
	a_{2\ell}(z)
.
\end{align*}
If $k\equiv 1\pmod{2}$, then
\begin{align*}
R2_{2\ell}\left(e^{\frac{2\pi i(h+iz)}{k}}\right)
&=
	-
	\frac{	\exp\left( -\frac{\pi z}{4k} + \frac{\pi}{16kz} - \frac{\pi i [-h]_k}{16k} \right)
		\xi\left(4h,\frac{[-h]}{4},k\right)}
	{\sin(\tfrac{\pi k}{4})}
	\sum_{a+b+c=\ell}
	k^a \kappa(a,b,c) z^{\frac{1}{2}-a-2c}2^{-2c}
	+	
	a_{2\ell}(z)
.
\end{align*}
Here $|a_{2\ell}(z)| \ll_{\ell} k^{\frac{1}{2}}|z|^{\frac{1}{2}-2\ell}$ 
as $z\rightarrow 0$, with the constants depending on $\ell$ but not $k$.
\end{corollary}

\section{The proof of Theorem \ref{TheoremRankMomentExpansions} and calculations}

\begin{proof}[Proof of Theorem \ref{TheoremRankMomentExpansions}]
We use the circle method as explained in Section 2 with $F(q) = R2_{2\ell}(q)$,
along with Corollary \ref{CorollaryRankMomentTransformationsFinal}, to find that
\begin{align}\label{EqCircleMethodRankMoment1}
&N2_{2\ell}(n)
=
	\sum_{\substack{ 0\le h< k\le N,\\(h,k)=1}}	
		\exp\left(-\tfrac{2\pi ihn}{k}\right)
	\int_{-\vartheta^{\prime}_{h,k}}^{\vartheta^{\prime\prime}_{h,k}}
		R2_{2\ell}\left( 
			e^{\frac{2\pi i}{k}(h+iz)}
		\right)
		\exp\left(\tfrac{2\pi nz}{k}\right)
	d\Phi
\nonumber\\
&=
	-i
	\sum_{\substack{ 0\le h< k\le N,\\ k\equiv 0\pmod{4},\\ (h,k)=1}}	
		\wt{h} \exp\left(-\tfrac{2\pi ihn}{k}\right)
		\xi\left(h,[-h]_{4k},\tfrac{k}{4}\right) 
	\sum_{a+b+c=\ell}
		k^a \kappa(a,b,c) 
	\nonumber\\&\qquad\times
		\int_{-\vartheta^{\prime}_{h,k}}^{\vartheta^{\prime\prime}_{h,k}}
			\exp\left(\tfrac{\pi}{4kz} -\tfrac{\pi z}{4k} + \tfrac{2\pi nz}{k}\right)
			z^{\frac{1}{2}-a-2c}
		d\Phi
	\nonumber\\&\quad		
	-
	i\sqrt{2} 
	\sum_{\substack{ 0\le h< k\le N,\\ k\equiv 2\pmod{4},\\ (h,k)=1}}	
		\exp\left(-\tfrac{2\pi ihn}{k} - \tfrac{\pi i}{2k}\left(1+(2h-1)[-2h]_{\frac{k}{2}}\right)  \right)
		\xi\left(2h,[-2h]_{\frac{k}{2}},\tfrac{k}{2}\right)  
	\sum_{a+b+c=\ell}
		k^a \kappa(a,b,c) 
	\nonumber\\&\qquad\times
	\int_{-\vartheta^{\prime}_{h,k}}^{\vartheta^{\prime\prime}_{h,k}}
		\exp\left(-\tfrac{\pi z}{4k} + \tfrac{2\pi nz}{k}\right)
		z^{\frac{1}{2}-a-2c}
	d\Phi
	\nonumber\\&\quad
	-
	\sum_{\substack{ 0\le h< k\le N,\\ k\equiv\pm 1\pmod{4} ,\\(h,k)=1}}	
		\frac{	\exp\left( -\frac{2\pi ihn}{k} - \frac{\pi i [-h]_k}{16k}  \right)
			\xi(4h,\tfrac{[-h]}{4},k)
		}{\sin(\tfrac{\pi k}{4})}
	\sum_{a+b+c=\ell}
		k^a \kappa(a,b,c) 2^{-2c}	
	\nonumber\\&\qquad\times
		\int_{-\vartheta^{\prime}_{h,k}}^{\vartheta^{\prime\prime}_{h,k}}
			\exp\left( -\tfrac{\pi z}{4k} + \tfrac{\pi}{16kz} + \tfrac{2\pi nz}{k} \right)
			z^{\frac{1}{2}-a-2c}
		d\Phi
	\nonumber\\&\quad
	+	
	\sum_{\substack{ 0\le h< k\le N,\\(h,k)=1}}	
		\exp\left( -\tfrac{2\pi ihn}{k} \right)
	\int_{-\vartheta^{\prime}_{h,k}}^{\vartheta^{\prime\prime}_{h,k}}
		a_{k,2\ell}(z)
		\exp\left(\tfrac{2\pi nz}{k}\right)
	d\Phi
,
\end{align}
where $a_{k,2\ell}(z) \ll_\ell  k^{\frac{1}{2}}|z|^{\frac{1}{2}-2\ell}$ .
In estimating the various error terms that appear, we recall
$z=\frac{k}{N^2}-ik\Phi$, so that
$\RE{z} = \frac{k}{N^2}\ge\frac{k}{n}$ and $\frac{k}{n}\le |z|\le \frac{\sqrt{2}}{N}$.
Additionally, from well known properties of Farey fractions, we have 
$\frac{1}{2kN}\le \vartheta^{\prime}_{h,k},\vartheta^{\prime\prime}_{h,k} < \frac{1}{k(N+1)}$ 
and $\RE{\frac{1}{z}}>\frac{k}{2}$.
Our estimates of integrals are only max-length type estimates.

We first estimate the integral involving the error term $a_{k,2\ell}(z)$. Here we 
have
\begin{align*}
\sum_{\substack{ 0\le h< k\le N,\\(h,k)=1}}	
	\exp\left(-\tfrac{2\pi ihn}{k}\right)
\int_{-\vartheta^{\prime}_{h,k}}^{\vartheta^{\prime\prime}_{h,k}}
	a_{k,2\ell}(z)
	\exp\left(\tfrac{2\pi nz}{k}\right)
d\Phi
&\ll_\ell
	\sum_{\substack{0\le h< k\le N,\\ (h,k)=1}}
	\frac{1}{k(N+1)}
	\left(\frac{k}{n}\right)^{\frac{1}{2}-2\ell}
	\sqrt{k}
\\
&\ll_\ell
	n^{2\ell-1}
	\sum_{k=1}^N k^{1-2\ell}	
	.
\end{align*}
We note for $\ell\ge 2$ that the series $\sum_{k=1}^\infty k^{1-2\ell}$ 
converges, whereas for $\ell=1$ we have that
$\sum_{k=1}^N k^{-1} \ll \log(N)$. 
As such,
\begin{align}\label{EqCircleMethodRankMoment2}
\sum_{\substack{ 0\le h< k\le N,\\(h,k)=1}}	
	\exp\left(-\tfrac{2\pi ihn}{k}\right)
\int_{-\vartheta^{\prime}_{h,k}}^{\vartheta^{\prime\prime}_{h,k}}
	a_{k,2\ell}(z)
	\exp\left(\tfrac{2\pi nz}{k}\right)
d\Phi
&\ll_\ell
	n^{2\ell-1}(1+\delta_{1,\ell}\log(n))
,
\end{align}
where $\delta_{1,\ell} = 1$ if $\ell=1$ and $\delta_{1,\ell}=0$ if $\ell\not=1$. 

Next we bound the integral involving $\exp\left(-\frac{\pi z}{4k} + \frac{2\pi nz}{k}\right)$.
We let $\omega = \frac{1}{N^2} - ik\Phi$, so that $z=k\omega$, and
\begin{align*}
	\int_{-\vartheta^{\prime}_{h,k}}^{\vartheta^{\prime\prime}_{h,k}}
		\exp\left(-\tfrac{\pi z}{4k} + \tfrac{2\pi nz}{k}\right)
		z^{\frac{1}{2}-a-2c}
	d\Phi
	&=
		-ik^{-\frac{1}{2}-a-2c}
		\int_{\frac{1}{N^2}-i\vartheta^{\prime\prime}_{h,k}}^{\frac{1}{N^2}+i\vartheta^{\prime}_{h,k}}
		\omega^{\frac{1}{2}-a-2c}
			\exp\left( \tfrac{(8n-1)\pi\omega}{4} \right)
		d\omega
	.
\end{align*}
We consider two cases given by $a+2c=0$ and $a+2c\ge 1$. When 
$a+2c=0$, the above yields the bound
\begin{align}\label{EqCircleMethodRankMoment3a}
	\left|
		\int_{-\vartheta^{\prime}_{h,k}}^{\vartheta^{\prime\prime}_{h,k}}
			\exp\left( -\tfrac{\pi z}{4k} + \tfrac{2\pi nz}{k} \right)
			z^{\frac{1}{2}-a-2c}
		d\Phi
	\right|
	&\le
		\frac{2}{k^{\frac{1}{2}}(N+1)}
		\left(\frac{1}{N^4} + \frac{1}{k^2(N+1)^2} \right)^{\frac{1}{4}}	
		\exp\left(\frac{(8n-1)\pi}{4N^2}\right)
	\ll
	k^{-\frac{5}{2}}n^{-\frac{3}{4}}
	.
\end{align}
When $a+2c\ge 1$, we instead have that
\begin{align}\label{EqCircleMethodRankMoment3b}
	\left|
		\int_{-\vartheta^{\prime}_{h,k}}^{\vartheta^{\prime\prime}_{h,k}}
			\exp\left( -\tfrac{\pi z}{4k} + \tfrac{2\pi nz}{k} \right)
			z^{\frac{1}{2}-a-2c}
		d\Phi
	\right|
	&\le
		\frac{ 2k^{-\frac{1}{2}-a-2c} }{k(N+1)} 
		\left(\frac{k}{n}\right)^{\frac{1}{2}-a-2c}
		\exp\left( \frac{(8n-1)\pi}{4N^2} \right)
	\ll
		k^{-1-2a-4c}n^{a+2c-1}
	.
\end{align}

In handling the remaining two integrals, we wish to express them in terms of a
Bessel function and an error term. While we could go through the calculations, 
this is well known as a general result. A form meeting our needs can
be found in \cite[Lemma 6.1]{Waldherr1}. In particular,
\begin{align}
	\label{EqCircleMethodRankMoment4}
	\int_{-\vartheta^{\prime}_{h,k}}^{\vartheta^{\prime\prime}_{h,k}}
		\exp\left(\tfrac{\pi}{4kz} -\tfrac{\pi z}{4k} + \tfrac{2\pi nz}{k}\right)
		z^{\frac{1}{2}-a-2c}
	d\Phi
	&=
		\frac{2\pi (8n-1)^{\frac{2a+4c-3}{4}}  }{k}  
		I_{a+2c-\frac{3}{2}}\left( \tfrac{\pi}{2k}\sqrt{8n-1} \right)
		+
		O\left( 
			\frac{ n^{\left|\frac{1}{2}-a-2c\right|-\frac{1}{2}} }
			{ k^{1+\left|\frac{1}{2}-a-2c\right|}}  
		\right)
	,\\
	\label{EqCircleMethodRankMoment5}
	\int_{-\vartheta^{\prime}_{h,k}}^{\vartheta^{\prime\prime}_{h,k}}
		\exp\left(\tfrac{\pi}{16kz} -\tfrac{\pi z}{4k} + \tfrac{2\pi nz}{k}\right)
		z^{\frac{1}{2}-a-2c}
	d\Phi
	&=
		\frac{2^{a+2c-\frac{1}{2}} \pi (8n-1)^{\frac{2a+4c-3}{4}}  }{k}  
		I_{a+2c-\frac{3}{2}}\left( \tfrac{\pi}{4k}\sqrt{8n-1} \right)
		\nonumber\\&\quad
		+
		O\left( 
			\frac{ n^{\left|\frac{1}{2}-a-2c\right|-\frac{1}{2}} }
			{ k^{1+\left|\frac{1}{2}-a-2c\right|}}  
		\right)
	.
\end{align}
We note the error terms in (\ref{EqCircleMethodRankMoment4}) and
(\ref{EqCircleMethodRankMoment5})
are at least as large as the bounds in 
(\ref{EqCircleMethodRankMoment3a}) and (\ref{EqCircleMethodRankMoment3b}).

With equations (\ref{EqCircleMethodRankMoment1}) through 
(\ref{EqCircleMethodRankMoment5}), to complete the proof of Theorem 
\ref{TheoremRankMomentExpansions}, it only remains to verify that 
\begin{align*}
\sum_{\substack{ 0\le h< k\le N,\\ (h,k)=1}}	
\sum_{a+b+c=\ell}
	\frac{ n^{\left|\frac{1}{2}-a-2c\right|-\frac{1}{2}} }
	{ k^{1+\left|\frac{1}{2}-a-2c\right| - a}}  
& \ll_\ell
n^{2\ell-1}
.
\end{align*}
This bound is easily deduced by breaking up the inner sum according to
$a+2c=0$, $a+2c\not=0$ with $c<\ell$, and $a+2c\not=0$ with $c=\ell$.
\end{proof}

It is worth running some calculations, which we perform with MAPLE, to see 
these asymptotics are accurate. In Table \ref{TableForRatios} we list 
approximations of the ratio given by the 
estimate in \eqref{EqN2RankMomentAsymptoticValue} divided by the exact value of 
$N2_{2\ell}(n)$, as well as the ratio given by the sums in Theorem 
\ref{TheoremRankMomentExpansions} (with the real and imaginary parts rounded to 
the nearest integers) divided by the exact value of $N2_{2\ell}(n)$. In Table 
\ref{TableForEstimates} we list the exact value of $N2_{2\ell}(n)$ along with 
the sums in Theorem \ref{TheoremRankMomentExpansions} (with the real and 
imaginary parts rounded to the nearest integers). For these values of $n$, the 
imaginary parts all round to $0$. We include fewer values in Table 
\ref{TableForEstimates} because of the difficulty in displaying such large 
numbers. In particular, $N2_2(10000)$ is a $96$ digit integer.

\begin{longtable}[h]{llcc}
\caption{Ratios For Asymptotic Estimates}\label{TableForRatios}
\\[-2ex]
	\toprule
	$\ell$ & 		$n$ 		& Ratio for (\ref{EqN2RankMomentAsymptoticValue})
		& Ratio for Theorem \ref{TheoremRankMomentExpansions}
	\\\midrule
	1 &			10		& 1.892666 		& 1.057143
	\\\midrule
	1 &			100		& 1.170779 		& $1+2.5\times 10^{-8}$
	\\\midrule
	1 &			1000	& 1.049075		& $1+1.3\times 10^{-28}$
	\\\midrule	
	1 &			10000	& 1.015085		& $1+4.2\times 10^{-94}$
	\\\midrule
	2 &			10		& 4.999495		& 1.172507
	\\\midrule
	2 &			100		& 1.447874		& $1+7.0\times 10^{-7}$
	\\\midrule
	2 &			1000	& 1.117096		& $1+2.5\times 10^{-25}$
	\\\midrule
	2 &			10000	& 1.035043		& $1+8.0\times 10^{-91}$
	\\\midrule
	3 &			10		& 23.68219		& 1.389405
	\\\midrule
	3 &			100		& 2.082709		& 1.000007
	\\\midrule
	3 &			1000	& 1.245479		& $1+2.0\times 10^{-24}$
	\\\midrule
	3 &			10000	& 1.070652		& $1+6.0\times 10^{-88}$
	\\\bottomrule	
\end{longtable}

\begin{longtable}[h]{llll}
\caption{Exact Values For Asymptotic Estimates}\label{TableForEstimates}
\\[-2ex]
	\toprule
	$\ell$ & $n$ & 	$N2_{2\ell}(n)$ & Theorem \ref{TheoremRankMomentExpansions} Estimate
	\\\midrule
	1 &		10 		& 70									& 74
	\\\midrule	
	1 & 	100		& 447153528								& 447153539
	\\\midrule
	1 &		1000	& 362167772560345987220442602052		& 362167772560345987220442602098
	\\\midrule  	
	2 &		10 		& 742									& 870
	\\\midrule
	2 & 	100		& 101241563496							& 101241634569
	\\\midrule
	2 &		1000	& 952322772130308063286982695330572     & 952322772130308063286982719167207
	\\\midrule  	
	3 &		10 		& 9910									& 13769
	\\\midrule
	3 & 	100		& 44527325322888						& 44527640083065
	\\\midrule
	3 &		1000	& 5403854807373412384336767926688986652	
					& 5403854807373412384336778811025822044
	\\
	\bottomrule
\end{longtable}

\section{Identities relevant to $R2\left(e^{\frac{2\pi ia}{c}};q\right)$}

We must determine appropriate transformations for 
$R2\big(e^{\frac{2\pi ia}{c}};e^{\frac{2\pi i(h+iz)}{k}}\big)$. While we
can reuse some of the calculations from Section 3, many must be redone. The
reason for this is in Section 3 we worked with
$R2\big(e^{2\pi iu};e^{\frac{2\pi i(h+iz)}{k}}\big)$,
took $u$ as a variable, and exploited a series expansion at $u=0$. But in 
this section $u=\frac{a}{c}$, which is a fixed constant and not near zero.
Furthermore, we require explicit upper bounds on the error terms.
To apply the circle method we need transformations for 
$R2\big(e^{\frac{2\pi ia}{c}};e^{\frac{2\pi i(h+iz)}{k}}\big)$
that allow us to easily determine the negative powers of $q$ in the resulting 
$\mu(u,v;\tau)$.
The required transformations are stated in Propositions
\ref{PropFinalBounds0Mod4}, \ref{PropFinalBounds2Mod4}, and \ref{PropFinalBoundsOdd},
however it is a slow process to deduce these results. The calculations
of this section iteratively refine transformations and bounds until we
arrive at these propositions.

\begin{proposition}\label{PropositionMuModularForRankAlt2}
Suppose $h$ and $k$ are relatively prime integers with $k>0$,
$u\in\mathbb{R}$ with $u\not\in\mathbb{Z}$, and $\RE{z}>0$. Then
\begin{align*}
&\mu\left(	u, \pm\tfrac{h+iz}{4k}	; \tfrac{h+iz}{k}\right)
\\
&=
	(-1)^{\Floor{ku}}
	\exp\left(
		-\tfrac{\pi z}{16k} + \tfrac{\pi}{kz}\left(\Fractional{ku} \mp \tfrac{\wt{h}}{4}\right)^2 
	 	\mp \tfrac{\pi iu}{2}
	\right)
	z^{-\frac{1}{2}}
	\Bigg(
		\exp\left(	\tfrac{\pi i\Floor{ku}}{k}\left(
			-[-h]_k\Floor{ku} \pm \tfrac{1+h[-h]_k-\wt{h}[-h]_k}{2} 
		\right)\right)
		\\&\quad\times
		\xi(h,[-h]_k,k)
		\mu\left(	
			\Fractional{ku}\tfrac{[-h]_k+i/z}{k} - u[-h]_k,	
			\pm\tfrac{\wt{h}([-h]_k+i/z)}{4k} \mp\tfrac{1+h[-h]_k}{4k}; 
			\tfrac{[-h]_k+i/z}{k}	
		\right)
		\\&\quad
		+
		\frac{i}{2\sqrt{k}}
		\sum_{\ell=0}^{k-1}
		\xi_{\ell}^{\mp}(h,k)
		\exp\left(	2\pi i\Floor{ku}\alpha^{\pm}(\ell,k)	\right)
		H\left(
			\left(\Fractional{ku}\mp\tfrac{\wt{h}}{4}\right)\tfrac{i}{kz}  + \alpha^{\pm}(\ell,k) ; \tfrac{i}{kz}
		\right)
	\Bigg)
	+E_1(u,h,k,z)
,
\end{align*}
where $E_1(u,h,k,z)$ is defined as follows.
When $\Fractional{ku}\mp\tfrac{\wt{h}}{4}=-\frac{1}{2}$
\begin{align*}
E_1(u,h,k,z)
&:=
	(-1)^{\Floor{ku}}
	\exp\left( -\tfrac{\pi z}{16k} \mp \tfrac{\pi iu}{2} \right)
	z^{-\frac{1}{2}}
	\\&\quad\times
	\bigg(
		\frac{i}{2}\xi(h,[-h]_k,k)
		\exp\left(
			\tfrac{\pi i[-h]_k}{k}\left(
				-\Floor{ku}^2+\Floor{ku} \pm (2\Floor{ku}-1)\tfrac{h-\wt{h}}{4} -\tfrac{1}{4}
			\right)
			\pm \tfrac{\pi i(2\Floor{ku}-1)}{4k}
		\right)
		\\&\quad		
		-
		\frac{i}{2\sqrt{k}}\sum_{\ell=0}^{k-1}\xi_\ell^{\mp}(h,k)
		\exp\left( \pi i(2\Floor{ku}-1)\alpha^{\pm}(\ell,k) \right)
	\bigg)
,
\end{align*}
when $-\frac{1}{2}<\Fractional{ku}\mp\tfrac{\wt{h}}{4}<\frac{1}{2}$
\begin{align*}
E_1(u,h,k,z) &:= 0,
\end{align*}
when $\Fractional{ku}\mp\tfrac{\wt{h}}{4}=\frac{1}{2}$
\begin{align*}
E_1(u,h,k,z)
&:=
	(-1)^{\Floor{ku}}
	\exp\left( -\tfrac{\pi z}{16k} \mp \tfrac{\pi iu}{2} \right)
	z^{-\frac{1}{2}}
	\\&\quad\times
	\bigg(
		\frac{i}{2}\xi(h,[-h]_k,k)
		\exp\left(
			\tfrac{\pi i[-h]_k}{k}\left(
				-\Floor{ku}^2-\Floor{ku} \pm (2\Floor{ku}+1)\tfrac{h-\wt{h}}{4} -\tfrac{1}{4}
			\right)
			\pm \tfrac{\pi i(2\Floor{ku}+1)}{4k}
		\right)
		\\&\quad
		-
		\frac{i}{2\sqrt{k}}\sum_{\ell=0}^{k-1}\xi_\ell^{\mp}(h,k)
		\exp\left( \pi i(2\Floor{ku}+1)\alpha^{\pm}(\ell,k) \right)
	\bigg)
,
\end{align*}
and when $\frac{1}{2}<\Fractional{ku}\mp\tfrac{\wt{h}}{4}<\frac{3}{2}$
\begin{align*}
E_1(u,h,k,z)
&:=
	(-1)^{\Floor{ku}}
	\exp\left( -\tfrac{\pi z}{16k}  
		+ \tfrac{\pi}{kz}\left(\Fractional{ku}\mp\tfrac{\wt{h}}{4}-\tfrac{1}{2}\right)^2  
		\mp \tfrac{\pi iu}{2}
	\right)
	z^{-\frac{1}{2}}
	\\&\quad\times
	\bigg(
		i\xi(h,[-h]_k,k)
		\exp\left(
			\tfrac{\pi i[-h]_k}{k}\left(
				-\Floor{ku}^2-\Floor{ku} \pm (2\Floor{ku}+1)\tfrac{h-\wt{h}}{4} -\tfrac{1}{4}
				\right)
			\pm \tfrac{\pi i(2\Floor{ku}+1)}{4k}
		\right)
		\\&\quad
		-
		\frac{i}{\sqrt{k}}\sum_{\ell=0}^{k-1}\xi_\ell^{\mp}(h,k)
		\exp\left( \pi i (2\Floor{ku}+1)\alpha^{\pm}(\ell,k) \right)
	\bigg)
.
\end{align*}
\end{proposition}
\begin{proof}
As is the proof of Proposition \ref{PropositionMuModular},
\begin{align*}
&\mu\left(	u, \pm\tfrac{h+iz}{4k}; \tfrac{h+iz}{k} \right)
\\&=
	\exp\left(
		-\tfrac{\pi z}{16k} + \tfrac{\pi\left(\wt{h} \mp4ku \right)^2}{16kz} \mp \tfrac{\pi iu}{2}
	\right)
	z^{-\frac{1}{2}}
	\Bigg(
		\xi(h,[-h]_k,k)
		\wt{\mu}\left( \tfrac{iu}{z}, \pm\tfrac{\wt{h}([-h]_k+i/z )}{4k} \mp\tfrac{1+h[-h]_k}{4k} ; \tfrac{[-h]_k+i/z}{k}   \right)
		\\&\quad
		-
		\frac{i}{2\sqrt{k}}
		\sum_{\ell=0}^{k-1}
		\xi_{\ell}^{\mp}(h,k)
		(S-H)\left( \tfrac{iu}{z} \mp \tfrac{i\wt{h}}{4kz} + \alpha^{\pm}(\ell,k); \tfrac{i}{kz} \right)
	\Bigg)
.
\end{align*}

We begin with an elliptic shift for $\tmu$ using that
\begin{align*}
\frac{iu}{z}
&=
	ku\frac{[-h]_k+i/z}{k} -u[-h]_k.
\end{align*}
With \eqref{EqMuTildeEllipticProperty} we find that
\begin{align*}
&\wt{\mu}\left(	\tfrac{iu}{z}, \pm\tfrac{\wt{h}([-h]_k+i/z)}{4k} \mp\tfrac{1+h[-h]_k}{4k}	; \tfrac{[-h]_k+i/z}{k}	\right)
\\
&=
	(-1)^{\Floor{ku}}
	\exp\left(
		-\tfrac{\pi}{kz}\left(  (ku)^2-	\Fractional{ku}^2 \mp\tfrac{\wt{h}\Floor{ku}}{2} \right) 
		+
		\tfrac{\pi i\Floor{ku}}{k}\left(
			-[-h]_k\Floor{ku} \pm \tfrac{1+h[-h]_k-\wt{h}[-h]_k}{2} 
		\right)
	\right)
	\\&\quad\times
	\wt{\mu}\left(	
		\Fractional{ku}\tfrac{[-h]_k+i/z}{k} - u[-h]_k,	
		\pm\tfrac{\wt{h}([-h]_k+i/z)}{4k} \mp\tfrac{1+h[-h]_k}{4k}	; \tfrac{[-h]_k+i/z}{k}	\right)
.
\end{align*}
By Propositions 1.2 and 1.9 of \cite{Zwegers1}, we deduce that for $n\in\mathbb{Z}$,
\begin{align*}
(S-H)(w+n\tau;\tau)
&=
	(-1)^n e^{ 2\pi inw + \pi in^2\tau} (S-H)(w;\tau)
.
\end{align*}
Thus
\begin{align*}
(S-H)\left(	\tfrac{iu}{z} \mp \tfrac{i\wt{h}}{4kz} + \alpha^{\pm}(\ell,k) ; \tfrac{i}{kz} \right)
&=
	(-1)^{\Floor{ku}}
	\exp\left(
		-\tfrac{\pi}{kz}\left(
			(ku)^2 - \Fractional{ku}^2 \mp\tfrac{\wt{h}}{2}\Floor{ku}
		\right)	
		+
		2\pi i\Floor{ku}\alpha^{\pm}(\ell,k)
	\right)
	\\&\quad\times
	(S-H)\left(
		\Fractional{ku}\tfrac{i}{kz} \mp\tfrac{i\wt{h}}{4kz} + \alpha^{\pm}(\ell,k) ; \tfrac{i}{kz}
	\right)
.
\end{align*}
Since
\begin{align*}
\tfrac{\left(\wt{h}\mp 4ku\right)^2}{16} - (ku)^2 + \Fractional{ku}^2 \pm \tfrac{\wt{h}}{2}\Floor{ku}
&=
	\left(\Fractional{ku} \mp \tfrac{\wt{h}}{4}\right)^2 
,
\end{align*}
we have
\begin{align*}
&\mu\left(	u, \pm\tfrac{h+iz}{4k}	; \tfrac{h+iz}{k}\right)
\\
&=
	(-1)^{\Floor{ku}}
	\exp\left(
		-\tfrac{\pi z}{16k} + \tfrac{\pi}{kz}\left(\Fractional{ku} \mp \tfrac{\wt{h}}{4}\right)^2 
	 	\mp \tfrac{\pi iu}{2}
	\right)
	z^{-\frac{1}{2}}
	\Bigg(
		\exp\left(
			\tfrac{\pi i\Floor{ku}}{k}\left(
				-[-h]_k\Floor{ku} \pm \tfrac{1+h[-h]_k-\wt{h}[-h]_k}{2} 
			\right)
		\right)
		\\&\quad\times
		\xi(h,[-h]_k,k)
		\mu\left(	
			\Fractional{ku}\tfrac{[-h]_k+i/z}{k} - u[-h]_k,	
			\pm\tfrac{\wt{h}([-h]_k+i/z)}{4k} \mp\tfrac{1+h[-h]_k}{4k}; 
			\tfrac{[-h]_k+i/z}{k}	
		\right)
		\\&\quad
		-
		\frac{i}{2\sqrt{k}}
		\sum_{\ell=0}^{k-1}
		\xi_{\ell}^{\mp}(h,k)
		\exp\left(	2\pi i\Floor{ku}\alpha^{\pm}(\ell,k)	\right)
		(S-H)\left(
			\left(\Fractional{ku}\mp\tfrac{\wt{h}}{4}\right)\tfrac{i}{kz}  + \alpha^{\pm}(\ell,k) ; \tfrac{i}{kz}
		\right)
		\\&\quad	
		+
		\exp\left(
			\tfrac{\pi i\Floor{ku}}{k}\left(
				-[-h]_k\Floor{ku} \pm \tfrac{1+h[-h]_k-\wt{h}[-h]_k}{2} 
			\right)
		\right)
		\\&\quad\times
		\frac{i\xi(h,k)}{2}	
		S\left(	
			\left(\Fractional{ku}\mp\tfrac{\wt{h}}{4}\right)\tfrac{[-h]_k+i/z}{k} 
			- u[-h]_k \pm\tfrac{1+h[-h]_k}{4k}; \tfrac{[-h]_k+i/z}{k}	
		\right)
	\Bigg)
.
\end{align*}

We next verify the cancellations between the $S(w;\tau)$ terms. These follow 
from arguments similar to those in the proof of Proposition 
\ref{PropositionMuModular}. We let 
$\tau=\frac{i}{kz}$ and suppose $\alpha,\beta,\gamma\in\mathbb{R}$.
We define $\delta_{\frac{1}{2},\gamma}$ by
\begin{align*}
\delta_{\frac{1}{2},\gamma}
&:=\begin{cases}
	1	&	\mbox{ if }\gamma\in\frac{1}{2}+\mathbb{Z},
	\\
	0	&	\mbox{ else}.
\end{cases}
\end{align*}
We find that
\begin{align*}
&S\left( \gamma\tau+\alpha;\tau+\beta \right)
\\
&=
	\frac{ q^{\frac12\gamma^2} }{\sqrt{\pi}}
	\sum_{\substack{n\in\frac{1}{2}+\mathbb{Z},\\n\not=-\gamma}}
	(-1)^{n-\frac{1}{2}}
	q^{-\frac12(n+\gamma)^2}
	e^{ -\pi in^2\beta - 2\pi in\alpha 	}
	\mbox{sgn}(n+\gamma)\Gamma\left(\tfrac{1}{2};\pi(n+\gamma)^2 2\IM{\tau}\right)
	\\&\quad
	+
	q^{\frac12\gamma^2}
	\sum_{\substack{n\in\frac{1}{2}+\mathbb{Z},\\ n\not=-\gamma}}
	(-1)^{n-\frac{1}{2}}
	q^{-\frac12(n+\gamma)^2}
	e^{ -\pi in^2\beta - 2\pi in\alpha }
	\left( 
		\mbox{sgn}(n)-\mbox{sgn}(n+\gamma)
	\right)	
	\\&\quad
	+
	\delta_{\frac{1}{2},\gamma}
	(-1)^{\gamma-\frac{1}{2}}\mbox{sgn}(\gamma)
	q^{\frac12\gamma^2}
	e^{ 	-\pi i\gamma^2\beta +2\pi i\gamma\alpha }	
.
\end{align*}
We note the series involving the incomplete $\Gamma$-function is of the form
\begin{align*}
	\sum_{n\in\mathbb{Q}\backslash\{0\}}a(n)\Gamma\left( \tfrac{1}{2}; 4\pi|n|\IM{\tau}\right)q^{-n}
,
\end{align*}
and we recall that a function of this form is identically zero when holomorphic. As such, the 
contribution from the various $S(w;\tau)$ terms is dependent on
whether or not $\gamma$ is a half integer and what values of
$n$ satisfy $\mbox{sgn}(n)=\mbox{sgn}(n+\gamma)$. We now apply
the above with $\gamma=\Fractional{ku}\mp\frac{\wt{h}}{4}$.
The calculations for the four ranges of $\Fractional{ku}\mp\frac{\wt{h}}{4}$
are similar, and so we only give the details for the final case.

When $\frac{1}{2}<\Fractional{ku}\mp\frac{\wt{h}}{4}<\frac{3}{2}$, we 
calculate that
\begin{align*}
&\,\,
\frac{i}{2}\xi(h,[-h]_k,k)
\exp\left( 
	\tfrac{\pi i\Floor{ku}}{k}\left(
		-[-h]_k\Floor{ku} \pm \tfrac{1+h[-h]_k-\wt{h}[-h]_k}{2}
	\right)
\right)
\\&\times
S\left(	\left(\Fractional{ku}\mp\tfrac{\wt{h}}{4}\right)\tfrac{[-h]_k+i/z}{k}
	-u[-h]_k \pm \tfrac{1+h[-h]_k}{4k} ; \tfrac{[-h]_k+i/z}{k}
\right)
\\&
-
\frac{i}{2\sqrt{k}}\sum_{\ell=0}^{k-1}\xi_\ell^{\mp}(h,k)
\exp\left( 2\pi i\Floor{ku}\alpha^{\pm}(\ell,k) \right)
S\left(	\left(\Fractional{ku}\mp\tfrac{\wt{h}}{4}\right)\tfrac{i}{kz} 
	+\alpha^{\pm}(\ell,k); \tfrac{i}{kz}
\right)
\\
&=
	i\xi(h,[-h]_k,k)
	\exp\left(
		\tfrac{\pi i\Floor{ku}}{k}\left(
			-[-h]_k\Floor{ku} \pm \tfrac{1+h[-h]_k-\wt{h}[-h]_k}{2}
		\right)	
		- \tfrac{\pi}{kz} \left(\Fractional{ku}\mp\tfrac{\wt{h}}{4}\right)^2
		- \tfrac{\pi i[-h]_k}{4k}
	\right)
	\\&\quad\times
	\exp\left(
		\tfrac{\pi}{kz} \left(\Fractional{ku}\mp\tfrac{\wt{h}}{4}-\tfrac{1}{2}\right)^2
		+ \pi i\left(
			\left(\Fractional{ku}\mp\tfrac{\wt{h}}{4}\right) \tfrac{[-h]_k}{k}
			-u[-h]_k \pm \tfrac{1+h[-h]_k}{4k}
		\right)
	\right)
	\\&\quad
	-
	\frac{i}{\sqrt{k}}\sum_{\ell=0}^{k-1}\xi_\ell^{\mp}(h,k)
	\exp\left( \pi i (2\Floor{ku}+1)\alpha^{\pm}(\ell,k)
		- \tfrac{\pi}{kz} \left(\Fractional{ku}\mp\tfrac{\wt{h}}{4}\right)^2
		+ \tfrac{\pi}{kz} \left(\Fractional{ku}\mp\tfrac{\wt{h}}{4}-\tfrac{1}{2}\right)^2
	\right)
\\
&=
	i\xi(h,[-h]_k,k)
	\exp\left(
		- \tfrac{\pi}{kz} \left(\Fractional{ku}\mp\tfrac{\wt{h}}{4}\right)^2
		+ \tfrac{\pi}{kz} \left(\Fractional{ku}\mp\tfrac{\wt{h}}{4}-\tfrac{1}{2}\right)^2
	\right)
	\\&\quad\times
	\exp\left(
		\tfrac{\pi i[-h]_k}{k}\left(
			-\Floor{ku}^2-\Floor{ku} \pm (2\Floor{ku}+1)\tfrac{h-\wt{h}}{4} -\tfrac{1}{4}
		\right)
		\pm \tfrac{\pi i(2\Floor{ku}+1)}{4k}
	\right)
	\\&\quad
	-
	\frac{i}{\sqrt{k}}\sum_{\ell=0}^{k-1}\xi_\ell^{\mp}(h,k)
	\exp\left( \pi i (2\Floor{ku}+1)\alpha^{\pm}(\ell,k)
		- \tfrac{\pi}{kz} \left(\Fractional{ku}\mp\tfrac{\wt{h}}{4}\right)^2
		+ \tfrac{\pi}{kz} \left(\Fractional{ku}\mp\tfrac{\wt{h}}{4}-\tfrac{1}{2}\right)^2
	\right)
.	
\end{align*}
The proposition then follows after elementary cancellations.
\end{proof}

Next we bound the $H$ terms appearing in Proposition
\ref{PropositionMuModularForRankAlt2}. As was the case with $N2_{2\ell}(n)$,
these will not contribute to the main term for $A\left(\frac{a}{c};n\right)$.

\begin{proposition}\label{PropositionHBounds}
Suppose $k$ is a positive integer, $\alpha,\beta\in\mathbb{R}$ with
$|\alpha|<\frac{1}{2}$ and $-\frac{1}{2}\le\beta<\frac{1}{2}$, and $\RE{z}>0$.
Then
\begin{align*}
\left|H\left(\tfrac{i\beta}{kz}+\alpha;\tfrac{i}{kz} \right)\right|
&\le
\begin{cases}
	\left|\sec(\pi \beta)\right| k^{\frac{1}{2}}\RE{\frac{1}{z}}^{-\frac{1}{2}} 
		\exp\left( -\frac{\pi\beta^2}{k}\RE{\frac{1}{z}} + \pi k\alpha^2\RE{\frac{1}{z}}^{-1} \right)
	&\mbox{if }\beta\not=-\frac{1}{2}
	,\\[.5ex]
	\left( 1+k^{\frac{1}{2}}\RE{\frac{1}{z}}^{-\frac{1}{2}} \right) 
		\exp\left( -\frac{\pi}{4k}\RE{\frac{1}{z}}  \right)
	&\mbox{if }\beta=-\frac{1}{2}
.
\end{cases}
\end{align*}
\end{proposition}
\begin{proof}
By definition,
\begin{align*}
H\left(\tfrac{i\beta}{kz}+\alpha;\tfrac{i}{kz} \right)
&=
	\int_{\mathbb{R}}
	\frac{ \exp\left(-\frac{\pi x^2}{kz}-2\pi x\left(\frac{i\beta}{kz}+\alpha\right)\right) }
		{\cosh(\pi x)}dx
\\
&=
	\exp\left(-\tfrac{\pi\beta^2}{kz}+2\pi i\alpha\beta\right)
	\int_{\mathbb{R}+i\beta}
	\frac{\exp\left( -\frac{\pi w^2}{kz} -2\pi w\alpha \right)}{\cosh(\pi(w-i\beta))}
	dw,
\end{align*}
where we have used the substitution $w=x+i\beta$. We wish to shift the
path of integration back to $\mathbb{R}$. For this we first note that
the integrand tends to zero as $|\RE{w}|\rightarrow\infty$.

When $\beta\not=-\frac{1}{2}$, the integrand has no poles
as $w$ varies from $0$ to $\beta$, and so for $\beta\not=-\frac{1}{2}$
we have that
\begin{align*}
H\left(\tfrac{i\beta}{kz}+\alpha;\tfrac{i}{kz} \right)
&=
	\exp\left(-\tfrac{\pi\beta^2}{kz}+2\pi i\alpha\beta\right)
	\int_{\mathbb{R}}
	\frac{\exp\left(-\frac{\pi w^2}{kz} -2\pi w\alpha\right) }{\cosh(\pi(w-i\beta))}
	dw
.
\end{align*}
Using the trivial bound
\begin{align*}
\frac{1}{|\cosh(\pi(w-i\beta))|}
&=
\frac{2}{|e^{\pi(w-i\beta)}+e^{-\pi(w-i\beta)}| }
\le
\frac{2}{|e^{-\pi i\beta}+e^{\pi i\beta} |}
=
|\sec(\pi \beta)|
,
\end{align*}
we find that
\begin{align*}
\left|H\left(\tfrac{i\beta}{kz}+\alpha;\tfrac{i}{kz} \right)\right|
&\le
	|\sec(\pi \beta)|
	\exp\left( -\tfrac{\pi\beta^2}{k}\RE{\tfrac{1}{z}} + \pi k\alpha^2\RE{\tfrac{1}{z}}^{-1}  \right)
	k^{\frac{1}{2}} \RE{\tfrac{1}{z}}^{-\frac{1}{2}}
.
\end{align*}

However, when $\beta=\frac{1}{2}$, the integrand has a simple pole at $w=0$
with residue $-\frac{i}{\pi}$. As such, we instead have
\begin{align*}
H\left(-\tfrac{i}{2kz}+\alpha;\tfrac{i}{kz} \right)
&=
	\exp\left(-\tfrac{\pi}{4kz}-\pi i\alpha\right)
	\left(
		1
		+
		\lim_{\varepsilon\rightarrow 0^+}\left(
			\int_{-\infty}^{-\varepsilon}
			+
			\int_{\varepsilon}^{\infty}
		\right)
		\frac{\exp\left(-\frac{\pi w^2}{kz} -2\pi w\alpha \right)}{\cosh(\pi(w+i/2))}  dw
	\right)
\\
&=
	\exp\left(-\tfrac{\pi}{4kz}-\pi i\alpha\right)
	\left(
		1
		+
		2i\int_{0}^{\infty} \exp\left( -\tfrac{\pi w^2}{kz} \right)
		\frac{\left( e^{2\pi w\alpha} - e^{-2\pi w\alpha} \right)}
			{\left( e^{\pi w}-e^{-\pi w} \right)} dw
	\right)
.
\end{align*}
Since $-\frac{1}{2}<\alpha<\frac{1}{2}$, we have that
$	
\left|\frac{e^{2\pi w\alpha} - e^{-2\pi w\alpha}}{e^{\pi w}-e^{-\pi w} }\right|
\le 1
$.
Thus
\begin{align*}
\left| H\left(-\tfrac{i}{2kz}+\alpha;\tfrac{i}{kz} \right) \right|
&\le
	\exp\left(-\tfrac{\pi}{4k}\RE{\tfrac{1}{z}}\right)
	\left(
		1
		+
		2\int_{0}^{\infty} \exp\left(-\tfrac{\pi w^2}{k}\RE{\tfrac{1}{z}} \right)
	\right)
\\
&=
	\exp\left(-\tfrac{\pi}{4k}\RE{\tfrac{1}{z}}\right)
	\left(
		1
		+
		k^{\frac{1}{2}} \RE{\tfrac{1}{z}}^{-\frac{1}{2}} 
	\right)
.
\end{align*}
\end{proof}

From Propositions \ref{PropositionMuModularForRankAlt2} and 
\ref{PropositionHBounds}, we deduce the following corollary.
This corollary contains the necessary transformation for the relevant
$\mu(u,v;\tau)$ and initial bounds.

\begin{corollary}\label{CorollaryMuTransformationWithBounds}
Suppose $h$ and $k$ are relatively prime integers with $k>0$,
$u\in\mathbb{R}$ with $u\not\in\mathbb{Z}$, and
$\RE{z}>0$.
If $-\frac{1}{2}\le \Fractional{ku}\mp\frac{\wt{h}}{4}\le\frac{1}{2}$, then
\begin{align*}
&\mu\left(u, \pm\tfrac{h+iz}{4k}; \tfrac{h+iz}{k}\right)
\\
&=
	(-1)^{\Floor{ku}}
	\exp\left(
		-\tfrac{\pi z}{16k}
		+\tfrac{\pi}{kz}\left(\Fractional{ku}\mp\tfrac{\wt{h}}{4}\right)^2
		\mp\tfrac{\pi iu}{2}
	\right)
	z^{-\frac{1}{2}}
	\exp\left(
		\tfrac{\pi i\Floor{ku}}{k}\left(
			-[-h]_k\Floor{ku} \pm \tfrac{1+h[-h]_k-\wt{h}[-h]_k}{2} 
		\right)
	\right)
	\\&\quad\times
	\xi(h,[-h]_k,k)
	\mu\left(
		\Fractional{ku}\tfrac{[-h]_k+i/z}{k} - u[-h]_k,
		\pm\tfrac{\wt{h}([-h]_k+i/z)}{4k} \mp \tfrac{1+h[-h]_k}{4k}; 
		\tfrac{[-h]_k+i/z}{k}
	\right)
	+
	E_2(u,h,k,z).
\end{align*}
If $\frac{1}{2}< \Fractional{ku}\mp\frac{\wt{h}}{4}<\frac{3}{2}$, then
\begin{align*}
&\mu\left(u, \pm\tfrac{h+iz}{4k}; \tfrac{h+iz}{k}\right)
\\
&=
	(-1)^{\Floor{ku}}
	\exp\left(
		-\tfrac{\pi z}{16k}
		+\tfrac{\pi}{kz}\left(\Fractional{ku}\mp\tfrac{\wt{h}}{4}\right)^2
		\mp\tfrac{\pi iu}{2}
	\right)
	z^{-\frac{1}{2}}
	\exp\left(
		\tfrac{\pi i\Floor{ku}}{k}\left(
			-[-h]_k\Floor{ku} \pm \tfrac{1+h[-h]_k-\wt{h}[-h]_k}{2} 
		\right)
	\right)
	\\&\quad\times
	\xi(h,[-h]_k,k)
	\mu\left(
		\Fractional{ku}\tfrac{[-h]_k+i/z}{k} - u[-h]_k,
		\pm\tfrac{\wt{h}([-h]_k+i/z)}{4k} \mp \tfrac{1+h[-h]_k}{4k}; 
		\tfrac{[-h]_k+i/z}{k}
	\right)
	\\&\quad
	+
	i(-1)^{\Floor{ku}}
	\exp\left(
		-\tfrac{\pi z}{16k}  
		+ \tfrac{\pi}{kz}\left(\Fractional{ku}\mp\tfrac{\wt{h}}{4}-\tfrac{1}{2}\right)^2  
		\mp \tfrac{\pi iu}{2}
	\right)
	z^{-\frac{1}{2}}\xi(h,[-h]_k,k)
	\\&\quad\times
	\exp\left(
		\tfrac{\pi i[-h]_k}{k}\left(
			-\Floor{ku}^2-\Floor{ku} \pm (2\Floor{ku}+1)\tfrac{h-\wt{h}}{4} -\tfrac{1}{4}
		\right)
		\pm \tfrac{\pi i(2\Floor{ku}+1)}{4k}
	\right)
	+
	E_2(u,h,k,z).
\end{align*}
Here $E_2(u,h,k,z)$ is bounded as follows,
\begin{align*}
|E_2(u,h,k,z)|
&\le
\begin{cases}
	\exp\left( -\frac{\pi\RE{z}}{16k} \right)
	|z|^{-\frac{1}{2}}
	\left(
		\tfrac{1}{2} + k^{\frac{1}{2}}+\tfrac{k}{2}\RE{\tfrac{1}{z}}^{-\frac{1}{2}}
	\right)
	&
	\mbox{if } \left|\Fractional{ku}\mp\frac{\wt{h}}{4}\right|=\frac{1}{2},
	\\[1.5ex]
	\left|\sec\left(\pi\left(\Fractional{ku}\mp\tfrac{\wt{h}}{4}\right)\right)\right|	
	\exp\left( -\frac{\pi\RE{z}}{16k} + \frac{\pi k}{4}\RE{\frac{1}{z}}^{-1} \right)
	|z|^{-\frac{1}{2}}
	\frac{k}{2}\RE{\frac{1}{z}}^{-\frac{1}{2}}
	&
	\mbox{else}.
\end{cases}
\end{align*}
\end{corollary}
\begin{proof}
When $-\frac12\le\Fractional{ku}\mp\tfrac{\wt{h}}{4}<\frac12$,
the corollary follows by directly applying the bounds in Proposition
\ref{PropositionHBounds} and noting that $|\alpha^{\pm}(\ell,k)|<\frac{1}{2}$.
We omit these calculations.

When $\Fractional{ku}\mp\tfrac{\wt{h}}{4}=\frac{1}{2}$, we note
by \eqref{EqHEven} that
\begin{align*}
H\left( \tfrac{i}{2kz} + \alpha^{\pm}(\ell,k) ; \tfrac{i}{kz} \right)
&=
H\left( -\tfrac{i}{2kz} - \alpha^{\pm}(\ell,k) ; \tfrac{i}{kz} \right)
.
\end{align*}
As such we obtain the same bounds as when 
$\Fractional{ku}\mp\tfrac{\wt{h}}{4}=-\tfrac{1}{2}$.

When $\frac{1}{2}<\Fractional{ku}\mp\tfrac{\wt{h}}{4}<\tfrac{3}{2}$, we 
use \eqref{EqHEllipticProperty1} to obtain that
\begin{align*}
H\left( \beta\tau+\alpha; \tau \right)
&=
	2\exp\left( \pi i((\beta-1)\tau +\alpha) + \tfrac{3\pi i\tau}{4}	\right)
	-
	\exp\left( 2\pi i\alpha + \pi i(2\beta-1)\tau \right)
	H\left( (\beta-1)\tau+\alpha ; \tau \right)
.
\end{align*}
We see that the additional contribution of 
\begin{align*}
&(-1)^{\Floor{ku}}
\exp\left(
	-\tfrac{\pi z}{16k} + \tfrac{\pi}{kz}\left(\Fractional{ku} \mp \tfrac{\wt{h}}{4}\right)^2 
 	\mp \tfrac{\pi iu}{2}
\right)
z^{-\frac{1}{2}}
\\&\times
\frac{i}{\sqrt{k}}
\sum_{\ell=0}^{k-1}
\xi_{\ell}^{\mp}(h,k)
\exp\left(
	2\pi i\Floor{ku}\alpha^{\pm}(\ell,k)	
	+\pi i\left( \Fractional{ku}\mp\tfrac{\wt{h}}{4} - 1 \right)\tfrac{i}{kz} +\pi i\alpha^{\pm}(\ell,k) 	 
	-\tfrac{3\pi }{4kz}	
\right)
\end{align*}
cancels exactly with the contribution of
\begin{align*}
&-(-1)^{\Floor{ku}}
\exp\left( -\tfrac{\pi z}{16k}  
	+ \tfrac{\pi}{kz}\left(\Fractional{ku}\mp\tfrac{\wt{h}}{4}-\tfrac{1}{2}\right)^2  
	\mp \tfrac{\pi iu}{2}
\right)
z^{-\frac{1}{2}}
\frac{i}{\sqrt{k}}\sum_{\ell=0}^{k-1}\xi_\ell^{\mp}(h,k)
\exp\left( \pi i (2\Floor{ku}+1)\alpha^{\pm}(\ell,k) \right)
\end{align*}
from $E_1(u,h,k,z)$.
Furthermore, when $\frac{1}{2}<\beta<\frac{3}{2}$, $|\alpha|<\frac{1}{2}$, and 
$\tau=\frac{i}{kz}$, we have that
\begin{align*}
&\left| 	\exp\left( -\tfrac{\pi(2\beta-1)}{kz} \right)
H\left( \tfrac{(\beta-1)i}{kz} + \alpha; \tfrac{i}{kz} \right)
\right|	
\\
&\le
	\left| \sec\left( \pi(\beta-1) \right) \right|
	\exp\left(
		-\tfrac{\pi(2\beta-1)}{k}\RE{\tfrac{1}{z}}
		-\tfrac{\pi(\beta-1)^2}{k}\RE{\tfrac{1}{z}}
		+\tfrac{\pi k}{4}\RE{\tfrac{1}{z}}^{-1}
	\right)
	k^{\frac{1}{2}} \RE{\tfrac{1}{z}}^{-\frac{1}{2}}
\\
&=
	\left| \sec\left( \pi\beta \right) \right|
	\exp\left(
		-\tfrac{\pi\beta^2}{k}\RE{\tfrac{1}{z}}
		+\tfrac{\pi k}{4}\RE{\tfrac{1}{z}}^{-1}
	\right)
	k^{\frac{1}{2}} \RE{\tfrac{1}{z}}^{-\frac{1}{2}}
.
\end{align*}
Thus we obtain the same bound when
$\frac{1}{2}<\Fractional{ku}\mp\tfrac{\wt{h}}{4}<\tfrac{3}{2}$,
as when
$-\frac{1}{2}<\Fractional{ku}\mp\tfrac{\wt{h}}{4}<\tfrac{1}{2}$.
\end{proof}

Depending on the value of $k$ modulo $4$, 
we will apply one of three transformations to
$R2\big(e^{\frac{2\pi ia}{c}};e^{\frac{2\pi i(h+iz)}{k}}\big)$. Each
of these three transformations results in two $\mu$-functions. For
each of these $\mu$-functions we must determine the $q$-terms with negative
exponents and explicitly bound the remaining terms. This is a 
straightforward, but lengthy process, and consumes the next six 
propositions. The proofs each require a similar set of calculations and as
such we omit many of the proofs of the later propositions.
The following proposition is for one of the two $\mu$-functions
corresponding the case when $k\equiv0\pmod{4}$.

\begin{proposition}\label{PropositionMuMainTerm1}
Suppose $u_1,u_2\in\mathbb{R}$ with $0\le u_1<1$, and $|q|^{\frac{1}{4}}<\frac{1}{2}$.
If $u_1=0$, then
\begin{align*}
q^{-\frac{1}{2}(u_1-\frac{1}{4})^2}
\mu\left( u_1\tau+u_2, \tfrac{\tau}{4};\tau \right)
&=
	-\frac{q^{-\frac{1}{32}}}{2\sin(\pi u_2)}+E
	,
\end{align*} 
where 
\begin{align*}
|E|
&\le 
	\frac{|q|^{\frac{23}{32}}}{1-|q|}
	+
	\frac{|q|^{\frac{7}{32}}}{1-2|q|^{\frac{1}{4}}}\left(
		\frac{1}{2|\sin(\pi u_2)|} +	\frac{1}{ 1-|q| }
	\right)
	+
	\frac{|q|^{\frac{39}{32}}	(1+|q|^{\frac{1}{4}})}
		{ (1-2|q|^{\frac{1}{4}})(1-|q|)^2 }
.
\end{align*}
If $u_1>0$, then
\begin{align*}
q^{-\frac{1}{2}(u_1-\frac{1}{4})^2}
\mu\left( u_1\tau+u_2, \tfrac{\tau}{4};\tau \right)
&=
	i\sum_{m=0}^{M_1} e^{\pi iu_2(2m+1)} q^{-\frac{u_1^2}{2} + \frac{3u_1}{4} -\frac{1}{32}+mu_1 }
	\\&\quad	
	+
	i\sum_{m=0}^{M_2} e^{-\pi iu_2(2m+1)} q^{-\frac{u_1^2}{2} - \frac{u_1}{4} +\frac{23}{32}+m(1-u_1) }
	+
	E
	,
\end{align*}
where $M_1=\left\lceil\frac{16u_1^2-56u_1+1}{32u_1}\right\rceil$,
$M_2=\left\lceil\frac{16u_1^2+40u_1-55}{32(1-u_1)}\right\rceil$, and
\begin{align*}
|E|
&\le 
	\frac{|q|^{\frac{39}{32}} (1+|q|^{\frac{1}{4}})}
		{ (1-2|q|^{\frac{1}{4}})(1-|q|)^2 }
	+
	\left(	\frac{1}{1-|q|^{u_1}} + \frac{1}{1-|q|^{1-u_1}}	\right)
	\left(	1+	\frac{|q|^{\frac{7}{32}}}{1-2|q|^{\frac{1}{4}}}	\right)
.
\end{align*}
\end{proposition}
\begin{proof}
By definition,
\begin{align*}
&q^{-\frac{1}{2}(u_1-\frac{1}{4})^2}
\mu\left( u_1\tau+u_2, \tfrac{\tau}{4};\tau \right)
=
	\frac{iq^{-\frac{1}{2}(u_1-\frac{1}{4})^2 +\frac{u_1}{2}} e^{\pi iu_2} }
		{\aqprod{q^{\frac{1}{4}}, q^{\frac{3}{4}}, q}{q}{\infty}}
	\sum_{n=-\infty}^\infty
	\frac{(-1)^n q^{\frac{n(n+1)}{2} +\frac{n}{4}}}
		{1- e^{2\pi iu_2}q^{n+u_1}}	
\\
&=
	\frac{iq^{-\frac{u_1^2}{2} +\frac{3u_1}{4} - \frac{1}{32}} e^{\pi iu_2} }
		{\aqprod{q^{\frac{1}{4}}, q^{\frac{3}{4}}, q}{q}{\infty}}
	\left(
		\frac{1}{1-e^{2\pi iu_2}q^{u_1}}		
		+
		\frac{ e^{-2\pi iu_2} q^{\frac{3}{4}-u_1}}{1-e^{-2\pi iu_2}q^{1-u_1}}
		+
		\sum_{n\in\mathbb{Z}\backslash\{0,-1\}}
		\frac{(-1)^n q^{\frac{n(n+1)}{2} +\frac{n}{4}}}
			{1-e^{2\pi iu_2}q^{n+u_1}}	
	\right)
	.
\end{align*}

First we bound the series term as
\begin{align*}
\left|		
	\sum_{n\in\mathbb{Z}\backslash\{0,-1\}}	
	\frac{(-1)^n q^{\frac{n(n+1)}{2} +\frac{n}{4}}}{1-e^{2\pi iu_2}q^{n+u_1}}	
\right|
&\le
	\frac{|q|^{\frac{1}{4}}}{1-|q|}
	\sum_{n=1}^\infty |q|^{\frac{n(n+1)}{2}}
	+
	\frac{|q|^{\frac{1}{2}}}{1-|q|}
	\sum_{n=1}^\infty |q|^{\frac{n(n+1)}{2}}
\le
	\frac{|q|^{\frac{5}{4}}+|q|^{\frac{3}{2}}}{(1-|q|)^2}
.
\end{align*}
In consideration of the infinite product, we let 
$p_{2,4}(n)$ denote the number of partitions of $n$ into
parts not congruent to $2$ modulo $4$. Thus
\begin{align*}
\left|\frac{1}{\aqprod{q^{\frac{1}{4}},q^{\frac{3}{4}},q}{q}{\infty}}\right|
&\le
\sum_{n=0}^\infty p_{2,4}(n)|q|^{\frac{n}{4}}
\le
\sum_{n=0}^\infty p(n)|q|^{\frac{n}{4}}
\le
\sum_{n=0}^\infty 2^n|q|^{\frac{n}{4}}
=
\frac{1}{1-2|q|^{\frac{1}{4}}}
.
\end{align*}
As such,
\begin{align*}	
\left| 
	\frac{iq^{-\frac{u_1^2}{2} +\frac{3u_1}{4} - \frac{1}{32}} e^{\pi iu_2} }
		{\aqprod{q^{\frac{1}{4}}, q^{\frac{3}{4}}, q}{q}{\infty}}
	\sum_{n\in\mathbb{Z}\backslash\{0,-1\}}
	\frac{(-1)^n q^{\frac{n(n+1)}{2} +\frac{n}{4}}} {1-e^{2\pi iu_2}q^{n+u_1}}	
\right|
&\le
	\frac{|q|^{\frac{39}{32}} (1+|q|^{\frac{1}{4}}) }
		{(1-2|q|^{\frac{1}{4}})(1-|q|)^2}
.
\end{align*}
Furthermore, we find that
\begin{align*}
&\left|
	iq^{-\frac{u_1^2}{2} +\frac{3u_1}{4} - \frac{1}{32}} e^{\pi iu_2}
	\left(	\frac{1}{\aqprod{q^{\frac{1}{4}}, q^{\frac{3}{4}}, q}{q}{\infty}} -1\right)
	\left(
		\frac{1}{1-e^{2\pi iu_2}q^{u_1}}		
		+
		\frac{e^{-2\pi iu_2}q^{\frac{3}{4}-u_1}}{1-e^{-2\pi iu_2}q^{1-u_1}}
	\right)
\right|
\\
&\le
	\frac{|q|^{-\frac{u_1^2}{2} +\frac{3u_1}{4} - \frac{1}{32}+\frac{1}{4}}}
		{1-2|q|^{\frac{1}{4}}}
	\left(
		\frac{1}{|1-e^{2\pi iu_2}q^{u_1}|}		
		+
		\frac{|q|^{\frac{3}{4}-u_1}}{1-|q|^{1-u_1}}
	\right)
\\
&\le
	\frac{|q|^{\frac{7}{32}}}{1-2|q|^{\frac{1}{4}}}
	\left(
		\frac{1}{|1-e^{2\pi iu_2}q^{u_1}|}		
		+
		\frac{1}{1-|q|^{1-u_1}}
	\right)
.
\end{align*}

We see the contribution to the main term is of a different form depending on whether $u_1=0$
or $u_1>0$. First we handle the case when $u_1=0$. We see that
\begin{align*}
\frac{iq^{-\frac{1}{32}} e^{\pi iu_2} }{1- e^{2\pi iu_2} }
&=
-\frac{q^{-\frac{1}{32}}}{2\sin(\pi u_2)}.
\end{align*}
The remaining term to bound is
\begin{align*}
\left|
\frac{iq^{-\frac{1}{32}+\frac{3}{4}} e^{-\pi iu_2} }{1-e^{-2\pi iu_2}q}
\right|
&\le
	\frac{|q|^{\frac{23}{32}}}{1-|q|}.
\end{align*}
Since the main term is as stated in the proposition, we need only verify the
error term $E$ is bounded as claimed. For this, we note the error
term is bounded by
\begin{align*}
&
\frac{|q|^{\frac{23}{32}}}{1-|q|}
+
\frac{|q|^{\frac{7}{32}}}{1-2|q|^{\frac{1}{4}}}
\left(	\frac{1}{|1-e^{2\pi iu_2}|}	+ \frac{1}{1-|q|} \right)
+
\frac{|q|^{\frac{39}{32}} (1+|q|^{\frac{1}{4}}) }
	{(1-2|q|^{\frac{1}{4}})(1-|q|)^2}
\\
&=
\frac{|q|^{\frac{23}{32}}}{1-|q|}
+
\frac{|q|^{\frac{7}{32}}}{1-2|q|^{\frac{1}{4}}}
\left(	\frac{1}{2|\sin(\pi u_2)|}	+ \frac{1}{1-|q|} \right)
+
\frac{|q|^{\frac{39}{32}} (1+|q|^{\frac{1}{4}}) }
	{(1-2|q|^{\frac{1}{4}})(1-|q|)^2}
.
\end{align*}

Next, when $u_1>0$, we begin by setting $b_1:=\frac{16u_1^2-24u_1+1}{32u_1}$ 
and $b_2:=\frac{16u_1^2+8u_1-23}{32(1-u_1)}$. In particular, this yields
\begin{align*}
&-\tfrac{u_1^2}{2} + \tfrac{3u_1}{4}-\tfrac{1}{32}+b_1u_1=0,\\
&-\tfrac{u_1^2}{2} + \tfrac{3u_1}{4}-\tfrac{1}{32}+\tfrac{3}{4}-u_1+b_2(1-u_1)=0,
\end{align*}
and $M_1=\lceil b_1-1\rceil$, $M_2=\lceil b_2-1\rceil$. As such,
$-\tfrac{u_1^2}{2} + \tfrac{3u_1}{4}-\tfrac{1}{32}+mu_1\ge 0$
exactly when $m>M_1$, and
$-\tfrac{u_1^2}{2} + \tfrac{3u_1}{4}-\tfrac{1}{32}+\tfrac{3}{4}-u_1+m(1-u_1)\ge 0$
exactly when $m>M_2$. We then write
\begin{align*}
\frac{1}{1-e^{2\pi iu_2}q^{u_1}}
&=
	\sum_{m=0}^{M_1}e^{2\pi iu_2m}q^{mu_1}
	+
	\frac{e^{2\pi iu_2(M_1+1)} q^{(M_1+1)u_1}}{1-e^{2\pi iu_2} q^{u_1}}
,\\
\frac{ e^{-2\pi iu_2} q^{\frac{3}{4}-u_1}}{1- e^{-2\pi iu_2} q^{1-u_1}}
&=
	\sum_{m=0}^{M_2}e^{-2\pi iu_2(m+1)} q^{\frac{3}{4}-u_1+m(1-u_1)}
	+
	\frac{ e^{-2\pi iu_2(M_2+2)} q^{\frac{3}{4}-u_1+(M_2+1)(1-u_1)}}
		{1- e^{ -2\pi iu_2} q^{1-u_1}}
,
\end{align*}
and observe the bounds
\begin{align*}
&\left|	iq^{-\frac{u_1^2}{2} +\frac{3u_1}{4} - \frac{1}{32}}e^{\pi iu_2}
\left(	
	\frac{ e^{2\pi iu_2(M_1+1)} q^{(M_1+1)u_1}}{1 - e^{2\pi iu_2} q^{u_1}}
	+
	\frac{e^{-2\pi iu_2(M_2+2)} q^{\frac{3}{4}-u_1+(M_2+1)(1-u_1)}}
		{1 - e^{-2\pi iu_2} q^{1-u_1}}
\right)\right|
\\
&\le
	\frac{1}{1-|q|^{u_1}}
	+
	\frac{1}{1-|q|^{1-u_1}}
.
\end{align*}
We see then the main term, when $u_1>0$, is as stated in the proposition.
Furthermore, we find the error term $E$ is bounded by
\begin{align*}
\frac{|q|^{\frac{39}{32}} (1+|q|^{\frac{1}{4}}) }{(1-2|q|^{\frac{1}{4}})(1-|q|)^2}
+
\left(	\frac{1}{1-|q|^{u_1}}	+	\frac{1}{1-|q|^{1-u_1}}	\right)
\left( 1+	\frac{|q|^{\frac{7}{32}}}{1-2|q|^{\frac{1}{4}}}	\right)
.
\end{align*}
\end{proof}

The next proposition handles the other $\mu$-function
corresponding to the case when $k\equiv0\pmod{4}$.

\begin{proposition}\label{PropositionMuMainTerm2}
Suppose $u_1,u_2\in\mathbb{R}$ with $0\le u_1<1$, and $|q|^{\frac{1}{4}}<\frac{1}{2}$.
If $u_1=0$, then
\begin{align*}
q^{-\frac{1}{2}(u_1+\frac{1}{4})^2}
\mu\left( u_1\tau+u_2, -\tfrac{\tau}{4};\tau \right)
&=
	\frac{q^{-\frac{1}{32}}}{2\sin(\pi u_2)}+E
	,
\end{align*} 
where 
\begin{align*}
|E|
&\le 
	\frac{|q|^{\frac{7}{32}}}{2|\sin(\pi u_2)|}\left(  
		1 +	\frac{1}{1-2|q|^{\frac{1}{4}}}
	\right)
	+
	\frac{|q|^{\frac{39}{32}}}{1-|q|}
	+
	\frac{|q|^{\frac{7}{32}}}{(1-2|q|^{\frac{1}{4}})(1-|q|)}
	+		
	\frac{|q|^{\frac{15}{32}} (1+|q|^{\frac{7}{4}})}
		{ (1-2|q|^{\frac{1}{4}})(1-|q|)^2 }
.
\end{align*}
If $u_1>0$, then
\begin{align*}
q^{-\frac{1}{2}(u_1+\frac{1}{4})^2}
\mu\left( u_1\tau+u_2, \tfrac{\tau}{4};\tau \right)
&=
	-i\sum_{m=0}^{M_3} e^{\pi iu_2(2m+1)} q^{-\frac{u_1^2}{2} + \frac{u_1}{4} -\frac{1}{32}+mu_1 }
	-i\sum_{m=0}^{M_4} e^{\pi iu_2(2m+1)} q^{-\frac{u_1^2}{2} + \frac{u_1}{4} +\frac{7}{32}+mu_1 }
	\\&\quad
	-i\sum_{m=0}^{M_5} e^{-\pi iu_2(2m+1)} q^{-\frac{u_1^2}{2} - \frac{3u_1}{4} +\frac{39}{32}+m(1-u_1) }
	+
	E
	,
\end{align*}
where $M_3=\left\lceil\frac{16u_1^2-40u_1+1}{32u_1}\right\rceil$,
$M_4=\left\lceil\frac{16u_1^2-40u_1-7}{32u_1}\right\rceil$,
$M_5=\left\lceil\frac{16u_1^2+56u_1-71}{32(1-u_1)}\right\rceil$, and
\begin{align*}
|E|
&\le 
	\frac{2}{1-|q|^{u_1}} 
	+ 
	\frac{1}{1-|q|^{1-u_1}}	
	+
	\frac{|q|^{\frac{7}{32}}}{1-2|q|^{\frac{1}{4}}}
		\left(	\frac{1}{1-|q|^{u_1}} + \frac{1}{1-|q|^{1-u_1}}	\right)
	+
	\frac{|q|^{\frac{15}{32}} (1+|q|^{\frac{7}{4}})}
		{ (1-2|q|^{\frac{1}{4}})(1-|q|)^2 }
.
\end{align*}
\end{proposition}
\begin{proof}
By definition,
\begin{align*}
&q^{-\frac{1}{2}(u_1+\frac{1}{4})^2}
\mu\left( u_1\tau+u_2, -\tfrac{\tau}{4};\tau \right)
\\
&=
	\frac{-iq^{-\frac{u_1^2}{2} +\frac{u_1}{4} - \frac{1}{32}} e^{\pi iu_2} }
		{\aqprod{q^{\frac{1}{4}}, q^{\frac{3}{4}}, q}{q}{\infty}}
	\left(
		\frac{1}{1- e^{2\pi iu_2} q^{u_1}}		
		+
		\frac{ e^{-2\pi iu_2} q^{\frac{5}{4}-u_1}}{1- e^{-2\pi iu_2} q^{1-u_1}}
		+
		\sum_{n\in\mathbb{Z}\backslash\{0,-1\}}
		\frac{(-1)^n q^{\frac{n(n+1)}{2} -\frac{n}{4}}}
			{1- e^{2\pi iu_2} q^{n+u_1}}	
	\right)
	.
\end{align*}
We bound the series term as
\begin{align*}
\left|		
	\sum_{n\in\mathbb{Z}\backslash\{0,-1\}	}
	\frac{(-1)^n q^{\frac{n(n+1)}{2} -\frac{n}{4}}}{1- e^{2\pi iu_2} q^{n+u_1}}	
\right|
&\le
	\frac{|q|^{\frac{3}{4}}+|q|^{\frac{7}{2}-u_1}}{(1-|q|)^2}
.
\end{align*}
Again we bound the infinite product by
$
\Big|\frac{1}{\aqprod{q^{\frac{1}{4}},q^{\frac{3}{4}},q}{q}{\infty}}\Big|
\le
\frac{1}{1-2|q|^{\frac{1}{4}}}$.
As such, we have that
\begin{align*}	
\left| 
	\frac{-iq^{-\frac{u_1^2}{2} +\frac{u_1}{4} - \frac{1}{32}} e^{\pi iu_2} }
		{\aqprod{q^{\frac{1}{4}}, q^{\frac{3}{4}}, q}{q}{\infty}}
	\sum_{n\in\mathbb{Z}\backslash\{0,-1\}}
	\frac{(-1)^n q^{\frac{n(n+1)}{2} -\frac{n}{4}}} {1- e^{2\pi iu_2} q^{n+u_1}}	
\right|
&\le
	\frac{|q|^{\frac{15}{32}} (1+|q|^{\frac{7}{4}}) }
		{(1-2|q|^{\frac{1}{4}})(1-|q|)^2}
.
\end{align*}
Furthermore,
\begin{align*}
\left|
	-iq^{-\frac{u_1^2}{2} +\frac{u_1}{4} - \frac{1}{32}} e^{\pi iu_2}
	\left(	\frac{1}{\aqprod{q^{\frac{1}{4}}, q^{\frac{3}{4}}, q}{q}{\infty}} -1\right)
	\frac{ e^{-2\pi iu_2} q^{\frac{5}{4}-u_1}}{1-e^{-2\pi iu_2} q^{1-u_1}}
\right|
&\le
	\frac{|q|^{\frac{7}{32}}}{(1-2|q|^{\frac{1}{4}})(1-|q|^{1-u_1})}
.
\end{align*}
We use that
\begin{align*}
\frac{1}{\aqprod{q^{\frac{1}{4}}, q^{\frac{3}{4}}, q}{q}{\infty}}
&= 
	1+q^{\frac{1}{4}}+\sum_{n=2}^\infty p_{2,4}(n)q^{\frac{n}{4}}
,
\end{align*}
where $p_{2,4}(n)$ is as in the proof of Proposition 
\ref{PropositionMuMainTerm1}, to obtain the bound
\begin{align*}
\left|
	-iq^{-\frac{u_1^2}{2} +\frac{u_1}{4} - \frac{1}{32}} e^{\pi iu_2}
	\left(	\frac{1}{\aqprod{q^{\frac{1}{4}}, q^{\frac{3}{4}}, q}{q}{\infty}} -1-q^{\frac{1}{4}}\right)
	\frac{1}{1- e^{2\pi iu_2} q^{u_1}}
\right|
&\le
	\frac{|q|^{\frac{7}{32}}}{(1-2|q|^{\frac{1}{4}})|1- e^{2\pi iu_2} q^{u_1}|}
.
\end{align*}

We see the contribution to the main term is different depending on whether $u_1=0$
or $u_1>0$. We first handle the case when $u_1=0$. We see that
\begin{align*}
\frac{-iq^{-\frac{1}{32}}e^{\pi iu_2}}{1-e^{2\pi iu_2}}
&=
\frac{q^{-\frac{1}{32}}}{2\sin(\pi u_2)}.
\end{align*}
The remaining  terms to bound are
\begin{align*}
\left|
	\frac{-iq^{-\frac{1}{32}+\frac{1}{4}}}{1-e^{2\pi iu_2}}
\right|
&\le
	\frac{|q|^{\frac{7}{32}}}{2|\sin(\pi u_2)|}
,&
\left|
	\frac{-iq^{-\frac{1}{32}+\frac{5}{4}} e^{-\pi iu_2} }{1-e^{-2\pi iu_2}q}
\right|
&\le
	\frac{|q|^{\frac{39}{32}}}{1-|q|}.
\end{align*}
The main term is as stated in the proposition. Furthermore, the error term
$E$ is bounded by
\begin{align*}
&
	\frac{|q|^{\frac{7}{32}}}{2|\sin(\pi u_2)|}
	+
	\frac{|q|^{\frac{39}{32}}}{1-|q|}
	+
	\frac{|q|^{\frac{7}{32}}}{2|\sin(\pi u_2)|(1-2|q|^{\frac{1}{4}})}
	+
	\frac{|q|^{\frac{7}{32}}}{(1-2|q|^{\frac{1}{4}})(1-|q|)}
	+
	\frac{|q|^{\frac{15}{32}} (1+|q|^{\frac{7}{4}}) }
		{(1-2|q|^{\frac{1}{4}})(1-|q|)^2}
.
\end{align*}

Next, when $u_1>0$, we begin by setting $b_3:=\frac{16u_1^2-8u_1+1}{32u_1}$,
$b_4:=\frac{16u_1^2-8u_1-7}{32u_1}$, and 
$b_5:=\frac{16u_1^2+24u_1-39}{32(1-u_1)}$. In particular, this yields
\begin{align*}
&-\tfrac{u_1^2}{2} + \tfrac{u_1}{4}-\tfrac{1}{32}+b_3u_1=0,\\
&-\tfrac{u_1^2}{2} + \tfrac{u_1}{4}-\tfrac{1}{32}+\tfrac{1}{4}+b_4u_1=0,\\
&-\tfrac{u_1^2}{2} + \tfrac{u_1}{4}-\tfrac{1}{32}+\tfrac{5}{4}-u_1+b_5(1-u_1)=0,
\end{align*}
and $M_j=\lceil b_j-1\rceil$. As such,
$-\tfrac{u_1^2}{2} + \tfrac{u_1}{4}-\tfrac{1}{32}+mu_1\ge 0$
exactly when $m>M_3$,
$-\tfrac{u_1^2}{2} + \tfrac{u_1}{4}-\tfrac{1}{32}+\frac{1}{4}+mu_1\ge 0$
exactly when $m>M_4$, and 
$-\tfrac{u_1^2}{2} + \tfrac{3u_1}{4}-\tfrac{1}{32}+\tfrac{5}{4}-u_1+m(1-u_1)\ge 0$
exactly when $m>M_5$. We then write
\begin{align*}
\frac{1+q^{\frac{1}{4}}}{1- e^{2\pi iu_2} q^{u_1}}
&=
	\sum_{m=0}^{M_3}e^{2\pi iu_2m} q^{mu_1}
	+
	\sum_{m=0}^{M_4}e^{2\pi iu_2m} q^{mu_1+\frac{1}{4}}
	+
	\frac{e^{2\pi iu_2(M_3+1)} q^{(M_3+1)u_1}}{1- e^{2\pi iu_2} q^{u_1}}
	\\&\quad
	+
	\frac{e^{2\pi iu_2(M_4+1)}q^{(M_4+1)u_1+\frac{1}{4}}}{1-e^{2\pi iu_2}q^{u_1}}
,\\
\frac{ e^{-2\pi iu_2} q^{\frac{5}{4}-u_1}}{1-e^{-2\pi iu_2} q^{1-u_1}}
&=
	\sum_{m=0}^{M_5} e^{-2\pi iu_2(m+1)} q^{\frac{5}{4}-u_1+m(1-u_1)}
	+
	\frac{ e^{-2\pi iu_2(M_5+2)} q^{\frac{5}{4}-u_1+(M_5+1)(1-u_1)}}
		{1- e^{-2\pi iu_2} q^{1-u_1}}
,
\end{align*}
and observe the bounds
\begin{align*}
&\left|	-iq^{-\frac{u_1^2}{2} +\frac{u_1}{4} - \frac{1}{32}} e^{\pi iu_2}
\left(	
	\frac{ e^{2\pi iu_2(M_3+1)} q^{(M_3+1)u_1}}{1- e^{2\pi iu_2} q^{u_1}}
	+
	\frac{ e^{2\pi iu_2(M_4+1)} q^{(M_4+1)u_1+\frac{1}{4}}}{1- e^{2\pi iu_2} q^{u_1}}
	\right.\right.\\&\qquad\qquad\qquad\qquad\qquad\qquad\qquad\qquad\qquad\qquad\qquad\qquad\left.\left.	
	+
	\frac{ e^{-2\pi iu_2(M_5+2)} q^{\frac{5}{4}-u_1+(M_5+1)(1-u_1)}}
		{1- e^{-2\pi iu_2} q^{1-u_1}}
\right)\right|
\\
&\le
	\frac{2}{1-|q|^{u_1}}
	+
	\frac{1}{1-|q|^{1-u_1}}
.
\end{align*}
We see then the main term, when $u_1>0$, is as stated in the proposition.
Furthermore, we find the error term $E$ is bounded 
as claimed.
\end{proof}

\sloppy
With Propositions \ref{PropositionMuMainTerm1} and \ref{PropositionMuMainTerm2},
we establish the transformation and required bounds
for $R2\big(e^{2\pi i u}; e^{\frac{2\pi i(h+iz)}{k}}\big)$ when
$k\equiv0\pmod{4}$.

\fussy

%Here we need Re(1/z)>k/2 so that the resulting q is small enough to apply
%the propositions with explicit error terms. Recall we needed small q to
%estimate by geometric series.
\begin{proposition}\label{PropFinalBounds0Mod4}
Suppose $a,c,k,n\in\mathbb{Z}$ with $c,k,n>0$ and $k\equiv 0\pmod{4}$,
$u=\frac{a}{c}$ with $2u\not\in\mathbb{Z}$, 
and $z\in\mathbb{C}$ with
$\RE{z}=\frac{k}{n}$ and
$\RE{\frac{1}{z}}>\frac{k}{2}$.
If $c\mid\frac{ka}{2}$, then
\begin{align*}
R2\left(e^{2\pi i u}; e^{\frac{2\pi i(h+iz)}{k}}\right)
&=
	-\wt{h}(-1)^{\frac{ku}{2} }
	\frac{\sin(\pi u)}{\sin(\pi u[-h]_{4k})}		
		z^{-\frac{1}{2}}
	\xi\left(h,[-h]_{4k},\tfrac{k}{4}\right)
	\exp\left(
		-\tfrac{\pi z}{4k} 
		+\tfrac{\pi}{4kz}
		-\pi i[-h]_{4k} ku^2
	\right)
	\\&\quad
	+E_3(u,k,n),
\end{align*}
where
\begin{align*}
|E_3(u,k,n)|
&\le 
	|\sin(\pi u)| \sqrt{\tfrac{n}{k}} 
	\left(0.208\csc(\tfrac{\pi}{c})+0.288 + 1.49\sqrt{k}
	\right)
.
\end{align*}
If $c\nmid\frac{ka}{2}$, then
\begin{align*}
&R2\left(e^{2\pi i u}; e^{\frac{2\pi i(h+iz)}{k}}\right)
=
	2i\sin(\pi u)(-1)^{\Floor{\frac{ku}{2}}} \xi\left(h,[-h]_{4k},\tfrac{k}{4}\right) z^{-\frac{1}{2}}
	\\&\quad\times\Bigg(
	-\wt{h}
	\sum_{m=0}^{M_1}
	\exp\left(
		\tfrac{4\pi i[-h]_{4k}}{k} \left(
			-\Floor{\tfrac{ku}{2}}^2 
			- (2m+\tfrac{3}{2})\Floor{\tfrac{ku}{2}} 
		\right)	
		-\tfrac{\pi z}{4k} 
		+\tfrac{8\pi }{kz}\left(
			\tfrac{1}{2}\Fractional{\tfrac{ku}{2}}^2 - \tfrac{3}{4}\Fractional{\tfrac{ku}{2}} 
			-m\Fractional{\tfrac{ku}{2}} + \tfrac{1}{32} 
		\right)
	\right)
	%%%%%%%%%%%%%%%%%%%%%%%%%%%%%%%%%%%%%%%%%%%%%%%%%%%%%%%%%%%%%%%%%%%%%%%%%%%%%%
	%%%%%%%%%%%%%%%%%%%%%%%%%%%%%%%%%%%%%%%%%%%%%%%%%%%%%%%%%%%%%%%%%%%%%%%%%%%%%%
	\\&\quad
	-
	\wt{h}
	\sum_{m=0}^{M_2}
	\exp\left(
		\tfrac{4\pi i[-h]_{4k}}{k} \left(
			-\Floor{\tfrac{ku}{2}}^2 
			+ (2m+\tfrac{1}{2})\Floor{\tfrac{ku}{2}} 
			+ 2m
			+ \tfrac{3}{2}
		\right)
	\right)
	\\&\qquad\times	
	\exp\left(
		-\tfrac{\pi z}{4k} 
		+\tfrac{8\pi }{kz}\left(
			\tfrac{1}{2}\Fractional{\tfrac{ku}{2}}^2 + \tfrac{1}{4}\Fractional{\tfrac{ku}{2}} 
			-m\left(1-\Fractional{\tfrac{ku}{2}}\right) - \tfrac{23}{32} 
		\right)
	\right)
	%%%%%%%%%%%%%%%%%%%%%%%%%%%%%%%%%%%%%%%%%%%%%%%%%%%%%%%%%%%%%%%%%%%%%%%%%%%%%%
	%%%%%%%%%%%%%%%%%%%%%%%%%%%%%%%%%%%%%%%%%%%%%%%%%%%%%%%%%%%%%%%%%%%%%%%%%%%%%%
	\\&\quad
	-
	\wt{h}
	\sum_{m=0}^{M_3}
	\exp\left(
		\tfrac{4\pi i[-h]_{4k}}{k} \left(
			-\Floor{\tfrac{ku}{2}}^2 
			- (2m+\tfrac{1}{2})\Floor{\tfrac{ku}{2}} 
		\right)	
		-\tfrac{\pi z}{4k} 
		+\tfrac{8\pi }{kz}\left(
			\tfrac{1}{2}\Fractional{\tfrac{ku}{2}}^2 - \tfrac{1}{4}\Fractional{\tfrac{ku}{2}} 
			-m\Fractional{\tfrac{ku}{2}} + \tfrac{1}{32} 
		\right)
	\right)
	%%%%%%%%%%%%%%%%%%%%%%%%%%%%%%%%%%%%%%%%%%%%%%%%%%%%%%%%%%%%%%%%%%%%%%%%%%%%%%
	%%%%%%%%%%%%%%%%%%%%%%%%%%%%%%%%%%%%%%%%%%%%%%%%%%%%%%%%%%%%%%%%%%%%%%%%%%%%%%
	\\&\quad
	-
	\wt{h}
	\sum_{m=0}^{M_4}
	\exp\left(
		\tfrac{4\pi i[-h]_{4k}}{k} \left(
			-\Floor{\tfrac{ku}{2}}^2 
			- (2m+\tfrac{1}{2})\Floor{\tfrac{ku}{2}} 
			+ \tfrac{1}{2} 
		\right)	
		-\tfrac{\pi z}{4k} 
		+\tfrac{8\pi }{kz}\left(
			\tfrac{1}{2}\Fractional{\tfrac{ku}{2}}^2 - \tfrac{1}{4}\Fractional{\tfrac{ku}{2}} 
			-m\Fractional{\tfrac{ku}{2}} - \tfrac{7}{32} 
		\right)
	\right)
	%%%%%%%%%%%%%%%%%%%%%%%%%%%%%%%%%%%%%%%%%%%%%%%%%%%%%%%%%%%%%%%%%%%%%%%%%%%%%%
	%%%%%%%%%%%%%%%%%%%%%%%%%%%%%%%%%%%%%%%%%%%%%%%%%%%%%%%%%%%%%%%%%%%%%%%%%%%%%%
	\\&\quad
	-
	\wt{h}
	\sum_{m=0}^{M_5} 
	\exp\left(
		\tfrac{4\pi i[-h]_{4k}}{k} \left(
			-\Floor{\tfrac{ku}{2}}^2 
			+ (2m+\tfrac{3}{2})\Floor{\tfrac{ku}{2}} 
			+2m
			+ \tfrac{5}{2}
		\right)	
	\right)
	\\&\qquad\times
	\exp\left(
		-\tfrac{\pi z}{4k} 
		+\tfrac{8\pi }{kz}\left(
			\tfrac{1}{2}\Fractional{\tfrac{ku}{2}}^2 + \tfrac{3}{4}\Fractional{\tfrac{ku}{2}} 
			-m\left(1-\Fractional{\tfrac{ku}{2}}\right) - \tfrac{39}{32} 
		\right)
	\right)
	\\&\quad
	-
	\delta_{1}
	\exp\left(
		\tfrac{4\pi i[-h]_{4k}}{k} \left(
			-\Floor{\tfrac{ku}{2}}^2 
			- \Floor{\tfrac{ku}{2}} 
			+ (2\Floor{\tfrac{ku}{2}}+1)\tfrac{h-\wt{h}}{4}			
			- \tfrac{1}{4}
		\right)	
		+
		\tfrac{\pi i}{k}(2\Floor{\tfrac{ku}{2}}+1)
	\right)
	\\&\qquad\times
	\exp\left(
		-\tfrac{\pi z}{4k} 
		+\tfrac{4\pi }{kz}\left(
			\Fractional{\tfrac{ku}{2}} - \tfrac{\wt{h}}{4} - \tfrac{1}{2} 
		\right)^2
	\right)
	\\&\quad
	+
	\delta_{2}
	\exp\left(
		\tfrac{4\pi i[-h]_{4k}}{k} \left(
			-\Floor{\tfrac{ku}{2}}^2 
			- \Floor{\tfrac{ku}{2}} 
			- (2\Floor{\tfrac{ku}{2}}+1)\tfrac{h-\wt{h}}{4}			
			- \tfrac{1}{4}
		\right)	
		-
		\tfrac{\pi i}{k}(2\Floor{\tfrac{ku}{2}}+1)
	\right)
	\\&\qquad\times
	\exp\left(
		-
		\tfrac{\pi z}{4k} 
		+\tfrac{4\pi }{kz}\left(
			\Fractional{\tfrac{ku}{2}} + \tfrac{\wt{h}}{4} - \tfrac{1}{2} 
		\right)^2
	\right)
	\Bigg)
	+
	E_3(u,k,n)
,
\end{align*}
where \sloppy
$M_1=\left\lceil\frac{16\Fractional{\frac{ku}{2}}^2-56\Fractional{\frac{ku}{2}}+1}
	{32\Fractional{\frac{ku}{2}}}\right\rceil$,
$M_2=\left\lceil\frac{16\Fractional{\frac{ku}{2}}^2+40\Fractional{\frac{ku}{2}}-55}
	{32(1-\Fractional{\frac{ku}{2}})}\right\rceil$,
$M_3=\left\lceil\frac{16\Fractional{\frac{ku}{2}}^2-40\Fractional{\frac{ku}{2}}+1}
	{32\Fractional{\frac{ku}{2}}}\right\rceil$,
$M_4=\left\lceil\frac{16\Fractional{\frac{ku}{2}}^2-40\Fractional{\frac{ku}{2}}-7}
	{32\Fractional{\frac{ku}{2}}}\right\rceil$,
$M_5=\left\lceil\frac{16\Fractional{\frac{ku}{2}}^2+56\Fractional{\frac{ku}{2}}-71}
	{32(1-\Fractional{\frac{ku}{2}})}\right\rceil$,
\begin{align*}
\delta_1
&=
\begin{cases}
	1	&	\text{if } \frac{1}{2}<\Fractional{\frac{ku}{2}}-\frac{\wt{h}}{4} < \frac{3}{2}
	\\
	0	&	 else
\end{cases}
,&
\delta_2
&=
\begin{cases}
	1	&	\text{if } \frac{1}{2}<\Fractional{\frac{ku}{2}}+\frac{\wt{h}}{4} < \frac{3}{2}
	\\
	0	&	 else
\end{cases}
,
\end{align*}
and
\begin{align*}
|E_3(u,k,n)|
&\le
	|\sin(\pi u)| \sqrt{\tfrac{n}{k}}
	\left(
		\frac{10.6}{1-\exp\left(-\frac{4\pi}{c}\right)}+0.0062
		+
		\frac{ 2.962+4.01\sqrt{k} }{ \sin(\frac{\pi}{2c}) }	
	\right)
.
\end{align*}
\end{proposition}
\begin{proof}
We recall that
\begin{align*}
R2\left(e^{2\pi i u}; e^{\frac{2\pi i(h+iz)}{k}}\right)
&=
2\sin(\pi u)\sum_{\pm}\mp e^{\pm \pi iu}
\mu\left(
	2u, \pm \tfrac{h+iz}{k} ; \tfrac{4(h+iz)}{k}
\right)
.
\end{align*}

We apply 
Corollary \ref{CorollaryMuTransformationWithBounds}
with $u\mapsto 2u$, $k\mapsto k/4$ and note that
$|\wt{h}|=1$ and $[-h]_{k/4}=[-h]_{4k}$. We only consider the case when 
$-\frac{1}{2}\le \Fractional{\frac{ku}{2}}\mp\frac{\wt{h}}{4} \le \frac{1}{2}$,
as the other case follows by similar cancellations and estimates.
We find that
\begin{align*}
&R2\left(e^{2\pi i u}; e^{\frac{2\pi i(h+iz)}{k}}\right)
=
	2\sin(\pi u)\sum_{\pm}\mp e^{\pm \pi iu}
	\mu\left( 2u, \pm \tfrac{h+iz}{k} ; \tfrac{h+iz}{k/4} \right)
\\
&=
	2\wt{h}\sin(\pi u)\sum_{\pm}\mp (-1)^{\Floor{\frac{ku}{2}}}
	\exp\left(
		-\tfrac{\pi z}{4k} 
		- \tfrac{4\pi i([-h]_{4k}+i/z) }{k}\left( \Fractional{\tfrac{ku}{2}} \mp \tfrac{1}{4}\right)^2
	\right)
	z^{-\frac{1}{2}}
	\xi\left( h,[-h]_{4k},\tfrac{k}{4} \right)
	\\&\quad\times
	\exp\left(
		\tfrac{4\pi i[-h]_{4k}}{k} \left(
			-\Floor{\tfrac{ku}{2}}^2 
			\mp
			\tfrac{1}{2}\Floor{\tfrac{ku}{2}} 
			+ \left( \Fractional{\tfrac{ku}{2}} \mp \tfrac{1}{4}\right)^2
		\right)
	\right)
	\\&\quad\times
	\mu\left(
		\Fractional{\tfrac{ku}{2}}\tfrac{4([-h]_k+i/z)}{k} - 2u[-h]_{4k}, 
		\pm \tfrac{[-h]_{4k}+i/z}{k}; 
		\tfrac{4([-h]_{4k}+i/z)}{k}
	\right)
	\\&\quad
	+
	E_3^{\prime}(u,h,k,z)
,
\end{align*}
where $|E_3^{\prime}(u,h,k,z)|\le 4|\sin(\pi u)E_2(2u,h,k/4,z)|$.
We apply Propositions \ref{PropositionMuMainTerm1} and 
\ref{PropositionMuMainTerm2} with $u_1=\Fractional{\frac{ku}{2}}$,
$u_2=-2u[-h]_{4k}$, and $\tau = \frac{4([-h]_{4k}+i/z)}{k}$. We
note that $|q|<e^{-4\pi}$. Furthermore, when $u_1=0$ we have that
$u_2$ is at least $\frac{1}{c}$ away from the nearest integer to $u_2$,
and when $u_1\not=0$ we have that $\frac{1}{c}\le u_1\le \frac{c-1}{c}$.

In the case when $c$ divides $\frac{ka}{2}$, we have that 
$\Fractional{\frac{ku}{2}}=0$, and so by Propositions \ref{PropositionMuMainTerm1} and
\ref{PropositionMuMainTerm2}, 
\begin{align*}
&\exp\left(
	- \tfrac{4\pi i([-h]_{4k}+i/z) }{k}\left( \Fractional{\tfrac{ku}{2}} - \tfrac{1}{4}\right)^2
\right)
\mu\left(
	\Fractional{\tfrac{ku}{2}}\tfrac{4([-h]_{4k}+i/z)}{k} -2u[-h]_{4k}, 
	\tfrac{[-h]_{4k}+i/z}{k}; \tfrac{4([-h]_{4k}+i/z)}{k}
\right)
\\
&=
	-\exp\left( -\tfrac{\pi i([-h]_{4k}+i/z) }{4k} \right)
	\frac{1}{2\sin(-2\pi u[-h]_{4k})}	
	+E,
\end{align*}
where $|E|<0.036\csc(\frac{\pi}{c})+0.071$, and
\begin{align*}
&\exp\left(
	- \tfrac{4\pi i([-h]_{4k}+i/z) }{k}\left( \Fractional{\tfrac{ku}{2}} + \tfrac{1}{4}\right)^2
\right)
\mu\left(
	\Fractional{\tfrac{ku}{2}}\tfrac{4([-h]_{4k}+i/z)}{k} -2u[-h]_{4k}, 
	-\tfrac{[-h]_{4k}+i/z}{k}; \tfrac{4([-h]_{4k}+i/z)}{k}
\right)
\\
&=
	\exp\left( -\tfrac{\pi i([-h]_{4k}+i/z) }{4k} \right)
	\frac{1}{2\sin(-2\pi u[-h]_{4k})}	
	+E,
\end{align*}
where $|E|<0.068\csc(\frac{\pi}{c})+0.073$.
Thus
\begin{align*}
&R2\left(e^{2\pi i u}; e^{\frac{2\pi i(h+iz)}{k}}\right)
\\
&=
	-\wt{h}(-1)^{\frac{ku}{2} }
	\frac{\sin(\pi u)}{\sin(\pi u[-h]_{4k})}		
		z^{-\frac{1}{2}}
	\xi\left(h,[-h]_{4k},\tfrac{k}{4}\right)
	\exp\left(
		-\tfrac{\pi z}{4k} 
		+\tfrac{\pi}{4kz}
		-\pi i[-h]_{4k} ku^2
	\right)
	+E_3(u,k,n)	
	.
\end{align*}
Here we see we can bound $E_3(u,k,n)$ by
\begin{align*}
|E_3(u,k,n)|
&\le
	2|\sin(\pi u)| |z|^{-\frac{1}{2}}
	\left(0.104\csc(\tfrac{\pi}{c})+0.144\right)
	+
	4|\sin(\pi u) E_2(2u,h,k/4,z) |
\\
&\le
	|\sin(\pi u)| \sqrt{\tfrac{n}{k}} 
	\left(0.208\csc(\tfrac{\pi}{c})+0.288 + 1.49\sqrt{k}
	\right)
.
\end{align*}

In the case when $c$ does not divide $\frac{ka}{2}$, we have that 
$\Fractional{\frac{ku}{2}}\not=0$, and so by Propositions \ref{PropositionMuMainTerm1} and
\ref{PropositionMuMainTerm2},
\begin{align*}
&\exp\left(
	- \tfrac{4\pi i([-h]_{4k}+i/z) }{k}\left( \Fractional{\tfrac{ku}{2}} - \tfrac{1}{4}\right)^2
\right)
\mu\left(
	\Fractional{\tfrac{ku}{2}}\tfrac{4([-h]_{4k}+i/z)}{k} -2u[-h]_{4k}, 
	\tfrac{[-h]_{4k}+i/z}{k}; \tfrac{4([-h]_{4k}+i/z)}{k}
\right)
\\
&=
	i\sum_{m=0}^{M_1} 
	\exp\left(
		-2\pi iu[-h]_{4k}(2m+1) 
		+
		\tfrac{8\pi i([-h]_{4k}+i/z)}{k}
		\left(
			-\tfrac{1}{2}\Fractional{\tfrac{ku}{2}}^2 + \tfrac{3}{4}\Fractional{\tfrac{ku}{2}} 
			+m\Fractional{\tfrac{ku}{2}} - \tfrac{1}{32} 
		\right)
	\right)		
	\\&\quad	
	+
	i\sum_{m=0}^{M_2} 
	\exp\left( 
		2\pi iu[-h]_{4k}(2m+1) 
		+
		\tfrac{8\pi i([-h]_{4k}+i/z)}{k}
		\left(
			-\tfrac{1}{2}\Fractional{\tfrac{ku}{2}}^2 - \tfrac{1}{4}\Fractional{\tfrac{ku}{2}} 
			+m\left(1-\Fractional{\tfrac{ku}{2}}\right) + \tfrac{23}{32} 
		\right)
	\right)		
	+
	E
	,
\end{align*}
where $|E|<\frac{2.15}{1-\exp\left(-\frac{4\pi}{c}\right)}+0.0000003$, and
\begin{align*}
&\exp\left(
	- \tfrac{4\pi i([-h]_{4k}+i/z) }{k}\left( \Fractional{\tfrac{ku}{2}} + \tfrac{1}{4}\right)^2
\right)
\mu\left(
	\Fractional{\tfrac{ku}{2}}\tfrac{4([-h]_{4k}+i/z)}{k} -2u[-h]_{4k}, 
	-\tfrac{[-h]_{4k}+i/z}{k}; \tfrac{4([-h]_{4k}+i/z)}{k}
\right)
\\
&=
	-i\sum_{m=0}^{M_3} 
	\exp\left(
		-2\pi iu[-h]_{4k}(2m+1) 
		+
		\tfrac{8\pi i([-h]_{4k}+i/z)}{k}
		\left(
			-\tfrac{1}{2}\Fractional{\tfrac{ku}{2}}^2 + \tfrac{1}{4}\Fractional{\tfrac{ku}{2}} 
			+m\Fractional{\tfrac{ku}{2}} - \tfrac{1}{32} 
		\right)
	\right)	
	\\&\quad
	-i\sum_{m=0}^{M_4} 
	\exp\left(
		-2\pi iu[-h]_{4k}(2m+1) 
		+
		\tfrac{8\pi i([-h]_{4k}+i/z)}{k}
		\left( 
			-\tfrac{1}{2}\Fractional{\tfrac{ku}{2}}^2 + \tfrac{1}{4}\Fractional{\tfrac{ku}{2}} 
			+m\Fractional{\tfrac{ku}{2}} + \tfrac{7}{32} 
		\right)
	\right)		
	\\&\quad
	-i\sum_{m=0}^{M_5} 
	\exp\left(
		2\pi iu[-h]_{4k}(2m+1) 
		+
		\tfrac{8\pi i([-h]_{4k}+i/z)}{k}
		\left( 
			-\tfrac{1}{2}\Fractional{\tfrac{ku}{2}}^2 - \tfrac{3}{4}\Fractional{\tfrac{ku}{2}} 
			+m\left(1-\Fractional{\tfrac{ku}{2}}\right)	+ \tfrac{39}{32} 
		\right)
	\right)	
	+
	E
	,
\end{align*}
where $|E|<\frac{3.15}{1-\exp\left(-\frac{4\pi}{c}\right)}+0.00303$.
As such, when $c$ does not divide $\frac{ka}{2}$,
\begin{align*}
&R2\left(e^{2\pi i u}; e^{\frac{2\pi i(h+iz)}{k}}\right)
\\
&=
	-2i\wt{h}\sin(\pi u)(-1)^{\Floor{\frac{ku}{2}}}
	\xi\left( h,[-h]_{4k},\tfrac{k}{4}\right)
	z^{-\frac{1}{2}}
	\\&\quad\times
	\Bigg(
	\sum_{m=0}^{M_1} 
	\exp\left( 
		-2\pi iu[-h]_{4k}(2m+1) 
		-
		\tfrac{\pi z}{4k} 
		+\tfrac{8\pi }{kz}\left(
			\tfrac{1}{2}\Fractional{\tfrac{ku}{2}}^2 - \tfrac{3}{4}\Fractional{\tfrac{ku}{2}} 
			-m\Fractional{\tfrac{ku}{2}} + \tfrac{1}{32} 
		\right)		
	\right)
	\\&\qquad\times
	\exp\left(
		\tfrac{4\pi i[-h]_{4k}}{k} \left(
			-\Floor{\tfrac{ku}{2}}^2 
			- \tfrac{1}{2}\Floor{\tfrac{ku}{2}} 
			+ \left( \Fractional{\tfrac{ku}{2}} - \tfrac{1}{4}\right)^2
			-\Fractional{\tfrac{ku}{2}}^2 + \tfrac{3}{2}\Fractional{\tfrac{ku}{2}} 
			+2m\Fractional{\tfrac{ku}{2}} - \tfrac{1}{16}
		\right)
	\right)
	%%%%%%%%%%%%%%%%%%%%%%%%%%%%%%%%%%%%%%%%%%%%%%%%%%%%%%%%%%%%%%%%%%%%%%%%%%%%%%
	%%%%%%%%%%%%%%%%%%%%%%%%%%%%%%%%%%%%%%%%%%%%%%%%%%%%%%%%%%%%%%%%%%%%%%%%%%%%%%
	\\&\quad
	+	
	\sum_{m=0}^{M_2} 
	\exp\left(
		 2\pi iu[-h]_{4k}(2m+1)
		-\tfrac{\pi z}{4k} 
		+\tfrac{8\pi }{kz}\left(
			\tfrac{1}{2}\Fractional{\tfrac{ku}{2}}^2 + \tfrac{1}{4}\Fractional{\tfrac{ku}{2}} 
			-m\left(1-\Fractional{\tfrac{ku}{2}}\right) - \tfrac{23}{32} 
		\right)		 
	\right)
	\\&\qquad\times
	\exp\left(
		\tfrac{4\pi i[-h]_{4k}}{k} \left(
			-\Floor{\tfrac{ku}{2}}^2 
			- \tfrac{1}{2}\Floor{\tfrac{ku}{2}} 
			+ \left( \Fractional{\tfrac{ku}{2}} - \tfrac{1}{4}\right)^2
			-\Fractional{\tfrac{ku}{2}}^2 - \tfrac{1}{2}\Fractional{\tfrac{ku}{2}} 
			+2m\left(1-\Fractional{\tfrac{ku}{2}}\right) + \tfrac{23}{16}
		\right)
	\right)
	%%%%%%%%%%%%%%%%%%%%%%%%%%%%%%%%%%%%%%%%%%%%%%%%%%%%%%%%%%%%%%%%%%%%%%%%%%%%%%
	%%%%%%%%%%%%%%%%%%%%%%%%%%%%%%%%%%%%%%%%%%%%%%%%%%%%%%%%%%%%%%%%%%%%%%%%%%%%%%
	\\&\quad
	+
	\sum_{m=0}^{M_3} 
	\exp\left(
		-2\pi iu[-h]_{4k}(2m+1) 
		-\tfrac{\pi z}{4k} 
		+\tfrac{8\pi }{kz}\left(
			\tfrac{1}{2}\Fractional{\tfrac{ku}{2}}^2 - \tfrac{1}{4}\Fractional{\tfrac{ku}{2}} 
			-m\Fractional{\tfrac{ku}{2}} + \tfrac{1}{32} 
		\right)
	\right)
	\\&\qquad\times	
	\exp\left(
		\tfrac{4\pi i[-h]_{4k}}{k} \left(
			-\Floor{\tfrac{ku}{2}}^2 
			+ \tfrac{1}{2}\Floor{\tfrac{ku}{2}} 
			+ \left( \Fractional{\tfrac{ku}{2}} + \tfrac{1}{4}\right)^2
			-\Fractional{\tfrac{ku}{2}}^2 + \tfrac{1}{2}\Fractional{\tfrac{ku}{2}} 
			+2m\Fractional{\tfrac{ku}{2}} - \tfrac{1}{16} 
		\right)
	\right)	
	%%%%%%%%%%%%%%%%%%%%%%%%%%%%%%%%%%%%%%%%%%%%%%%%%%%%%%%%%%%%%%%%%%%%%%%%%%%%%%
	%%%%%%%%%%%%%%%%%%%%%%%%%%%%%%%%%%%%%%%%%%%%%%%%%%%%%%%%%%%%%%%%%%%%%%%%%%%%%%
	\\&\quad
	+
	\sum_{m=0}^{M_4}
	\exp\left(
		-2\pi iu[-h]_{4k}(2m+1) 
		-\tfrac{\pi z}{4k} 
		+\tfrac{8\pi }{kz}\left(
			\tfrac{1}{2}\Fractional{\tfrac{ku}{2}}^2 - \tfrac{1}{4}\Fractional{\tfrac{ku}{2}} 
			-m\Fractional{\tfrac{ku}{2}} - \tfrac{7}{32} 
		\right)
	\right)
	\\&\qquad\times	
	\exp\left(
		 \tfrac{4\pi i[-h]_{4k}}{k} \left(
			-\Floor{\tfrac{ku}{2}}^2 
			+ \tfrac{1}{2}\Floor{\tfrac{ku}{2}} 
			+ \left( \Fractional{\tfrac{ku}{2}} + \tfrac{1}{4}\right)^2
			-\Fractional{\tfrac{ku}{2}}^2 + \tfrac{1}{2}\Fractional{\tfrac{ku}{2}} 
			+2m\Fractional{\tfrac{ku}{2}} + \tfrac{7}{16} 
		\right)
	\right)
	%%%%%%%%%%%%%%%%%%%%%%%%%%%%%%%%%%%%%%%%%%%%%%%%%%%%%%%%%%%%%%%%%%%%%%%%%%%%%%
	%%%%%%%%%%%%%%%%%%%%%%%%%%%%%%%%%%%%%%%%%%%%%%%%%%%%%%%%%%%%%%%%%%%%%%%%%%%%%%
	\\&\quad
	+
	\sum_{m=0}^{M_5}
	\exp\left( 
		2\pi iu[-h]_{4k}(2m+1) 
		-\tfrac{\pi z}{4k} 
		+\tfrac{8\pi }{kz}\left(
			\tfrac{1}{2}\Fractional{\tfrac{ku}{2}}^2 + \tfrac{3}{4}\Fractional{\tfrac{ku}{2}} 
			-m\left(1-\Fractional{\tfrac{ku}{2}}\right) - \tfrac{39}{32} 
		\right)
	\right)
	\\&\qquad\times
	\exp\left(
		\tfrac{4\pi i[-h]_{4k}}{k} \left(
			-\Floor{\tfrac{ku}{2}}^2 
			+ \tfrac{1}{2}\Floor{\tfrac{ku}{2}} 
			+ \left( \Fractional{\tfrac{ku}{2}} + \tfrac{1}{4}\right)^2
			-\Fractional{\tfrac{ku}{2}}^2 - \tfrac{3}{2}\Fractional{\tfrac{ku}{2}} 
			+2m\left(1-\Fractional{\tfrac{ku}{2}}\right) + \tfrac{39}{16}
		\right)
	\right)
	\\&\quad
	+
	E_3(u,k,n)
%%%%%%%%%%%%%%%%%%%%%%%%%%%%%%%%%%%%%%%%%%%%%%%%%%%%%%%%%%%%%%%%%%%%%%%%%%%%%%
%%%%%%%%%%%%%%%%%%%%%%%%%%%%%%%%%%%%%%%%%%%%%%%%%%%%%%%%%%%%%%%%%%%%%%%%%%%%%%
\\
&=
	-2i\wt{h}\sin(\pi u)(-1)^{\Floor{\frac{ku}{2}}}
	\xi\left(h,[-h]_{4k},\tfrac{k}{4}\right)
	z^{-\frac{1}{2}}
	\\&\quad\times
	\Bigg(
	\sum_{m=0}^{M_1}
	\exp\left(
		\tfrac{4\pi i[-h]_{4k}}{k} \left(
			-\Floor{\tfrac{ku}{2}}^2 
			- (2m+\tfrac{3}{2})\Floor{\tfrac{ku}{2}} 
		\right)	
		-\tfrac{\pi z}{4k} 
		+\tfrac{8\pi }{kz}\left(
			\tfrac{1}{2}\Fractional{\tfrac{ku}{2}}^2 - \tfrac{3}{4}\Fractional{\tfrac{ku}{2}} 
			-m\Fractional{\tfrac{ku}{2}} + \tfrac{1}{32} 
		\right)
	\right)
	%%%%%%%%%%%%%%%%%%%%%%%%%%%%%%%%%%%%%%%%%%%%%%%%%%%%%%%%%%%%%%%%%%%%%%%%%%%%%%
	%%%%%%%%%%%%%%%%%%%%%%%%%%%%%%%%%%%%%%%%%%%%%%%%%%%%%%%%%%%%%%%%%%%%%%%%%%%%%%
	\\&\quad
	+
	\sum_{m=0}^{M_2}
	\exp\left(
		\tfrac{4\pi i[-h]_{4k}}{k} \left(
			-\Floor{\tfrac{ku}{2}}^2 
			+ (2m+\tfrac{1}{2})\Floor{\tfrac{ku}{2}} 
			+ 2m
			+ \tfrac{3}{2}
		\right)	
	\right)
	\\&\qquad\times
	\exp\left(
		-
		\tfrac{\pi z}{4k} 
		+\tfrac{8\pi }{kz}\left(
			\tfrac{1}{2}\Fractional{\tfrac{ku}{2}}^2 + \tfrac{1}{4}\Fractional{\tfrac{ku}{2}} 
			-m\left(1-\Fractional{\tfrac{ku}{2}}\right) - \tfrac{23}{32} 
		\right)
	\right)
	%%%%%%%%%%%%%%%%%%%%%%%%%%%%%%%%%%%%%%%%%%%%%%%%%%%%%%%%%%%%%%%%%%%%%%%%%%%%%%
	%%%%%%%%%%%%%%%%%%%%%%%%%%%%%%%%%%%%%%%%%%%%%%%%%%%%%%%%%%%%%%%%%%%%%%%%%%%%%%
	\\&\quad
	+
	\sum_{m=0}^{M_3}
	\exp\left(
		\tfrac{4\pi i[-h]_{4k}}{k} \left(
			-\Floor{\tfrac{ku}{2}}^2 
			- (2m+\tfrac{1}{2})\Floor{\tfrac{ku}{2}} 
		\right)	
		-
		\tfrac{\pi z}{4k} 
		+\tfrac{8\pi }{kz}\left(
			\tfrac{1}{2}\Fractional{\tfrac{ku}{2}}^2 - \tfrac{1}{4}\Fractional{\tfrac{ku}{2}} 
			-m\Fractional{\tfrac{ku}{2}} + \tfrac{1}{32} 
		\right)
	\right)
	%%%%%%%%%%%%%%%%%%%%%%%%%%%%%%%%%%%%%%%%%%%%%%%%%%%%%%%%%%%%%%%%%%%%%%%%%%%%%%
	%%%%%%%%%%%%%%%%%%%%%%%%%%%%%%%%%%%%%%%%%%%%%%%%%%%%%%%%%%%%%%%%%%%%%%%%%%%%%%
	\\&\quad
	+
	\sum_{m=0}^{M_4}
	\exp\left(
		\tfrac{4\pi i[-h]_{4k}}{k} \left(
			-\Floor{\tfrac{ku}{2}}^2 
			- (2m+\tfrac{1}{2})\Floor{\tfrac{ku}{2}} 
			+ \tfrac{1}{2} 
		\right)	
		-\tfrac{\pi z}{4k} 
		+\tfrac{8\pi }{kz}\left(
			\tfrac{1}{2}\Fractional{\tfrac{ku}{2}}^2 - \tfrac{1}{4}\Fractional{\tfrac{ku}{2}} 
			-m\Fractional{\tfrac{ku}{2}} - \tfrac{7}{32} 
		\right)
	\right)
	%%%%%%%%%%%%%%%%%%%%%%%%%%%%%%%%%%%%%%%%%%%%%%%%%%%%%%%%%%%%%%%%%%%%%%%%%%%%%%
	%%%%%%%%%%%%%%%%%%%%%%%%%%%%%%%%%%%%%%%%%%%%%%%%%%%%%%%%%%%%%%%%%%%%%%%%%%%%%%
	\\&\quad
	+
	\sum_{m=0}^{M_5} 
	\exp\left(
		\tfrac{4\pi i[-h]_{4k}}{k} \left(
			-\Floor{\tfrac{ku}{2}}^2 
			+ (2m+\tfrac{3}{2})\Floor{\tfrac{ku}{2}} 
			+2m
			+ \tfrac{5}{2}
		\right)	
	\right)
	\\&\qquad\times
	\exp\left(
		-
		\tfrac{\pi z}{4k} 
		+\tfrac{8\pi }{kz}\left(
			\tfrac{1}{2}\Fractional{\tfrac{ku}{2}}^2 + \tfrac{3}{4}\Fractional{\tfrac{ku}{2}} 
			-m\left(1-\Fractional{\tfrac{ku}{2}}\right) - \tfrac{39}{32} 
		\right)
	\right)
	\Bigg)
	+
	E_3(u,k,n)
.
\end{align*}
We find that $E_3(u,k,n)$ can be bounded as
\begin{align*}
|E_3(u,k,n)|
&\le
	2|\sin(\pi u)| |z|^{-\frac{1}{2}}
	\left(
		\frac{5.3}{1-\exp\left(-\frac{4\pi}{c}\right)}+0.0031
	\right)
	\\&\quad
	+
	4|\sin(\pi u)| |z|^{-\frac{1}{2}}
	\csc(\tfrac{\pi}{2c}) \exp\left(\tfrac{\pi}{8}\right)
	\left(
		\tfrac{1}{2} + \tfrac{\sqrt{k}}{2} + \tfrac{k}{8}\sqrt{\tfrac{2}{k}}
	\right)	
\\
&\le
	|\sin(\pi u)| \sqrt{\tfrac{n}{k}}
	\left(
		\frac{10.6}{1-\exp\left(-\frac{4\pi}{c}\right)}+0.0062
		+
		\frac{ 2.962+4.01\sqrt{k} }{ \sin(\frac{\pi}{2c}) }	
	\right)
.
\end{align*}
\end{proof}

The following proposition handles the two $\mu$-functions
when $k\equiv2\pmod{4}$. The proof is a sequence of calculations in a manner 
similar to the proof
of Propositions \ref{PropositionMuMainTerm1} and \ref{PropositionMuMainTerm2},
which we omit for the sake of brevity.

\begin{proposition}\label{PropositionMuMainTerm3}
Suppose $u_1,u_2\in\mathbb{R}$ with $0\le u_1<1$, $u_3$ is an odd integer, and
$|q|^2<\frac{1}{2}$.
If $u_1=0$, then
\begin{align*}
\left|
q^{-\frac{1}{2}(u_1-\frac{1}{2})^2}
\mu\left( u_1\tau+u_2, \tfrac{\tau}{2}+\tfrac{u_3}{4};\tau	\right)
\right|
&\le
	\frac{1}{(1-2|q|^2)^2}\left(
		\frac{1}{2|\sin(\pi u_2)|}
		+
		\frac{|q|^{\frac{1}{2}} (1+|q|) }{(1-|q|)^2}
	\right)
,\\
\left|
q^{-\frac{1}{2}(u_1+\frac{1}{2})^2}
\mu\left( u_1\tau+u_2, -\tfrac{\tau}{2}+\tfrac{u_3}{4};\tau	\right)
\right|
&\le
	\frac{1}{(1-2|q|^2)^2}\left(
		\frac{1}{2|\sin(\pi u_2)|}
		+
		\frac{|q|^{\frac{1}{2}} (1+|q|) }{(1-|q|)^2}
	\right)
.
\end{align*}
If $u_1\not=0$, then
\begin{align*}
\left|
q^{-\frac{1}{2}(u_1-\frac{1}{2})^2}
\mu\left( u_1\tau+u_2, \tfrac{\tau}{2}+\tfrac{u_3}{4};\tau	\right)
\right|
&\le
	\frac{1}{(1-2|q|^2)^2}\left(
		\frac{1}{1-|q|^{u_1}}
		+	
		\frac{1}{1-|q|^{1-u_1}}
		+
		\frac{2|q|^{\frac{3}{2}} }{(1-|q|)^2}
	\right)
,\\
q^{-\frac{1}{2}(u_1+\frac{1}{2})^2}\mu\left(u_1\tau+u_2,-\tfrac{\tau}{2}+\tfrac{u_3}{4}\right)
&=
	-iq^{-\frac{u_1^2}{2} } \exp\left(\pi iu_2-\tfrac{\pi iu_3}{4}\right) 
	+
	E
,
\end{align*}
where
\begin{align*}
|E|
&\le
	\frac{1}{1-|q|^{u_1}}
	\left( 1 + \frac{|q|^{\frac{3}{2}}}{(1-2|q|^2)^2} \right)
	+
	\frac{1}{(1-|q|^{1-u_1})(1-2|q|^2)^2}
	+
	\frac{(1+|q|^{\frac{5}{2}})}{(1-2|q|^2)^2(1-|q|)^2}
.
\end{align*}
\end{proposition}

Using Proposition \ref{PropositionMuMainTerm3}, we deduce the transformation
and bounds needed in the case when $k\equiv 2\pmod{4}$. In particular, in 
this case, there is no contribution to the main term.

%Here we need Re(1/z)>k/2 so that the resulting q is small enough to apply
%the propositions with explicit error terms. Recall we needed small q to
%estimate by geometric series.
\begin{proposition}\label{PropFinalBounds2Mod4}
Suppose $a,c,k,n\in\mathbb{Z}$ with $c,k,n>0$ and $k\equiv 2\pmod{4}$,
$u=\frac{a}{c}$ with $2u\not\in\mathbb{Z}$, 
and $z\in\mathbb{C}$ with
$\RE{z}=\frac{k}{n}$ and
$\RE{\frac{1}{z}}>\frac{k}{2}$.
If $c\mid ka$, then
\begin{align*}
\left|R2\left(e^{2\pi i u}; e^{\frac{2\pi i(h+iz)}{k}}\right)\right|
&\le
	|\sin(\pi u)|\sqrt{\tfrac{n}{k}}\left(
		\frac{1.43}{\sin(\frac{\pi}{c})}
		+
		3.42\sqrt{k}
		+
		2.1
	\right)
.
\end{align*}
If $c\nmid ka$, then
\begin{align*}
\left| R2\left( e^{2\pi i u}; e^{\frac{2\pi i(h+iz)}{k}} \right) \right|
&\le
	|\sin(\pi u)| \sqrt{\tfrac{n}{k}}\left(
		\frac{5.72}{1-\exp\left(-\frac{\pi}{c}\right)} 
		+ 6.81\csc(\tfrac{\pi}{c})\sqrt{k}	
		+ 1.6 
	\right)
.
\end{align*}
\end{proposition}
\begin{proof}
As the proof is similar to that of Proposition \ref{PropFinalBounds0Mod4},
we are brief with the details.
We apply 
Corollary \ref{CorollaryMuTransformationWithBounds}
with $u\mapsto 2u$, $k\mapsto k/2$, $h\mapsto 2h$, $z\mapsto 2z$ and note that
$\wt{2h}=2$. We have that
$-\frac{1}{2}\le \Fractional{ku}-\frac{1}{2} \le \frac{1}{2}$
for all $k$ and $u$,
and 
$-\frac{1}{2}\le \Fractional{ku}+\frac{1}{2} \le \frac{1}{2}$
exactly when $c$ divides $ka$.
In writing out the result of applying Corollary \ref{CorollaryMuTransformationWithBounds},
we only give the case where $c$ does divide $ka$ and note the other case has a similar expression.
Here we find that
\begin{align*}
&R2\left(e^{2\pi i u}; e^{\frac{2\pi i(h+iz)}{k}}\right)
\\
&=
	\sqrt{2}\sin(\pi u) (-1)^{\Floor{ku} } 
	\xi\left(2h,[-2h]_{\frac{k}{2}},\tfrac{k}{2}\right)
	z^{-\frac{1}{2}}
	\exp\left( -\tfrac{2\pi i [-2h]_{\frac{k}{2}} \Floor{ku}^2}{k}   \right)
	\\&\quad\times
	\sum_{\pm}\mp 
	\exp\left(
		-\tfrac{\pi z}{4k} 
	 	+ \tfrac{2\pi i [-2h]_{\frac{k}{2}}}{k}\left( \Fractional{ku} \mp \tfrac{1}{2}\right)^2  
	\right)
	\\&\qquad\times
	\exp\left( 
		\pm \tfrac{\pi i \Floor{ku}}{k} \left(1+2h[-2h]_{\frac{k}{2}}-2[-2h]_{\frac{k}{2}}\right)  
		-
		\tfrac{2\pi i\left([-2h]_{\frac{k}{2}}+\frac{i}{2z}\right)}{k}\left( \Fractional{ku} \mp \tfrac{1}{2}\right)^2	
	\right)
	\\&\qquad\times
	\mu\left(
		\Fractional{ku}\tfrac{2\left([-2h]_{\frac{k}{2}}+\frac{i}{2z}\right)}{k}  -2u[-2h]_{\frac{k}{2}}, 
		\pm \tfrac{[-2h]_{\frac{k}{2}}+\frac{i}{2z}}{k} \mp \tfrac{1+2h[-2h]_{\frac{k}{2}}}{2k}; 
		\tfrac{2\left([-2h]_{\frac{k}{2}}+\frac{i}{2z}\right)}{k}
	\right)
	\\&\quad
	+
	E_4^{\prime}(u,h,k,z)
,
\end{align*}
where $|E_4^{\prime}(u,h,k,z)|\le 4|\sin(\pi u)E_2(2u,2h,k/2,2z)|$.
We apply Proposition \ref{PropositionMuMainTerm3}
with $u_1=\Fractional{ku}$, $u_2=-2u[-2h]_{k/2}$, 
$u_3 = \mp\frac{1+2h[-2h]_{k/2}}{k/2}$ and 
$\tau=\frac{2( [-2h]_{k/2}+\frac{i}{2z} )}{k}$. We note that 
$|q|<e^{-\pi}$. 

In the case that $c$ divides $ka$, we have that $u_1=0$ and so
\begin{align*}
\left|
\exp\left(
	- \tfrac{2\pi i\left([-2h]_{\frac{k}{2}}+\frac{i}{2z}\right)}{4k}	
\right)
\mu\left(
	-2u[-2h]_{\frac{k}{2}}, 
	\pm \tfrac{[-2h]_{\frac{k}{2}}+\frac{i}{2z}}{k} \mp \tfrac{1+2h[-2h]_{\frac{k}{2}}}{2k}; 
	\tfrac{2\left([-2h]_{\frac{k}{2}}+\frac{i}{2z}\right)}{k}
\right)
\right|
&\le
	\frac{0.504}{\sin(\frac{\pi}{c})} + 0.24
.
\end{align*}
From this, it follows that
\begin{align*}
\left|R2(e^{2\pi i u}; e^{\frac{2\pi i(h+iz)}{k}})\right|
&\le
	|\sin(\pi u)|\sqrt{\tfrac{n}{k}}\left(
		\frac{1.43}{\sin(\frac{\pi}{c})}
		+
		3.42\sqrt{k}
		+
		2.1
	\right)
.
\end{align*}

When $c$ does not divide $ka$, we instead have
\begin{gather*}
\left|
q^{-\frac{1}{2}(u_1-\frac{1}{2})^2}
\mu\left(u_1\tau+u_2,\tfrac{\tau}{2}+\tfrac{u_3}{4};\tau\right)
\right|
\le
	\frac{2.02}{1-|q|^{\frac{1}{c}}}+0.02
,\\
q^{-\frac{1}{2}(u_1+\frac{1}{2})^2}
\mu\left(u_1\tau+u_2,-\tfrac{\tau}{2}+\tfrac{u_3}{4};\tau\right)
=
	-i\exp\left(
		-\tfrac{2\pi i\left([-2h]_{\frac{k}{2}}+\frac{i}{2z} \right) \Fractional{ku}^2 }{k}
		-2\pi iu[-2h]_{\frac{k}{2}}
		-\tfrac{\pi i\left(1+2h[-2h]_{\frac{k}{2}}\right) }{2k}
	\right)
	\\\hspace{-10em}
	+E
,
\end{gather*}
where $|E|<\frac{2.02}{1-|q|^{\frac{1}{c}}}+1.103$.
Noting there is an extra term from Corollary 
\ref{CorollaryMuTransformationWithBounds} since
$\Fractional{ku}+\frac{1}{2}>\frac{1}{2}$, we find that
\begin{align*}
&R2(e^{2\pi i u}; e^{\frac{2\pi i(h+iz)}{k}})
\\
&=
	-i\sqrt{2}\sin(\pi u)(-1)^{\Floor{ku}}\xi\left(2h,[-2h]_{\frac{k}{2}},\tfrac{k}{2}\right)z^{-\frac{1}{2}}
	\exp\left(
		-\tfrac{\pi z}{4k} + \tfrac{\pi\Fractional{ku}^2}{kz}
	\right)
	\\&\quad\times
	\exp\left(
		\tfrac{2\pi i[-2h]_{\frac{k}{2}}}{k}\left(
			-\Floor{ku}^2 + \left(\Fractional{ku}+\tfrac{1}{2}\right)^2
			-h\Floor{ku} + \Floor{ku} - \Fractional{ku}^2 - ku -\tfrac{h}{2}
		\right)
		-\tfrac{\pi i(2\Floor{ku}+1)}{2k}
	\right)
	\\&\quad
	+
	i\sqrt{2}\sin(\pi u) (-1)^{\Floor{ku} } 
	\xi\left(2h,[-2h]_{\frac{k}{2}},\tfrac{k}{2}\right)
	z^{-\frac{1}{2}}
	\exp\left(
		-\tfrac{\pi z}{4k} + \tfrac{\pi\Fractional{ku}^2}{kz}
	\right)
	\\&\quad\times
	\exp\left(
		\tfrac{2\pi i[-2h]_{\frac{k}{2}}}{k}\left(
			-\Floor{ku}^2 - \Floor{ku} 
			-\tfrac{(2\Floor{ku}+1)(2h-2)}{4}		
			-\tfrac{1}{4}
		\right)
		-
		\tfrac{\pi i(2\Floor{ku}+1)}{2k}
	\right)
	+
	E_4(u,k,n)
\\
&=
	E_4(u,k,n)
.
\end{align*}
We see that we can bound $E_4(u,k,n)$ as
\begin{align*}
|E_4(u,k,n)|
&\le	
	|\sin(\pi u)| \sqrt{\tfrac{n}{k}}\left(
		\frac{5.72}{1-\exp\left(-\frac{\pi}{c}\right)} 
		+ 6.81\csc(\tfrac{\pi}{c})\sqrt{k}	
		+ 1.6 
\right)
.
\end{align*}
\end{proof}

We now consider the two $\mu$-functions appearing in the case when 
$k\equiv\pm1\pmod{4}$. The following proposition handles both of these
$\mu$-functions.
The proof is much the same as that of Propositions \ref{PropositionMuMainTerm1}
and \ref{PropositionMuMainTerm2},
and as such is omitted.

\begin{proposition}\label{PropositionMuMainTerm5}
Suppose $u_1,u_2\in\mathbb{R}$ with $0\le u_1<1$, $u_3$ is an odd integer, and
$|q|<\frac{1}{2}$.
If $u_1=0$, then
\begin{align*}
q^{-\frac{u_1^2}{2}} \mu\left( u_1\tau+u_2, \tfrac{u_3}{4};\tau	\right)
&=
	\frac{-iq^{-\frac{1}{8} } }{4\sin\left(\frac{\pi u_3}{4}\right) \sin(\pi u_2)}
	+
	E
,
\end{align*}
where
\begin{align*}
|E|
&\le
	\frac{\sqrt{2}|q|^{\frac{7}{8}}}{2(1-|q|)}
	\left( 1+\frac{1}{1-2|q|} \right)
	+
	\frac{\sqrt{2}|q|^{\frac{7}{8}}}{4(1-2|q|)|\sin(\pi u_2)|}
	+
	\frac{\sqrt{2}|q|^{\frac{7}{8}} (1+|q|) }
	{2(1-2|q|)(1-|q|)^2}
.
\end{align*}
If $u_1\not=0$, then
\begin{align*}
q^{-\frac{u_1^2}{2}} \mu\left( u_1\tau+u_2, \tfrac{u_3}{4};\tau	\right)
&=
	-\frac{1}{2\sin\left(\frac{\pi u_3}{4}\right)}
	\sum_{m=0}^{M_6} e^{\pi i(2m+1)u_2} q^{-\frac{u_1^2}{2}+\frac{u_1}{2}-\frac{1}{8}+mu_1}
	\\&\quad
	+
	\frac{i^{u_3}}{2\sin\left(\frac{\pi u_3}{4}\right)}
	\sum_{m=0}^{M_7} e^{-\pi i(2m+1)u_2} q^{-\frac{u_1^2}{2}-\frac{u_1}{2}+\frac{7}{8} +m(1-u_1)}
	+
	E
,
\end{align*}
where
$M_6=\left\lceil\frac{4u_1^2-12u_1+1}{8u_1}\right\rceil$,
$M_7=\left\lceil\frac{4u_1^2+12u_1-15}{8(1-u_1)}\right\rceil$, and
\begin{align*}
|E|
&\le
	\frac{\sqrt{2}}{2}\left(
		\frac{1}{1-|q|^{u_1}}
		+
		\frac{1}{1-|q|^{1-u_1}}
	\right)
	\left(	1 +	\frac{|q|^{\frac{7}{8}}}{1-2|q|}	\right)
	+
	\frac{\sqrt{2}|q|^{\frac{7}{8}}(1+|q|)}{2(1-2|q|)(1-|q|)^2}
.
\end{align*}
\end{proposition}

We now give the last proposition for a transformation formula and
bounds for $R2\big( e^{2\pi i u}; e^{\frac{2\pi i(h+iz)}{k}} \big)$.
This corresponds to the case when $k\equiv\pm1\pmod{4}$ and
uses Proposition \ref{PropositionMuMainTerm5}.

%Here we need Re(1/z)>k/2 so that the resulting q is small enough to apply
%the propositions with explicit error terms. Recall we needed small q to
%estimate by geometric series.
\begin{proposition}\label{PropFinalBoundsOdd}
Suppose $a,c,k,n\in\mathbb{Z}$ with $c,k,n>0$ and $k\equiv 1\pmod{2}$,
$u=\frac{a}{c}$ with $2u\not\in\mathbb{Z}$, 
and $z\in\mathbb{C}$ with
$\RE{z}=\frac{k}{n}$ and
$\RE{\frac{1}{z}}>\frac{k}{2}$.
If $c\mid 2ka$, then
\begin{align*}
&R2\left( e^{2\pi i u}; e^{\frac{2\pi i(h+iz)}{k}} \right)
\\
&=
	\frac{i (-1)^{2ku} \sin(\pi u) \cos( \pi u (1+h[-h]_{k}) ) } 
		{2 \sin\left( \frac{\pi k}{4} \right) \sin\left( \frac{\pi u[-h]_{k}}{2} \right) }
	z^{-\frac{1}{2}} \xi\left( 4h,\tfrac{[-h]_k}{4},k \right)
	\exp\left( 
		-\pi iku^2[-h]_k
		-\tfrac{\pi i[-h]_k}{16k}
		-\tfrac{\pi z}{4k} 
		+\tfrac{\pi}{16kz}
	\right)	
	\\&\quad
	+
	E_5(u,k,n)
,
\end{align*}
where
\begin{align*}
|E_5(u,k,n)|
&\le
	|\sin(\pi u)| \sqrt{\tfrac{n}{k}}
	\left( 4.04\csc(\tfrac{\pi}{c}) + 1515\sqrt{k}+ 55.86 \right)
.
\end{align*}
If $c\nmid 2ka$, then
\begin{align*}
&R2\left( e^{2\pi i u}; e^{\frac{2\pi i(h+iz)}{k}} \right)
\\
&=
	(-1)^{\Floor{2ku}}
	\sin(\pi u) 	
	z^{-\frac{1}{2}} 
	\xi\left(4h,\tfrac{[-h]_k}{4},k\right)
	\\&\quad\times
	\Bigg(
	-\frac{ \cos\left( \frac{\pi\Floor{2ku}k}{2} \right)} {\sin\left( \frac{\pi k}{4} \right)}	
	\sum_{m=0}^{M_6}
	\exp\left(
		\tfrac{\pi i[-h]_k}{4k} \left(
			-\Floor{2ku}^2
			-\Floor{2ku}(2m+1)
			-\tfrac{1}{4}
		\right)	
	\right)		
	\\&\qquad\times
	\exp\left(
		-
		\tfrac{\pi z}{4k} 
		+		
		\tfrac{\pi}{2kz}
		\left(
			\tfrac{\Fractional{2ku}^2}{2}-\tfrac{\Fractional{2ku}}{2}-m\Fractional{2ku}+\tfrac{1}{8}
		\right)
	\right)
	\\&\quad
	-
	i^{1+k}
	\frac{ \sin\left( \frac{ \pi\Floor{2ku}k }{2} \right)}
		{\sin\left( \frac{ \pi k}{4} \right)}
	\sum_{m=0}^{M_7} 
	\exp\left( 
		\tfrac{\pi i[-h]_k}{4k}\left( 
			-\Floor{2ku}^2
			+\Floor{2ku}(2m+1)
			+2m+\tfrac{7}{4}
		\right)
	\right)
	\\&\qquad\times
	\exp\left(
		-
		\tfrac{\pi z}{4k} 
		+		
		\tfrac{\pi}{2kz}
		\left(
			\tfrac{\Fractional{2ku}^2}{2}+\tfrac{\Fractional{2ku}}{2}-m(1-\Fractional{2ku})-\tfrac{7}{8} 
		\right)
	\right)
	\\&\quad
	+
	\delta_3 2
	\sin\left( 	\tfrac{\pi(2\Floor{2ku}+1)k}{4} \right)
	\exp\left(
		\tfrac{\pi i[-h]_k}{4k}\left( -\Floor{2ku}^2-\Floor{2ku} -\tfrac{1}{4} \right)
		-
		\tfrac{\pi z}{4k}
		+
		\tfrac{\pi}{4kz}\left(\Fractional{2ku}-\tfrac{1}{2}\right)^2
	\right)
	\Bigg)
	+	
	E_5(u,k,n)
,
\end{align*}
where 
$M_6=\left\lceil\frac{4\Fractional{2ku}^2-12\Fractional{2ku}+1}{8\Fractional{2ku}}\right\rceil$,
$M_7=\left\lceil\frac{4\Fractional{2ku}^2+12\Fractional{2ku}-15}{8(1-\Fractional{2ku})}\right\rceil$, 
\begin{align*}
\delta_3
&=
\begin{cases}
	1	& \text{if } \frac{1}{2} < \Fractional{2ku},
	\\
	0	& \text{else}.
\end{cases}
\end{align*}
and
\begin{align*}
|E_5(u,k,n)|
&\le
	|\sin(\pi u)| \sqrt{\tfrac{n}{k}} 
	\left(	
		\frac{18.98}{1-\exp\left(-\frac{\pi}{4c}\right)}
		+
		1515\csc(\tfrac{\pi}{2c})\sqrt{k}
		+
		4.83\sqrt{k}
		+  
		40.72		
		\right)
.
\end{align*}
\end{proposition}
\begin{proof}
As with the proofs of Propositions \ref{PropFinalBounds0Mod4} and 
\ref{PropFinalBounds2Mod4}, we begin with an application of 
Corollary \ref{CorollaryMuTransformationWithBounds}.
Here we use $u\mapsto 2u$, $h\mapsto 4h$, $z\mapsto 4z$ and note that
$\wt{4h}=0$ and $[-4h]_k=\frac{[-h]_k}{4}$. The proposition then
follows from Proposition \ref{PropositionMuMainTerm5} and a lengthy
calculation. 
\end{proof}

Lastly we require a proposition to rewrite the terms corresponding
to negative powers of $q$ as Bessel functions. As mentioned in Section 3,
this process is well known. However, we require a version with the 
error terms bounded explicitly. We give this is the following proposition.

\begin{proposition}\label{MainTermFromIntegrals}
Suppose
$\frac{h_0}{k_0},\frac{h}{k},\frac{h_1}{k_1}$ are three
consecutive Farey fractions of level $N=\lfloor\sqrt{n}\rfloor$,
$r>0$,
and $z=\frac{k}{n}-ik\Phi$. Then for $n\ge 2$,
\begin{align*}
\int_{-\vartheta^{\prime}_{h,k}}^{\vartheta^{\prime\prime}_{h,k}}
	z^{-\frac{1}{2}}
	\exp\left( \tfrac{2\pi nz}{k}-\tfrac{\pi z }{4k} + \tfrac{\pi r}{kz}\right)
	d\Phi
&=
	\frac{4}{k^{\frac{1}{2}}(8n-1)^{\frac{1}{2}}}
	\cosh\left(\frac{\pi}{k}\sqrt{r(8n-1)}\right)
	+
	E,
\end{align*}
where 
\begin{align*}
|E|
&\le
	\frac{\sqrt{2}\left(1+6e^{2\pi(1+2r)} \right)}{3n^{\frac34}}
	+
	\frac{2(2-\sqrt{2})}{kn^{\frac{1}{4}}}
.
\end{align*}
\end{proposition}
\begin{proof}
We note that
\begin{align*}
\int_{-\vartheta^{\prime}_{h,k}}^{\vartheta^{\prime\prime}_{h,k}}
	z^{-\frac{1}{2}}
	\exp\left(\tfrac{2\pi nz}{k}-\tfrac{\pi z }{4k} + \tfrac{\pi r}{kz}\right)
	d\Phi
&=
	\frac{1}{ik}
	\int_{\frac{k}{n}-\frac{i}{k_1+k}}^{\frac{k}{n}+\frac{i}{k_0+k}}
	z^{-\frac{1}{2}}
	\exp\left(\tfrac{2\pi nz}{k}-\tfrac{\pi z }{4k} + \tfrac{\pi r}{kz}\right)
	dz	
.
\end{align*}

We use a well known integral representation of $I_v(u)$ \cite[p. 181]{Watson2} given by
\begin{align*}
I_v(u)
&=
	\frac{(u/2)^v}{2\pi i} \int_{-\infty}^{(0+)}
	t^{-v-1} \exp\left(t+\tfrac{u^2}{4t}\right) dt.
\end{align*}
Here the path of integration starts just below the negative real axis, loops
counterclockwise around the origin, and ends just above the negative real axis.
Using the change of variable $t=\frac{z\pi(8n-1)}{4k}$, this becomes
\begin{align*}
I_v(u)
&=
	\frac{1}{2\pi i}
	\left(\frac{2ku}{\pi(8n-1)}\right)^v
	\int_{-\infty}^{(0+)}
	z^{-v-1}
	\exp\left(\tfrac{2\pi n}{k}-\tfrac{\pi z}{4k}+\tfrac{ku^2}{z\pi(8n-1)}\right)
	dz
.
\end{align*}
Setting $v=-\frac{1}{2}$ and $u=\frac{\pi}{k}\sqrt{r(8n-1)}$ yields
\begin{align*}
I_{-\frac{1}{2}}\left(\frac{\pi}{k}\sqrt{r(8n-1)}\right)	
&=
	-\frac{i\sqrt{2}(8n-1)^{\frac{1}{4}}}{4\pi r^{\frac{1}{4}}}
	\int_{-\infty}^{(0+)}
	z^{-\frac{1}{2}}
	\exp\left(\tfrac{2\pi n}{k}-\tfrac{\pi z}{4k}+\tfrac{\pi r}{kz}\right)
	dz
.
\end{align*}
As such,
\begin{align}\label{EqBesselFunctionIntegral}
\int_{-\infty}^{(0+)}
	z^{-\frac{1}{2}}
	\exp\left(\tfrac{2\pi n}{k}-\tfrac{\pi z}{4k}+\tfrac{\pi r}{kz}\right)
	dz
&=
	\frac{2\sqrt{2} \pi i r^{\frac{1}{4}}}{(8n-1)^{\frac{1}{4}}}
	I_{-\frac{1}{2}}\left(\frac{\pi}{k}\sqrt{r(8n-1)}\right)	
.
\end{align}
By Cauchy's theorem, we can alter the path of integration 
in \eqref{EqBesselFunctionIntegral}
to the path indicated 
in Figure \ref{Figure1} (noting that $k_i+k\le 2N$).
\begin{figure}[h]\caption{} \label{Figure1}
\begin{tikzpicture}[scale=1,>=stealth]
	\draw[dashed] (-5,0) -- (3,0);
	\draw[dashed] (0,2.5) -- (0,-2.5);
	\node at (0.15,0.15) {$0$};

	\draw (-5,-0.5)--(0,-0.5);	
	\node at (-2.5,-0.5) [inner sep=-1,label=below:$L_1$] {$>$};
	\node [circle,fill,inner sep=1] at (0,-0.5) [label=below left:$\scriptstyle-\frac{i}{2N}$] {};		

	\draw (0,-0.5)--(0,-1.75);	
	\node at (0,-2.25/2) [inner sep=-1, label=right:$L_2$] {$\vee$};
	\node [circle,fill,inner sep=1] at (0,-1.75) [label=below left:$\scriptstyle-\frac{i}{k_1+k}$] {};		
	
	\draw (0,-1.75)--(1.5,-1.75);	
	\node at (0.75,-1.75) [inner sep=-1, label=below:$L_3$] {$>$};
	\node [circle,fill,inner sep=1] at (1.5,-1.75) [label=below right:$\scriptstyle\frac{k}{n}-\frac{i}{k_1+k}$] {};		

	\draw (1.5,-1.75)--(1.5,1.75);
	\node at (1.5,0) [inner sep=-1, label=right:$L_4$] {$\wedge$};

	\draw (1.5,1.75)--(0,1.75);	
	\node at (0.75,1.75) [inner sep=-1, label=above:$L_5$] {$<$};
	\node [circle,fill,inner sep=1] at (1.5,1.75) [label=above right:$\scriptstyle\frac{k}{n}+\frac{i}{k_0+k}$] {};		

	\draw (0,1.75)--(0,0.5);	
	\node at (0,2.25/2) [inner sep=-1, label=right:$L_6$] {$\vee$};
	\node [circle,fill,inner sep=1] at (0,1.75) [label=above left:$\scriptstyle\frac{i}{k_0+k}$] {};		

	\draw (0,0.5)--(-5,0.5);	
	\node at (-2.5,0.5) [inner sep=-1, label=above:$L_7$] {$<$};
	\node [circle,fill,inner sep=1] at (0,0.5) [label=above left:$\scriptstyle\frac{i}{2N}$] {};		
\end{tikzpicture}
\end{figure}

Since
\begin{align*}
\int_{\frac{k}{n}-\frac{i}{k_1+k}}^{\frac{k}{n}+\frac{i}{k_0+k}}
	z^{-\frac{1}{2}}
	\exp\left(\tfrac{2\pi nz}{k}-\tfrac{\pi z }{4k} + \tfrac{\pi r}{kz}\right)
	dz	
&=
\int_{L_4}
	z^{-\frac{1}{2}}
	\exp\left( \tfrac{2\pi nz}{k}-\tfrac{\pi z }{4k} + \tfrac{\pi r}{kz}\right)
	dz	
,
\end{align*}
we bound the integrals over the remaining line segments. As these bounds are
only max-length estimates and evaluations of elementary integrals, we simply
state them. The following bounds hold,
\begin{align*}
\left|\int_{L_1}
	z^{-\frac{1}{2}}
	\exp\left(\tfrac{2\pi nz}{k}-\tfrac{\pi z }{4k} + \tfrac{\pi r}{kz}\right)
	dz	
\right|
&\le 
	\frac{\sqrt{2}k}{6n^{\frac34}}
,\\
\left|\int_{L_2}
	z^{-\frac{1}{2}}
	\exp\left(\tfrac{2\pi nz}{k}-\tfrac{\pi z }{4k} + \tfrac{\pi r}{kz}\right)
	dz	
\right|
&\le	(2-\sqrt{2})n^{-\frac{1}{4}}
,\\
\left|\int_{L_3}
	z^{-\frac{1}{2}}
	\exp\left(\tfrac{2\pi nz}{k}-\tfrac{\pi z }{4k} + \tfrac{\pi r}{kz}\right)
	dz	
\right|
&\le
	\frac{\sqrt{2} e^{2\pi(1+2r)} k}{n^{\frac{3}{4}}}
.
\end{align*}
We note that we obtain the same bound on $L_7$ as $L_1$,
on $L_6$ as $L_2$, and on $L_5$ as $L_3$.
Thus
\begin{align*}
\int_{-\vartheta^{\prime}_{h,k}}^{\vartheta^{\prime\prime}_{h,k}}
	z^{-\frac{1}{2}}
	\exp\left(\tfrac{2\pi nz}{k}-\tfrac{\pi z }{4k} + \tfrac{\pi r}{kz}\right)
	d\Phi
&=
	\frac{1}{ki}
	\int_{\frac{k}{n}-\frac{i}{k_1+k}}^{\frac{k}{n}+\frac{i}{k_0+k}}
	z^{-\frac{1}{2}}
	\exp\left(\tfrac{2\pi nz}{k}-\tfrac{\pi z }{4k} + \tfrac{\pi r}{kz}\right)
	dz
\\
&=
	\frac{2\sqrt{2}\pi r^{\frac{1}{4}}}{k(8n-1)^{\frac{1}{4}}}
	I_{-\frac{1}{2}}\left(\frac{\pi}{k}\sqrt{r(8n-1)}\right)
	+
	E,
\end{align*}
where 
\begin{align*}
|E|
&\le
	\frac{\sqrt{2}}{3n^{\frac34}}
	+
	\frac{2(2-\sqrt{2})}{kn^{\frac{1}{4}}}
	+
	\frac{2\sqrt{2} e^{2\pi(1+2r)} }{n^{\frac{3}{4}}}
.
\end{align*}
To finish the proof we use that half order Bessel functions can be 
expressed in terms of trigonometric functions \cite[page 80]{Watson2}. In 
particular,
$I_{-\frac{1}{2}}(u)=\sqrt{\frac{2}{\pi u}}\cosh(u)$.
\end{proof}

\section{Proof of Theorem \ref{TheoremAsymptoticForM2Rank} }

\begin{proof}
We recall
\begin{align*}
A\left(\frac{a}{c};n\right)
&=
\sum_{\substack{0\leq h<k\leq N\\(h,k)=1}}
\exp\left(-\tfrac{2\pi ihn}{k}\right)
\int_{-\vartheta^{\prime}_{h,k}}^{\vartheta^{\prime\prime}_{h,k}}
R2\left(e^{\frac{2\pi ia}{c}}; e^{\frac{2\pi i(h+iz)}{k}}\right)
\exp\left(\tfrac{2\pi nz}{k}\right) d\Phi,
\end{align*}
where $z=\frac{k}{n}-ik\Phi$ and $N=\Floor{\sqrt{n}}$.
We let $\Sigma_0$, $\Sigma_1$, and $\Sigma_2$ denote the sums when 
$k\equiv 0$, $\pm 1$, and $2$ modulo $4$ respectively. Next we work out their
contributions to the main term.

We take the main terms arising from
Propositions \ref{PropFinalBounds0Mod4} and \ref{PropFinalBoundsOdd},
which are of the form
\begin{align*}
\sum_{\substack{0\leq h<k\leq N\\(h,k)=1,\\k\equiv j\pmod{4}}}
\exp\left(-\tfrac{2\pi ihn}{k}\right) 
\int_{-\vartheta^{\prime}_{h,k}}^{\vartheta^{\prime\prime}_{h,k}}
	z^{-\frac12}
	\exp\left(\tfrac{2\pi nz}{k}-\tfrac{\pi z}{4k} +\tfrac{\pi r}{kz}\right)
	d\Phi,
\end{align*}
and evaluate them with Proposition \ref{MainTermFromIntegrals}.
We omit most of the details, as they are little more than copying the statements
of Propositions \ref{PropFinalBounds0Mod4} and \ref{PropFinalBoundsOdd}
with $u$ replaced by $\frac{a}{c}$. However, we do briefly explain
where each main term comes from. 

When $k\equiv 0\pmod{4}$, we apply Proposition \ref{PropFinalBounds0Mod4}.
The case when $c\mid \frac{ka}{2}$ gives the $C_{0,a,c,k,n}$ term.
The case when $c\nmid \frac{ka}{2}$ gives
the $D_{j,a,c,k,n}(m)$ terms (for $1\le j\le 5$) from the
sums with $M_j$, and the $\delta_1$ and $\delta_2$ terms combine to
give the $D^{+}_{a,c,k,n}$ and $D^{-}_{a,c,k,n}$ terms.
When $k\equiv 1\pmod{2}$, we apply Proposition \ref{PropFinalBoundsOdd}.
The case when $c\mid 2ka$ gives the $C_{1,a,c,k,n}$ term.
The case when $c\nmid 2ka$ gives
the $D_{j,a,c,k,n}(m)$ terms (for $j=6,7$) from the
sums with $M_j$, and the $\delta_3$ term
gives the $D_{a,c,k,n}$ term. We note there are no main terms
corresponding to $k\equiv 2\pmod{4}$.

It then only remains to obtain an explicit bound on the error terms.
We note the following bounds hold:
\begin{align*}
\sum_{\substack{1\le k\le N,\\ k\equiv 0\pmod{m}}}\frac{1}{\sqrt{k}}	
	&\le \frac{2\sqrt{N}}{m},&
\sum_{\substack{1\le k\le N,\\ k\not\equiv 0\pmod{m}}}\frac{1}{\sqrt{k}}	
	&\le \frac{2(m-1)\sqrt{N}}{m},
\\
\sum_{\substack{1\le k\le N,\\ k\equiv 0\pmod{m}}}k	
	&\le \frac{N(N+m)}{2m},&
\sum_{\substack{1\le k\le N,\\ k\not\equiv 0\pmod{m}}} k	
	&\le \frac{(m-1)(N+1)^2}{2m}
.
\end{align*}
In simplifying the error terms, we assume that $a$ and $c$ are relatively prime.
This assumption does affect the explicit constants, but does not
change the big $O$ term. We begin with $\Sigma_2$, as it is the simplest.

Applying Proposition \ref{PropFinalBounds2Mod4} leads to
\begin{align*}
\left| {\sum}_2 \right|
&\le
	\sum_{\substack{0\le h<k\le N,\\(h,k)=1,\\k\equiv 2\pmod{4}}}
	\int_{-\vartheta^\prime_{h,k}}^{\vartheta^{\prime\prime}_{h,k}}
	\left| R2\left(e^{\frac{2\pi ia}{c}} ; e^{\frac{2\pi i(h+iz)}{k}} \right)\right|
	\exp\left( \tfrac{2\pi n}{k}\RE{z}\right)
	d\Phi
\\
&\le
	2e^{2\pi} \left|\sin\left(\frac{\pi a}{c}\right)\right|
	\sum_{\substack{1\le k\le N,\\ k\equiv 2\pmod{4},\\ c\mid ka}}
	\frac{1}{\sqrt{k}}	
	\left(
		\frac{1.43}{\sin\left(\frac{\pi}{c}\right)}	
		+ 3.42\sqrt{k}
		+ 2.1
	\right)
	\\&\quad
	+
	2e^{2\pi} \left|\sin\left(\frac{\pi a}{c}\right)\right|
	\sum_{\substack{1\le k\le N,\\ k\equiv 2\pmod{4},\\ c\nmid ka}}
	\frac{1}{\sqrt{k}}
	\left(
		\frac{5.72}{1- \exp\left(-\frac{\pi}{c}\right) }	
		+ \frac{6.81\sqrt{k}}{\sin\left(\frac{\pi}{c}\right)}
		+ 1.6
	\right)
\\
&\le
	2e^{2\pi} \left|\sin\left(\frac{\pi a}{c}\right)\right|
	\left(
		\left(3.42 + \frac{6.81}{\sin\left(\frac{\pi}{c}\right)}\right)n^{\frac12}
		+
		\frac{2}{c}\left(
			2.1 + 1.6(c-1) + \frac{1.43}{\sin\left(\frac{\pi}{c}\right)}
			+ \frac{5.72(c-1)}{1-\exp\left(-\frac{\pi}{c}\right) }
		\right)n^{\frac14}
	\right)
.
\end{align*}

To bound the error term from $\Sigma_1$, we must determine an explicit 
upper bound on the number of summands appearing in the sums with index bounds
$M_6$ and $M_7$ in Proposition \ref{PropFinalBoundsOdd}. This amounts to only an
elementary calculus exercise. The number of terms appearing in the
sums with $M_6$ and $M_7$ are bounded above, respectively, by
$B_{6,c}$ and $B_{7,c}$, where 
\begin{align*} 
B_{6,c} &:= \left\lceil \frac{c^2-4c+4}{8c}  \right\rceil
,&
B_{7,c} &:= \begin{cases}
0 &\mbox{ if } c<12
,\\
\left\lceil \frac{c^2-12c+4}{8c}  \right\rceil &\mbox{ if } c\geq 12. 
\end{cases}
\end{align*}
Using this, along with the fact that $r_{j,a,c,k}(m)\leq \frac{1}{16}$ for $j=6,7$,
we deduce (after a very lengthy calculation) that the error term from $\Sigma_1$ is
bounded by
\begin{align*}
&	\frac{\left|\sin\left(\frac{\pi a}{c}\right)\right| (c_1-1) \left(1+6\exp\left(\frac{9\pi}{4}\right) \right) 
		(1+B_{6,c}+B_{7,c})}{3c_1}
	n^{-\frac{3}{4}}	
	\\&
	+
	\left|\sin\left(\frac{\pi a}{c}\right)\right| \left(1+6\exp\left(\tfrac{9\pi}{4}\right) \right)
	\left(
		\frac{1}{6 \sin\left(\frac{\pi}{2c}\right)}
		+\frac{2(c_1-1)(1+B_{6,c}+B_{7,c})}{3c_1}
	\right) 
	n^{-\frac{1}{4}}	
	\\&
	+
	\frac{\left|\sin\left(\frac{\pi a}{c}\right)\right| }{c_1}
	\Bigg(
		\left(1+6\exp\left(\tfrac{9\pi}{4}\right) \right)\left(
			\frac{1}{6 \sin\left(\frac{\pi}{2c}\right)}
			+\frac{(c_1-1)(1+B_{6,c}+B_{7,c})}{3}
		\right)
		\\&\quad
		+
		\sqrt{2}(2-\sqrt{2})\left(
			\frac{1}{\sin\left(\frac{\pi}{2c}\right)}
			+2(c_1-1)(1+B_{6,c}+B_{7,c})
		\right)
		\\&\quad
		+
		2e^{2\pi} \left(
			\frac{8.08}{\sin\left(\frac{\pi}{c}\right)}
			+ 111.72
			+ \frac{37.96(c_1-1)}{1-\exp\left(-\frac{\pi}{4c}\right) }
			+ 81.44(c_1-1)
		\right)
	\Bigg)
	n^{\frac{1}{4}}
	\\&
	+
	\frac{\left|\sin\left(\frac{\pi a}{c}\right)\right| 2e^{2\pi} }{c_1}\left(
		1515 + \frac{1515(c_1-1)}{\sin\left(\frac{\pi}{2c}\right)}
		+ 4.83(c_1-1)
	\right)
	n^{\frac{1}{2}}	
.
\end{align*}
with $c_1:=\frac{c}{(2,c)}$.

To bound the error from $\Sigma_0$, we first note that the number of terms 
appearing in the sums with the $M_i$ are bounded above by $B_{i,c}$, with
\begin{align*}
B_{1,c} &:= \left\lceil\frac{c^2-24c+16}{32c}\right\rceil
,
&B_{2,c} &:= \left\lceil\frac{c^2-40c+16}{32c}\right\rceil
,
&B_{3,c} &:= 
	\begin{cases}
	1 & \mbox{ if } c=4,\\ 
	\left\lceil\frac{c^2-8c+16}{32c}\right\rceil & \mbox{ if } c\not=4,
	\end{cases}
\\
B_{4,c} &:= \left\lceil\frac{c^2-24c+16}{32c(c-1)}\right\rceil
,
&B_{5,c} &:= 
	\begin{cases}
	0 & \mbox{ if } c<24,\\ 
	\left\lceil\frac{c^2-56c+16}{32c}\right\rceil & \mbox{ if } c>24.
	\end{cases}
&&
\end{align*}
After another long calculation, we find that the error term from $\Sigma_0$ 
is bounded by
\begin{align*}
&	\frac{\left|\sin\left(\frac{\pi a}{c}\right)\right| \sqrt{2}(c-1)
		\left(
			(1+B_{1,c}+B_{2,c}+B_{4,c}+B_{5,c})(1+6e^{3\pi} )
			+ (1+B_{3,c})(1+6e^{11\pi} )
		\right)
	}{3c}
	n^{-\frac34}	
	\\&
	+
	\left|\sin\left(\tfrac{\pi a}{c}\right)\right|\sqrt{2}
	\Bigg(
		\frac{1+6e^{3\pi} }{6\sin\left(\frac{\pi}{c}\right)}	
		+\frac{(1+B_{1,c}+B_{2,c}+B_{4,c}+B_{5,c})2(c-1)(1+6e^{3\pi} )}{3c}
		\\&\quad
		+\frac{(1+B_{3,c})2(c-1)(1+6e^{11\pi} )}{3c}
	\Bigg)	
	n^{-\frac14}
	\\&
	+
	\frac{\left|\sin\left(\frac{\pi a}{c}\right)\right|}{c}
	\Bigg(
		4e^{2\pi}\left( \frac{0.208}{\sin\left(\frac{\pi}{c}\right)} + 0.288 \right)	
		+
		4e^{2\pi}(c-1)\left( 
			\frac{10.6}{1-\exp\left(-\frac{4\pi}{c}\right)} + 0.0062 + \frac{2.962}{\sin\left(\frac{\pi}{2c}\right)} 
		\right)
		\\&\quad
		+
		\sqrt{2}(1+6e^{3\pi})\left(
			\frac{1}{6\sin\left(\frac{\pi}{c}\right)}
			+ \frac{(1+B_{1,c}+B_{2,c}+B_{4,c}+B_{5,c})(c-1)}{3}
		\right)
		+
		\frac{ \sqrt{2}(1+6e^{11\pi}) (1+B_{3,c}) (c-1) }{3}
		\\&\quad
		+
		2(2-\sqrt{2})\left(
			\frac{1}{\sin\left(\frac{\pi}{c}\right)}
			+2(c-1)(2+B_{1,c}+B_{2,c}+B_{3,c}+B_{4,c}+B_{5,c})
		\right)
	\Bigg)
	n^{\frac14}
	\\&
	+
	\frac{\left|\sin\left(\frac{\pi a}{c}\right)\right| 2e^{2\pi} }{c}
	\left(
		1.49+\frac{4.01(c-1)}{\sin\left(\frac{\pi}{2c}\right)}
	\right)		
	n^{\frac12}	
.
\end{align*}

We now see that the error term in Theorem \ref{TheoremAsymptoticForM2Rank} is
indeed $O(\sqrt{n})$. Furthermore, for $\gcd(a,c)=1$ and $n\ge2$, and explicit
upper bound on the size of the error is
\begin{align}\label{EqFinalBoundOnErrorTerm}
\left|\sin\left(\frac{\pi a}{c}\right)\right|\left(
a_1n^{-\frac34}+a_2n^{-\frac14}+a_3n^{\frac14}+a_4n^{\frac12}
\right),
\end{align}
where
\begin{align*}
a_1 
&:=
	\frac{2350(c_1-1)(1+B_{6,c}+B_{7,c})}{c_1}	
	+
	\frac{35050(c-1)(1+B_{1,c}+B_{2,c}+B_{4,c}+B_{5,c})}{c}	
	\\&\quad
	+
	\frac{2.9\cdot10^{15}(c-1)(1+B_{3,c})}{c}	
,\\
a_2 
&:=
	\frac{1175}{\sin\left(\frac{\pi}{2c}\right)}
	+
	\frac{17525}{\sin\left(\frac{\pi}{c}\right)}
	+
	\frac{4699(c_1-1)(1+B_{6,c}+B_{7,c})}{c_1}	
	+
	\frac{70099(c-1)(1+B_{1,c}+B_{2,c}+B_{4,c}+B_{5,c})}{c}	
	\\&\quad
	+
	\frac{5.8\cdot10^{15}(c-1)(1+B_{3,c})}{c}	
,\\
a_3
&:=	
	\frac{5116}{c}
	+ 
	\frac{1.2\cdot 10^5}{c_1}
	+ 
	\frac{2.9\cdot 10^{15}(c-1)}{c}
	+ 
	\frac{89572(c_1-1)}{c_1}
	+ 
	\frac{35052(c-1)(B_{1,c}+B_{2,c}+B_{4,c}+B_{5,c})}{c}
	\\&\quad
	+
	\frac{ 2.9\cdot 10^{15} (c-1)B_{3,c} }{c}
	+
	\frac{2351(c_1-1)(B_{6,c}+B_{7,c})}{c_1}
	+
	\frac{21035}{c\sin\left(\frac{\pi}{c}\right)}
	+
	\frac{8654}{c_1\sin\left(\frac{\pi}{c}\right)}
	+
	\frac{1176}{c_1 \sin\left(\frac{\pi}{2c}\right)}
	+
	\frac{6345(c-1)}{c\sin\left(\frac{\pi}{2c}\right)} 
	\\&\quad	
	+ 
	\frac{22705(c-1)}{c\left(1-\exp\left(-\frac{4\pi}{c}\right)\right)} 
	+
	\frac{12253 (c-1)}{c\left(1-\exp\left(-\frac{\pi}{c}\right)\right)}
	+ 
	\frac{40655(c_1-1)}{c_1\left(1-\exp\left(-\frac{\pi}{4c}\right)\right)}
,\\
a_4 
&:= 
	3663 
	+ \frac{1596}{c}
	+\frac{1.7\cdot 10^6}{c_1}
	+ \frac{5173(c_1-1)}{c_1}
	+\frac{7294}{\sin\left(\frac{\pi}{c}\right)}
	+\frac{4295(c-1)}{c\sin\left(\frac{\pi}{2c}\right)}
	+\frac{1.7\cdot 10^6(c_1-1)}{c_1\sin\left(\frac{\pi}{2c}\right)}
.
\end{align*}

As a further bound of use, we find that if $r_1$ and $r_2$ are the first and
second largest values of $r$, with $
\frac{\cosh\left(\pi r\sqrt{8n-1}\right)}{\sqrt{8n-1}}$ appearing with non-zero
coefficient in the main term, then the contributions to the main term other than
from $r_1$ may be bounded by
\begin{align*}
&\frac{8\left(5+B_{1,c}+B_{2,c}+B_{3,c}+B_{4,c}+B_{5,c}+B_{6,c}+B_{7,c}\right)
\cosh\left(\pi r_2\sqrt{8n-1}\right)}{\sqrt{8n-1}}
\sum_{k=1}^{N}\sqrt{k}
\\
&\le
\frac{16\left(5+B_{1,c}+B_{2,c}+B_{3,c}+B_{4,c}+B_{5,c}+B_{6,c}+B_{7,c}\right)
	\left(\sqrt{n}+1\right)^{\frac{3}{2}} \cosh\left(\pi r_2\sqrt{8n-1}\right)}
{3\sqrt{8n-1}}
.
\end{align*}

\end{proof}

\section{A Few Inequalities}

We let $N2(r,m,n)$ denote the number of partitions of $n$ without repeated odd
parts and with $M_2$-rank congruent to $r$ modulo $m$. For convenience we set
$\zeta_c:=\exp\left(\frac{2\pi i}{c}\right)$.
It is not difficult to
see $N2(r,m,n)$ is relevant to our study of $R2(\zeta_c^a;q)$, due
to the fact that
\begin{align*}
R2\left(\zeta_c^a;q\right)
&=
\sum_{n=0}^\infty \sum_{r=0}^{c-1} N2(r,c,n)\zeta_c^{ar} q^n.
\end{align*}
With partitions ranks, such as $R(\zeta;q)$ and $R2(\zeta;q)$,
a common point of study is dissection formulas for $R(\zeta_c;q)$
(or in our case $R2(\zeta_c;q)$) and the equivalent identities
among the $N(r,c,n)$ (or $N2(r,c,n)$). 

For example, 
$N(r,5,5n+4)=\frac{p(5n+4)}{5}$ follows from showing that
\begin{align}\label{EqRankThing1}
\sum_{n\ge0}\left(N(r,5,5n+4)-N(0,5,5n+4)\right)q^n&=0,
\end{align}
for $1\le r\le 4$. By using the minimal polynomial for $\zeta_5$
and the linear independence over $\mathbb{Q}$ of
$\zeta_5$, $\zeta_5^2$, $\zeta_5^3$, $\zeta_5^4$ one finds 
\eqref{EqRankThing1} is equivalent to determining that the 
$q^{5n+4}$ terms of $R(\zeta_5;q)$ are all zero.
For the rank function $R(z;q)$, identities equivalent to the
$5$-dissection of $R(\zeta_5;q)$ and the $7$-dissection of
$R(\zeta_7;q)$ were established by Atkin and Swinnerton-Dyer 
\cite{AtkinSwinnertonDyer1} to prove Dyson's conjectures on
$N(r,5,5n+4)$ and $N(r,7,7n+5)$. Identities equivalent to the $3$-dissection
of $R2(\zeta_3;q)$ and the $5$-dissection of $R2(\zeta_5;q)$ were given by
Lovejoy and Osburn \cite{LovejoyOsburn1} and identities equivalent to the
$3$-dissection of $R2(\zeta_6;q)$ and the $5$-dissection of $R2(\zeta_{10};q)$
were given by Mao \cite{Mao1}. 

In some cases, one can deduce inequalities from such dissection formulas. For example,
one of the formulas from \cite{LovejoyOsburn1} is
\begin{align*}
\sum_{n\ge0} \left(N2(0,3,3n+1)-N2(1,3,3n+1)\right) q^{3n}
&=
\frac{\aqprod{-q^3,q^6}{q^6}{\infty}}{\aqprod{q^2,q^4}{q^6}{\infty}}.
\end{align*}
One finds, with the assistance of the $q$-binomial theorem 
\cite[Theorem 2.1]{Andrews1}, that the above product has non-negative
coefficients when viewed as a series in $q$. As such it must be that
$N2(0,3,3n+1)\ge N2(1,3,3n+1)$ for $n\ge0$. This is one of the many inequalities 
Mao established for $N2(r,m,n)$ in \cite{Mao1}. Among these inequalities are
$N2(0,6,3n+j)+N2(1,6,3n+j)> N2(2,6,3n+j)+N2(2,6,3n+j)$ for $j=0,1$;
$N2(0,10,5n+j)+N2(1,10,5n+j)> N2(4,10,5n+j)+N2(5,10,5n+j)$ for $j=1,2,3$;
and
$N2(1,10,5n+j)+N2(2,10,5n+j)\ge N2(3,10,5n+j)+N2(4,10,5n+j)$ for $j=1,3,4$.
Mao conjectured additional inequalities, which we rephrase as
\begin{align}\label{EqMaoConjecture1}
N2(0,6,n)+N2(1,6,n)&> N2(2,6,n)+N2(3,6,n) &\mbox{for }n\ge0
,\\\label{EqMaoConjecture2}
N2(0,10,n)+N2(1,10,n)&> N2(4,10,n)+N2(5,10,n) &\mbox{for }n\ge0
,\\\label{EqMaoConjecture3}
N2(1,10,n)+N2(2,10,n)&> N2(3,10,n)+N2(4,10,n) &\mbox{for }n\ge3.
\end{align}
The restricted cases of \eqref{EqMaoConjecture1} when $n=9m+5$ and $n=9m+8$
were proved by Barman and Pal Singh Sachdevain in \cite{BarmanSachdeva1}, 
and \eqref{EqMaoConjecture2} was 
fully resolved in \cite{AlwaiseIannuzziSwisher1}
by Alwaise, Iannuzzi, and Swisher.
Shortly we will prove inequalities \eqref{EqMaoConjecture1} and  \eqref{EqMaoConjecture3}
(along with several others) by finding they hold asymptotically and then
verifying the inequality for a suitable number of initial values of $n$.

To begin, we note the following identities hold, all of which follow from the
standard properties of roots of unity and the fact that $N2(r,m,n)=N2(m-r,m,n)$.
In each case $\zeta_c^a$ is assumed to be a primitive $c$-th root of unity,
that is to say $a$ and $c$ are relatively prime. We have that
\begin{align*}
R2(\zeta_3^a,q)
&=
	\sum_{n\ge0}	
	\left(N2(0,3,n)-N2(1,3,n)\right)q^n,
\\
R2(\zeta_4^a,q)
&=
	\sum_{n\ge0}
	\left(N2(0,4,n)-N2(2,4,n)\right)q^n,
\\
R2(\zeta_5^a,q)
&=
	\sum_{n\ge0}
	\left(N2(0,5,n)-N2(1,5,n)\right)q^n
	+\left(\zeta_5^{2a}+\zeta_5^{3a}\right)\left(N2(0,5,n)-N2(1,5,n)\right)q^n,
\\
R2(\zeta_6^a,q)
&=
	\sum_{n\ge0}
	\left(N2(0,6,n)+N2(1,6,n)-N2(2,6,n)-N2(3,6,n)\right)q^n,
\\
R2(\zeta_7^a,q)
&=
	\sum_{n\ge0}
	\left(N2(0,7,n)-N2(1,7,n)\right)q^n
	+\left(\zeta_7^{2a}+\zeta_7^{5a}\right)\left(N2(2,7,n)-N2(1,7,n)\right)q^n
	\\&\qquad\qquad
	+\left(\zeta_7^{3a}+\zeta_7^{4a}\right)\left(N2(3,7,n)-N2(1,7,n)\right)q^n,
\\
R2(\zeta_8^a,q)
&=
	\sum_{n\ge0}
	\left(N2(0,8,n)-N2(4,8,n)\right)q^n
	+\left(\zeta_8^{a}-\zeta_8^{3a}\right)\left(N2(1,8,n)-N2(3,8,n)\right)q^n,
\\
R2(\zeta_9^a,q)
&=
	\sum_{n\ge0}
	\left(N2(0,9,n)-N2(3,9,n)\right)q^n
	+\left(\zeta_9^{a}-\zeta_9^{2a}-\zeta_9^{5a}\right)\left(N2(1,9,n)-N2(2,9,n)\right)q^n
	\\&\qquad\qquad
	+\zeta_9^{4a}\left(N2(4,9,n)-N2(2,9,n)\right)q^n,
\\
R2(\zeta_{10}^a;q)
&=
	\sum_{n\ge0}\left(N2(0,10,n)+N2(1,10,n)-N2(4,10,n)-N2(5,10,n)\right)q^n
	\\&\qquad\qquad
	+(\zeta_{10}^{2a}-\zeta_{10}^{3a})
	\left(N2(1,10,n)+N2(2,10,n)-N2(3,10,n)-N2(4,10,n)\right)q^n
.
\end{align*}

Using Theorem \ref{TheoremAsymptoticForM2Rank}, we see that the 
asymptotic value of $A\left(\frac{a}{c};n\right)$ is obtained by taking
the term with largest $r$ in $\cosh(\pi r\sqrt{8n-1})$
that has a non-zero coefficient.
Some care must be taken in determining which values of $k$ give this leading
term. In particular, a single
value of $k$ may appear in multiple sums and different values of $k$
may give the same hyperbolic cosine. Furthermore, if the hyperbolic cosine
with largest $r$ appears multiple times, there may be cancellation.
In determining the values of $k$ that contribute to the leading term,
it is useful to note that
$r_{j,a,c,k}(m)\le \frac{1}{4}$ for $j=1,2,4,5$,
$r_{3,a,c,k}(m)\le \frac{9}{4}$, and
$r_{j,a,c,k}(m)\le \frac{1}{16}$ for $j=6,7$.
We find that the lead term asymptotics for various
$A\left(\frac{a}{c};n\right)$,
the coefficients of $R2(e^{\frac{2\pi i a}{c}},q)$, are as follows:
\begin{align*}
A\left(\frac{1}{3};n\right)
&\sim
   	\frac{\varepsilon_3(n) \cosh\left( \frac{\pi\sqrt{8n-1}}{12} \right)}{3\sqrt{8n-1}} 
,&
A\left(\frac15;n\right)
&\sim 
	\frac{\varepsilon_5(n) \sin\left(\frac{\pi}{5}\right) \cosh\left(\frac{\pi\sqrt{8n-1}}{20}\right)}
	{5\sqrt{8n-1}}
,\\
A\left(\frac{1}{6};n\right)
&\sim
	\frac{2\sqrt{2}\cosh\left(\frac{\pi}{12}\sqrt{8n-1}\right)}{\sqrt{8n-1}}
,&
A\left(\frac{1}{7};n\right)
&\sim
	\frac{4\sqrt{2}\sin\left(\frac{\pi}{7}\right)\cosh\left( \frac{3\pi\sqrt{8n-1}}{28} \right) }
	{\sqrt{8n-1}}
,\\
A\left(\frac{1}{8};n\right)
&\sim
	\frac{4\sqrt{2}\sin\left(\frac{\pi}{8}\right)\cosh\left( \frac{\pi\sqrt{8n-1}}{8} \right) }
	{\sqrt{8n-1}}
,&
A\left(\frac{1}{9};n\right)
&\sim
	\frac{4\sqrt{2}\sin\left(\frac{\pi}{9}\right)\cosh\left( \frac{5\pi\sqrt{8n-1}}{36} \right) }
	{\sqrt{8n-1}}
,\\
A\left(\frac{1}{10};n\right)
&\sim
	\frac{\sqrt{2}(\sqrt{5}-1) \cosh\left( \frac{3\pi\sqrt{8n-1}}{20} \right)}
	{\sqrt{8n-1}} 
,
\end{align*}
where
\begin{align*}
\varepsilon_3(n) &:=
\begin{cases}
    -4\sqrt{6} 	&	\mbox{if } n \equiv 0 \pmod{3},
	\\
    2\sqrt{6} 	&    \mbox{if } n\equiv1,2\pmod{3},
\end{cases}
&
\varepsilon_5(n) &:=
\begin{cases}
	6\sqrt{2}\left(5-\sqrt{5}\right)
	&
	\mbox{ if } n\equiv 0 \pmod{5}
	,\\
	2\sqrt{2}\left(5+\sqrt{5}\right)
	&
	\mbox{ if } n\equiv 1 \pmod{5}
	,\\
	20\sqrt{2}
	&
	\mbox{ if } n\equiv 2 \pmod{5}
	,\\
	2\sqrt{2}\left(5-\sqrt{5}\right)
	&
	\mbox{ if } n\equiv 3 \pmod{5}
	,\\
	6\sqrt{2}\left(5+\sqrt{5}\right)
	&
	\mbox{ if } n\equiv 4 \pmod{5}
	.
\end{cases}
\end{align*}

Although we can determine formulas for other values of $\frac{a}{c}$, we have
restricted ourselves to displaying those values with simple formulas for $c\leq10$.

\begin{proof}[Proof of \eqref{EqMaoConjecture1}]
We are to show that $A(\frac{1}{6};n)>0$ for $n\ge0$. From the asymptotic
for $A\left(\frac{1}{6};q\right)$, the inequality holds for $n$ sufficiently 
large. We find that the next exponential term arising in the exact formula for
$A(\frac{1}{6};q)$ is 
$\frac{\cosh\left(\frac{\pi}{24}\sqrt{8n-1}\right)}{\sqrt{8n-1}}$,
from which we deduce that 
\begin{align*}
A\left(\frac{1}{6};q\right)
&= \frac{2\sqrt{2}\cosh\left(\frac{\pi\sqrt{8n-1}}{12}\right)}{\sqrt{8n-1}}+E,
\end{align*}
where $E$ is bounded as
\begin{align*}
|E|&\le
\frac{112 \left(\sqrt{n}+1\right)^\frac{3}{2} \cosh\left(\frac{\pi}{24}\sqrt{8n-1}\right)}	
	{3\sqrt{8n-1}}
+10^{16}\sqrt{n}.
\end{align*}
This shows the inequality holds for $n\ge 3823$, and with the assistance of MAPLE
we find that the inequality also holds for the initial values of $n$.
\end{proof}

\begin{proof}[Proof of \eqref{EqMaoConjecture3}]
Since
\begin{align*}
\frac{R2(\zeta_{10};q)-R2(\zeta_{10}^3;q)}
{\left(4\cos\left(\frac{2\pi}{5}\right)+1\right)}
&=
	\sum_{n=0}^\infty\left(N2(1,10,n)+N2(2,10,n)-N2(3,10,n)-N2(4,10,n)\right)q^n,
\end{align*}
we are to prove that $A(\frac{1}{10};n)>A(\frac{3}{10};n)$ for $n\ge3$. We find the
the second largest exponential term in the expansion for $A\left(\frac{1}{10};n\right)$
is $\frac{\cosh\left(\frac{\pi}{40}\sqrt{8n-1}\right)}{\sqrt{8n-1}}$.
For $A\left(\frac{3}{10};n\right)$ we find that the largest exponential term
is at most $\frac{\cosh\left(\frac{\pi}{20}\sqrt{8n-1}\right)}{\sqrt{8n-1}}$
(in fact, it is exponentially smaller for $n\equiv 2\mod{3}$). Therefore,
\begin{align*}
R2(\zeta_{10};q)-R2(\zeta_{10}^3;q)
&=
\frac{\sqrt{2}(\sqrt{5}-1)\cosh\left(\frac{3\pi\sqrt{8n-1}}{20}\right)}{\sqrt{8n-1}}
+E,
\end{align*}
where $E$ is bounded as
\begin{align*}
|E|&\le
\frac{112 \left(\sqrt{n}+1\right)^\frac{3}{2}\left( 
	\cosh\left(\frac{\pi}{20}\sqrt{8n-1}\right)+\cosh\left(\frac{\pi}{40}\sqrt{8n-1}\right)   
\right)}{3\sqrt{8n-1}}
+2.4\cdot10^{16}\sqrt{n}.
\end{align*}
This shows the inequality holds for $n\ge 1190$, and with the assistance of MAPLE
we find that the inequality also holds for the initial values of $n$.
\end{proof}

We now give a few new inequalities. As the proof method is the same as above,
we summarize the results in Table \ref{TableInequalities}. The columns are
arranged to state the inequality for a combination of
$N2(r,c,n)$, the value of the lower bound on $n$ that is required for the inequality,
the equivalent inequality between certain $A\left(\frac{a}{c};n\right)$,
the asymptotic value of the previous column, and the number of
initial terms we must check with a computer.
For clarity, we include the two inequalities proved above.

\begin{longtable}[h]{ccc@{$\!$}c@{$\!\!$}c}
\caption{Some Inequalities}
\label{TableInequalities}
\\[-2ex]
	\toprule
	\mbox{Inequality} & $n$ 	& \mbox{Equivalent Inequality} & \mbox{Asymptotic} 
	& $\begin{array}{c}\mbox{Initial}\\ \mbox{Terms}\end{array}$ 
	\\\midrule
	$N2(1,3,3n)>N2(0,3,3n) $
	&			
	$1$
	&
	$-A\left(\frac{1}{3};3n\right)>0$
	&
	$\frac{4\sqrt{6}\cosh\left(\frac{\pi\sqrt{24n-1}}{12}\right)}{3\sqrt{24n-1}}	$
	&
	1286	
	\\ \midrule
	$N2(2,4,8n)>N2(0,4,8n)$
	&			
	$6$
	&
	$-A\left(\frac{1}{4};8n\right)>0$
	&	
	$\frac{4\sin\left(\frac{\pi}{16}\right) \cosh\left(\frac{\pi\sqrt{64n-1}}{16}\right)}{\sqrt{64n-1}}	$
	&
 	934		
	\\ \midrule
	$N2(0,4,8n+1)>N2(2,4,8n+1)$
	&			
	$0$
	&
	$A\left(\frac{1}{4};8n+1\right)>0$
	&	
	$\frac{4\sin\left(\frac{5\pi}{16}\right) \cosh\left(\frac{\pi\sqrt{8(8n+1)-1}}{16}\right)}{\sqrt{8(n+1)-1}}	$
	&
 	876		
	\\ \midrule
	$N2(0,4,8n+2)>N2(2,4,8n+2)$
	&			
	$0$
	&
	$A\left(\frac{1}{4};8n+2\right)>0$
	&	
	$\frac{4\sin\left(\frac{7\pi}{16}\right) \cosh\left(\frac{\pi\sqrt{8(8n+2)-1}}{16}\right)}{\sqrt{8(n+2)-1}}	$
	&
 	870		
	\\ \midrule
	$N2(2,4,8n+3)>N2(0,4,8n+3)$
	&			
	$3$
	&
	$-A\left(\frac{1}{4};8n+3\right)>0$
	&	
	$\frac{4\sin\left(\frac{3\pi}{16}\right) \cosh\left(\frac{\pi\sqrt{8(8n+3)-1}}{16}\right)}{\sqrt{8(n+3)-1}}	$
	&
 	892		
	\\ \midrule
	$N2(0,4,8n+4)>N2(2,4,8n+4)$
	&			
	$8$
	&
	$A\left(\frac{1}{4};8n+4\right)>0$
	&	
	$\frac{4\sin\left(\frac{\pi}{16}\right) \cosh\left(\frac{\pi\sqrt{8(8n+4)-1}}{16}\right)}{\sqrt{8(n+4)-1}}	$
	&
 	934		
	\\ \midrule
	$N2(2,4,8n+5)>N2(0,4,8n+5)$
	&			
	$1$
	&
	$-A\left(\frac{1}{4};8n+5\right)>0$
	&
	$\frac{4\sin\left(\frac{5\pi}{16}\right) \cosh\left(\frac{\pi\sqrt{8(8n+5)-1}}{16}\right)}{\sqrt{8(n+5)-1}}	$
	&
 	876		
	\\ \midrule
	$N2(2,4,8n+6)>N2(0,4,8n+6)$
	&			
	$0$
	&
	$-A\left(\frac{1}{4};8n+6\right)>0$
	&	
	$\frac{4\sin\left(\frac{7\pi}{16}\right) \cosh\left(\frac{\pi\sqrt{8(8n+6)-1}}{16}\right)}{\sqrt{8(n+6)-1}}	$
	&
 	869		
	\\ \midrule
	$N2(0,4,8n+7)>N2(2,4,8n+7)$
	&			
	$0$
	&
	$A\left(\frac{1}{4};8n+7\right)>0$
	&	
	$\frac{4\sin\left(\frac{3\pi}{16}\right) \cosh\left(\frac{\pi\sqrt{8(8n+7)-1}}{16}\right)}{\sqrt{8(n+7)-1}}	$
	&
 	892		
	\\\midrule
	$\begin{array}{l}N2(0,6,n)+N2(1,6,n)\\>N2(2,6,n)+N2(3,6,n)\end{array}$
	&
	$0$
	&
	$A(\frac{1}{6};n)>0$
	&
	$\frac{2\sqrt{2}\cosh\left(\frac{\pi\sqrt{8n-1}}{12}\right)}{\sqrt{8n-1}}$
	&3823
	\\ \midrule
	$\begin{array}{l}N2(1,10,n)+N2(2,10,n)\\>N2(3,10,n)+N2(4,10,n)\end{array}$
	&
	$3	$
	&
	$A(\frac{1}{10};n)-A(\frac{3}{10};n)>0$
	&
	$\frac{\sqrt{2}(\sqrt{5}-1)\cosh\left(\frac{3\pi\sqrt{8n-1}}{20}\right)}{\sqrt{8n-1}}$
	&1190
	\\ \midrule
	$\begin{array}{l}N2(0,10,n)+N2(3,10,n)\\>N2(2,10,n)+N2(5,10,n)\end{array}$
	&
	$7	$
	&
	$\begin{array}{l}\left(1-\cos\left(\frac{2\pi}{5}\right)\right)A(\frac{1}{10};n)\\-A(\frac{3}{10};n)>0\end{array}$
	&
	$\begin{array}{l}\scriptstyle\left(1-\cos\left(\frac{2\pi}{5}\right)\right)\sqrt{2}(\sqrt{5}-1)\\ 
		\times\frac{\cosh\left(\frac{3\pi\sqrt{8n-1}}{20}\right)}{\sqrt{8n-1}}
	\end{array}$
	&1233
\\
\bottomrule
\end{longtable}

With the results in Table \ref{TableInequalities}, it is clear that
we should expect many more inequalities. Below we list additional 
inequalities, but omit much of the information of Table \ref{TableInequalities}.
However, we do note that among these inequalities,
the strictest requirement on $n$ is $n\ge 36$
and the largest number of initial terms we must verify is $2838$.
By \cite{LovejoyOsburn1} we know that $N2(1,5,5n+1)=N2(2,5,5n+1)$
and $N2(0,5,5n+3)=N2(2,5,5n+3)$, and so we omit the inequalities duplicated by
this fact. With the inequalities of Mao \cite{Mao1} and
Table \ref{TableInequalities}, these account for all of the inequalities
that follow immediately from our asymptotics for $3\leq c\leq10$,
\begin{gather*}
N2(1,5,5n)<N2(2,5,5n)<N2(0,5,5n),\\
N2(1,5,5n+1)<N2(0,5,5n+1),\\
N2(2,5,5n+2)<N2(1,5,5n+2)<N2(0,5,5n+2),\\
N2(0,5,5n+3)<N2(1,5,5n+3),\\
N2(2,5,5n+4)<N2(0,5,5n+4)<N2(1,5,5n+4),\\
N2(3,7,n)<N2(2,7,n)<N2(1,7,n)<N2(0,7,n),\\
N2(4,8,n)<N2(0,8,n),\\
N2(3,8,n)<N2(1,8,n),\\
N2(3,9,n)<N2(0,9,n),\\
N2(4,9,n)<N2(2,9,n)<N2(1,9,n).
\end{gather*}
We leave it to the interested reader to derive additional 
inequalities.

\section{Remarks}

We have given asymptotics for the moments of the generating function of
the $M_2$-rank of partitions without repeated odd parts by following the methods
established in \cite{BringmannMahlburgRhoades1}. We note these
methods were used by Mao in \cite{Mao3} to determine asymptotics
for the moments of the generating functions of both the rank of overpartitions 
and the $M_2$-rank of overpartitions. Furthermore, these techniques were used in
\cite{Waldherr1} by Waldherr to determine asymptotics for the moments of 
the generating functions of Garvan's $k$-rank of partitions. As such, it is 
clear that the techniques of \cite{BringmannMahlburgRhoades1} should be 
considered widely applicable.

Motivated by the asymptotics from \cite{Bringmann1} for the coefficients of
$R(e^{\frac{2\pi ia}{c}};q)$, along with the conjectured inequalities from \cite{Mao1},
we gave asymptotics for the coefficients of 
$R2(e^{\frac{2\pi ia}{c}};q)$. However, given the representation used for
for $R2(\zeta;q)$ in terms of $\mu(u,v;\tau)$, we do not have asymptotics
for the coefficients of $R2(-1;q)$. These asymptotics can be obtained by
similar techniques, but one must actually carry out the proofs and 
calculations. In particular, $R2(-1;q)$ is $\mu(-q)$, where $\mu(q)$ a second 
order mock theta function \cite{McIntosh1}. Also, it is worth noting
that Mao \cite{Mao2} has given asymptotics for $N2(m,n)$.

We have proved a number of inequalities among the $N2(r,m,n)$, and have done
so asymptotically. There is a question of which of these inequalities can
also be proved by $q$-series techniques. It is desirable to have both proofs.

\bibliographystyle{abbrv}
\bibliography{AsymptoticsForM2RankRef}

\end{document}